\documentclass[reqno]{amsart}

\usepackage[margin= 3cm]{geometry}
\usepackage{enumerate}
\usepackage{tikz}
\usepackage{tikz-cd}
\usepackage{amsthm}
\usepackage{amsmath}
\usepackage{amssymb}
\usepackage{mathtools}
\usepackage{hyperref}
\usepackage[T1]{fontenc}
\usepackage[utf8]{inputenc}

\theoremstyle{plain}
\newtheorem{theo}{Theorem}[section]
\newtheorem{cor}[theo]{Corollary}
\newtheorem{lem}[theo]{Lemma}
\newtheorem{prop}[theo]{Proposition}

\theoremstyle{definition}
\newtheorem{cons}[theo]{Construction}

\newtheorem{defi}[theo]{Definition}
\newtheorem{ex}[theo]{Example}
\newtheorem{notation}[theo]{Notation}

\newtheorem{rem}[theo]{Remark}
\newtheorem{introthm}{Theorem}

\newcommand{\Aa}{\mathcal{A}}
\newcommand{\Bb}{\mathcal{B}}
\newcommand{\Cc}{\mathcal{C}}
\newcommand{\Dd}{\mathcal{D}}
\newcommand{\Ee}{\mathcal{E}}
\newcommand{\Ff}{\mathcal{F}}

\newcommand{\Ii}{\mathcal{I}}

\newcommand{\Ll}{\mathcal{L}}

\newcommand{\Nn}{\mathcal{N}}
\newcommand{\Oo}{\mathcal{O}}
\newcommand{\Pp}{\mathcal{P}}

\newcommand{\Ss}{\mathcal{S}}
\newcommand{\Tt}{\mathcal{T}}
\newcommand{\Uu}{\mathcal{U}}
\newcommand{\Vv}{\mathcal{V}}
\newcommand{\Ww}{\mathcal{W}}
\newcommand{\Xx}{\mathcal{X}}
\newcommand{\Yy}{\mathcal{Y}}
\newcommand{\Zz}{\mathcal{Z}}

\let\lim\relax
\DeclareMathOperator{\lim}{lim}
\DeclareMathOperator{\colim}{colim}
\DeclareMathOperator{\const}{const}
\DeclareMathOperator{\fgt}{fgt}
\DeclareMathOperator{\ev}{ev}
\DeclareMathOperator{\Lan}{Lan}
\DeclareMathOperator{\Ran}{Ran}
\DeclareMathOperator{\OpLAdj}{\Lambda_{G/G}}
\DeclareMathOperator{\incl}{incl}
\DeclareMathOperator{\AlgE}{Alg_{E_{\infty}}}
\DeclareMathOperator{\CAlg}{CAlg}
\DeclareMathOperator{\NAlg}{NAlg}
\DeclareMathOperator{\Hom}{Hom}
\DeclareMathOperator{\HOM}{HOM}
\DeclareMathOperator{\Fun}{Fun}
\DeclareMathOperator{\Span}{Span}
\DeclareMathOperator{\Triv}{Triv}
\DeclareMathOperator{\Env}{Env}
\DeclareMathOperator{\Ar}{Ar}
\DeclareMathOperator{\Free}{Free}
\DeclareMathOperator{\FCopC}{\mathbb{F}}
\DeclareMathOperator{\Cofree}{\Fun(\underline{\Tt}^{\op},}
\DeclareMathOperator{\CofreeG}{\Fun(\underline{\Oo}_{\mathit{G}}^{\op},}
\DeclareMathOperator{\CorCof}{\Fun(\underline{\Oo}_{\mathit{G}}^{\simeq},}
\DeclareMathOperator{\CMon}{CMon}
\DeclareMathOperator{\Mack}{Mack}
\DeclareMathOperator{\PMack}{\underline{\Mack}}
\DeclareMathOperator{\Stab}{Stab}
\DeclareMathOperator{\pr}{pr}
\DeclareMathOperator{\id}{id}
\DeclareMathOperator{\Nm}{Nm}
\DeclareMathOperator{\Uncc}{Un^{\text{cc}}}
\DeclareMathOperator{\Unct}{Un^{\text{ct}}}

\DeclareMathOperator{\FunProd}{\Fun^{\! \times}\!}
\DeclareMathOperator{\FunTen}{\Fun^{\! \otimes}\!}
\DeclareMathOperator{\FunTenG}{\Fun_{G}^{\otimes}}
\DeclareMathOperator{\FunLax}{\Fun^{\! \textup{lax}}\!}
\DeclareMathOperator{\FunLaxG}{\Fun_{G}^{\textup{lax}}}
\DeclareMathOperator{\CorCofMul}{\FunTen(\underline{\FCopC}_{\mathit{G}}^{\simeq,\amalg},}
\DeclareMathOperator{\Ho}{Ho}

\newcommand{\op}{\textup{op}}
\newcommand{\RatQ}{\mathbb{Q}}
\newcommand{\IntZ}{\mathbb{Z}}
\newcommand{\OG}{\Oo_G}
\newcommand{\OGo}{\Oo_G^{\op}}
\newcommand{\FinG}{\FCopC_G}
\newcommand{\Spc}{\textup{Spc}}
\newcommand{\Ab}{\textup{Ab}}
\newcommand{\Cat}{\textup{Cat}}

\newcommand{\PSMCat}{\widetilde{\Cat}\vphantom{\Cat}^{\otimes}}
\newcommand{\SMCat}{\Cat^{\otimes}}
\newcommand{\PSMCatG}{\widetilde{\Cat}\vphantom{\Cat}^{\otimes}_{G}}
\newcommand{\SMCatG}{\Cat^{\otimes}_{G}}
\newcommand{\POpG}{\widetilde{\textup{Op}}_{G}}
\newcommand{\OpG}{\textup{Op}_{G}}
\newcommand{\POp}{\widetilde{\textup{Op}}}
\newcommand{\Op}{\textup{Op}}
\newcommand{\RelCat}{\textup{RelCat}}
\newcommand{\AdTrip}{\textup{AdTrip}}
\newcommand{\Pcat}{\underline{\textup{cat}}}
\newcommand{\PSpc}{\underline{\Spc}}
\newcommand{\PrL}{\mathrm{Pr^{L}}}
\newcommand{\PrLT}{\mathrm{Pr^{L}_{\Tt}}}
\newcommand{\PrLG}{\mathrm{Pr}^{\textup{L}}_{G}}
\newcommand{\BC}{\textup{BC}}
\newcommand{\Sp}{\textup{Sp}}
\newcommand{\PSp}{\underline{\Sp}}
\newcommand{\Sphere}{\mathbb{S}}
\newcommand{\inv}{\textit{-}1}

\title{A MODEL FOR NORMED ALGEBRAS IN RATIONAL $G$-SPECTRA}
\author{Giorgi Tigilauri}

\begin{document}
\begin{abstract}
    For a finite group $G$, we construct a simplified model for the $G$-symmetric monoidal $G$-$\infty$-category $\PSp_{G,\RatQ}^{\otimes}$ of rational $G$-spectra. Using this model, we classify $\Ii$-normed algebras in rational $G$-spectra for a given indexing system $\Ii$. We show that such an algebra is equivalently described as a collection $\{\Xx(G/H)\}_{(H\leq G)}$ of commutative algebras in nonequivariant rational spectra, indexed by conjugacy classes of subgroups of $G$, together with compatible morphisms of commutative algebras $\Xx(G/K)\xrightarrow{}\Xx(G/H)$ whenever $K\leq H$ and the induced map $G/K\xrightarrow{}G/H$ is in $\Ii$. This generalizes a result by Wimmer \cite{Wimmer}.
    \end{abstract}

\setcounter{tocdepth}{2}
\date{}

\maketitle
\tableofcontents

\section{Introduction}

Cohomology theories are among the central tools of algebraic topology. They provide means to approach problems that are intractable when considered directly. The category of spectra offers a natural framework in which cohomology theories can be studied via homotopy-theoretic methods. Despite their utility, spectra possess an intricate structure that is often difficult to analyze. This complexity simplifies significantly upon passing to the rational setting. Serre’s computation of the rational homotopy groups of spheres \cite{Serre}, imply that the rational stable homotopy category is equivalent to the category of graded vector spaces over the rational numbers
\begin{equation*}
    \Ho(\Sp_{\RatQ})\simeq \RatQ\textup{-GrMod}.
\end{equation*}
Thus, in the nonequivariant setting, rational homotopy theory becomes entirely algebraic.

Equivariant cohomology theories arise naturally when studying spaces equipped with a group action. Foundational work of Lewis, May, and Steinberger \cite{lewisMaySteinberger}, together with Mandell and May \cite{MandellMay}, establishes that equivariant cohomology theories are represented by $G$-spectra, giving rise to the equivariant stable homotopy category as the appropriate setting for their study. Once again, the category of $G$-spectra simplifies significantly upon rationalization. It was conjectured by Greenlees \cite{GreenleesConjecture}, that for each compact Lie group $G$ there exists an abelian category $\Aa(G)$, such that the homotopy category of rational $G$-spectra is equivalent to the derived category $\Dd(\Aa(G))$. In the case of a finite group $G$, it follows from the work of Greenlees and May \cite[Appendix A]{GreenleesMay} that the conjecture holds for the abelian category 
\begin{equation*}
    \Aa(G)=\prod_{(H\leq G)}\RatQ[\textup{W}_{G}H]\textup{-Mod},
\end{equation*}
which is a product over the conjugacy classes of subgroups of $G$ of rational vector spaces with an action of the Weyl group $\textup{W}_{G}H=\textup{N}_{G}(H)/H$. The conjecture was resolved for various groups and classes of groups. For example, the cases $G=\textup{SO}(2),\textup{O}(2),\textup{\textup{SO}(3)}$ were treated in \cite{GreenleesConj1}, \cite{GreenleesConj2}, and \cite{GreenleesConj3}, respectively. 

Multiplicative structures on the cohomology theories give a useful refinement of the theory, encoding additional information beyond the underlying additive invariants. In the context of spectra, such structures are modeled by highly structured ring objects such as $\textup{E}_{\infty}$-ring spectra, which encode coherently associative and commutative multiplications. The framework of $\infty$-categories allows one to systematically encode the higher coherences required for $\textup{E}_{\infty}$-algebras. In this setting, one is naturally led again to the study of rational $\textup{E}_{\infty}$-ring spectra. Following Shipley \cite{ShipleyCDGA}, one obtains an equivalence
\begin{equation*}
    \AlgE(\Sp_{\RatQ})\simeq \textup{CDGA}_{\RatQ},
\end{equation*}
between the $\infty$-category $\AlgE(\Sp_{\RatQ})$ of rational $\textup{E}_{\infty}$-ring spectra and the underlying $\infty$-category of rational commutative differential graded algebras. 

In the equivariant setting, the situation changes substantially, as one obtains nonequivalent notions of commutative algebras. For a finite group $G$, one can consider the $\infty$-category $\Sp_{G}$ of $G$-spectra and take $\textup{E}_{\infty}$-algebras inside to obtain the $\infty$-category $\AlgE(\Sp_{G})$. On the other hand, one can consider the category $\CMon(\Sp_{G}^{O})$ of strictly commutative monoids in the category $\Sp_{G}^{O}$ of orthogonal $G$-spectra and equip it with the lifted positive stable model category structure, see \cite{MandellMay}. Surprisingly, the $\infty$-category $\AlgE(\Sp_{G})$ is not equivalent to the underlying $\infty$-category of $\CMon(\Sp_{G}^{O})$. The reason for this is that strictly commutative monoids in orthogonal $G$-spectra encode richer structure. Namely, for a strictly commutative monoid object $X$ in $\Sp_{G}^{O}$, in addition to the regular multiplication map $X\otimes X\xrightarrow{}X$ one obtains norm maps 
\begin{equation*}
    n_{K}^{H}\colon\textup{N}_{K}^{H}X\xrightarrow{}X,
\end{equation*}
for subgroups $K\leq H\leq G$. Here
\begin{equation*}
    \textup{N}_{K}^{H}\colon \Sp_{K}\xrightarrow{}\Sp_{H}
\end{equation*}
denotes the HHR-norm, see Hill, Hopkins, and Ravenel \cite{HHR}. Moreover, in the work of Blumberg and Hill \cite{BlumbergHill}, the authors introduce the notion of $\textup{N}_{\infty}$-operads for a finite group $G$. Algebras over such an operad in the category of orthogonal $G$-spectra encode a certain level of commutativity, with norm maps $n_{K}^{H}$ governed by the operad. The collection of allowed norm maps is specified by an associated indexing system.

The $\infty$-category of rational $G$-ring spectra with the highest level of commutativity was studied by Wimmer in \cite{Wimmer}, where the author first shows that the geometric fixed points assemble into a symmetric monoidal equivalence 
\begin{equation*}
    \Sp_{G,\RatQ}\simeq\Fun(\OG^{\simeq},\Sp_{\RatQ})
\end{equation*}
of $\infty$-categories. This equivalence immediately implies the statement about the rational $G$-ring spectra with the lowest level of commutativity. Wimmer then shows that the geometric fixed points also assemble into an equivalence
\begin{equation}\tag{i}\label{EqIntroduction1}
    \textup{N}(\CMon(\Sp_{G}^{O}))[\text{Stable weak equivalences}^{\inv}]_{\RatQ}\simeq \Fun(\OG,\AlgE(\Sp_{\RatQ}))
\end{equation}
of $\infty$-categories, where on the left we have the rationalization of the underlying $\infty$-category of strictly commutative monoids in orthogonal $G$-spectra. These two results cover the lowest and the highest levels of commutativity. For intermediate levels of commutativity, the situation is more complicated, largely due to the difficulty of handling $\textup{N}_{\infty}$-operads and algebras over them. The notion of normed algebras makes things more tractable.

\subsection{Parametrized homotopy theory}

Just as $\infty$-categories provide a natural framework for studying homotopy theory, parametrized $\infty$-categories offer a natural setting to study equivariant homotopy theory. The development of the theory of parametrized $\infty$-categories was started by Barwick, Dotto, Glasman, Nardin, and Shah in \cite{Expose}, with further development by the works of Martini and Wolf \cite{martini2021colimits}, and Cnossen, Lenz, and Linskens \cite{ParStab}. The notion important for the current paper is that of $G$-$\infty$-categories for a finite group $G$, which is a special case of the general parametrized $\infty$-categories. 

A $G$-$\infty$-category is defined to be a contravariant functor
\begin{equation*}
    \underline{\Cc}\colon\OGo\xrightarrow{}\Cat_{\infty},
\end{equation*}
from the orbit category of $G$ to the $\infty$-category of $\infty$-categories. Many notions of category theory generalize to the theory of $G$-$\infty$-categories. One can define notions of $G$-limits, $G$-adjunctions, $G$-presentability, $G$-stability and so on. The main example in this paper will be that of the $G$-$\infty$-category of genuine $G$-spectra, which is defined to be a functor
\begin{equation*}
    \PSp_{G}\colon\OGo\xrightarrow{}\Cat_{\infty}, \quad G/H\mapsto \Sp_{H},
\end{equation*}
sending a coset $G/H$ to the $\infty$-category of genuine $H$-spectra, with the functoriality given by restriction functors. 

In this setting we can talk about the $G$-symmetric monoidal $G$-$\infty$-categories, which are defined to be finite product preserving functors
\begin{equation*}
    \underline{\Cc}^{\otimes}\colon\Span(\FinG)\xrightarrow{} \Cat_{\infty}.
\end{equation*}
The work of Bachmann and Hoyois \cite{NormsMotivic}, upgrades the $G$-$\infty$-category of genuine $G$-spectra to a $G$-symmetric monoidal $G$-$\infty$-category
\begin{equation*}
    \PSp_{G}^{\otimes}\colon\Span(\FinG)\xrightarrow{}\Cat_{\infty}.
\end{equation*}
For a chain of subgroups $K\leq H\leq G$, the functoriality of $\PSp_{G}^{\otimes}$ on the forward morphisms $\nabla\colon G/H\amalg G/H\xrightarrow{}G/H$ and $G/K\xrightarrow{}G/H$ lets us recover the smash product $-\otimes-\colon \Sp_{H}\times\Sp_{H}\xrightarrow{}\Sp_{H}$ and the HHR-norm $\textup{N}_{K}^{H}\colon \Sp_{K}\xrightarrow{}\Sp_{H}$, respectively.

Since, the $G$-symmetric monoidal $G$-category $\PSp_{G}^{\otimes}$ naturally encodes HHR-norms, one can consider algebras encoding various levels of equivariant commutativity. To do this one can work with the $\infty$-category $\OpG$ of $G$-$\infty$-operads, which can be defined as a certain non-full subcategory of the slice $\infty$-category $(\Cat_{\infty})_{/\Span(\FinG)}$, see the work of Lenz, Linskens, and P\"{u}tzst\"{u}ck \cite{NormsEq}. The lax functors from the terminal $G$-$\infty$-operad $\id\colon\Span(\FinG)\xrightarrow{}\Span(\FinG)$ to the underlying $G$-$\infty$-operad of $\PSp_{G}^{\otimes}$ define the $\infty$-category 
\begin{equation*}
    \NAlg_{\FinG}(\PSp_{G}^{\otimes})
\end{equation*}
of $\FinG$-normed algebras in the $G$-symmetric monoidal $G$-$\infty$-category of $G$-spectra. It follows from \cite{NormsEq}, that the $\infty$-category $\NAlg_{\FinG}(\PSp_{G}^{\otimes})$ of normed algebras models the $G$-ring spectra with the highest level of commutativity. On the other end, the lax functors from the $G$-$\infty$-operad $\Span_{\text{all},\FCopC[\OG^{\simeq}]}(\FinG)\xrightarrow{}\Span(\FinG)$ to the underlying $G$-$\infty$-operad of $\PSp_{G}^{\otimes}$ define the $\infty$-category 
\begin{equation*}
    \NAlg_{\FCopC[\OG^{\simeq}]}(\PSp_{G}^{\otimes}),
\end{equation*}
of $\FCopC[\OG^{\simeq}]$-normed algebras in the $G$-symmetric monoidal $G$-$\infty$-category of $G$-spectra. In the above, $\Span_{\text{all},\FCopC[\OG^{\simeq}]}(\FinG)$ denotes the wide subcategory of the span category $\Span(\FinG)$ where the forward morphisms are finite coproducts of codiagonals. The $\infty$-category $\NAlg_{\FCopC[\OG^{\simeq}]}(\PSp_{G}^{\otimes})$ models $G$-ring spectra with the lowest level of commutativity. It is then natural to consider $G$-$\infty$-suboperads of the terminal $G$-$\infty$-operad $\id\colon\Span(\FinG)\xrightarrow{}\Span(\FinG)$, which contain the $G$-$\infty$-operad $\Span_{\text{all},\FCopC[\OG^{\simeq}]}(\FinG)\xrightarrow{}\Span(\FinG)$. It turns out that such $G$-$\infty$-operads are in bijection with indexing systems in the sense of \cite{BlumbergHill}. Namely, for an indexing system $\Ii$ we can define a wide subcategory $\Span_{\text{all},\Ii}(\FinG)$ of $\Span(\FinG)$, with forward morphisms restricted according to $\Ii$. In a similar way as before, we can then define the $\infty$-category
\begin{equation*}
    \NAlg_{\Ii}(\PSp_{G}^{\otimes})
\end{equation*}
of $\Ii$-normed algebras. It is expected that the $\infty$-categories $\NAlg_{\Ii}(\PSp_{G}^{\otimes})$ are equivalent to the underlying $\infty$-categories of algebras over the associated $\textup{N}_{\infty}$-operads in orthogonal $G$-spectra, but currently this is not known.

\subsection{Main results}
In this paper, we complete the classification of rational $G$-ring spectra for finite groups 
$G$ by treating the intermediate levels of equivariant commutativity. Inspired by the equivalence (\ref{EqIntroduction1}), we define for an indexing system $\Ii$ a subcategory $\Oo_{G,\Ii}$ of the orbit category, where the morphisms correspond to the norm maps present. The main result of this paper states:
\begin{introthm}[Theorem \ref{MainTh}]\label{IntroTheoA}
    Let $G$ be a finite group, and let $\Ii$ be an indexing system. There exists an equivalence
    \begin{equation*}
        \NAlg_{\Ii}(\PSp_{G,\RatQ}^{\otimes})\simeq \Fun(\Oo_{G,\Ii},\AlgE(\Sp_{\RatQ}))
    \end{equation*}
    of $\infty$-categories.
\end{introthm}

The equivalence is given by associating to each $\Ii$-normed algebra the diagram consisting of the geometric fixed points of the underlying $G$-spectrum together with norm maps between them. In particular, for $\Ii$ the maximal indexing system, this result recovers the equivalence (\ref{EqIntroduction1}) from \cite{Wimmer}.

To prove the above theorem we first show that the $G$-symmetric monoidal $G$-$\infty$-category $\PSp_{G,\RatQ}^{\otimes}$ admits a simplified description. We constuct a $G$-symmetric monoidal $G$-$\infty$-category
\begin{equation*}
    \CorCofMul(\Sp_{\RatQ},\otimes))\colon\Span(\FinG)\xrightarrow{}\Cat_{\infty},
\end{equation*}
which sends a coset $G/H$ to the $\infty$-category $\Fun(\Oo_{H}^{\simeq},\Sp_{\RatQ})$ of functors from the underlying groupoid of $\Oo_{H}$ to the $\infty$-category of rational spectra. The functoriality of $\CorCofMul(\Sp_{\RatQ},\otimes))$ on forward morphisms $G/H\amalg G/H\xrightarrow{}G/H$ recovers the pointwise symmetric monoidal structure induced by the smash product of rational spectra. We then proceed to construct a $G$-symmetric monoidal $G$-functor
\begin{equation*}
    \phi^{\otimes}_{\RatQ}\colon \PSp_{G,\RatQ}^{\otimes}\xrightarrow{}\CorCofMul(\Sp_{\RatQ},\otimes)),
\end{equation*}
which sends a rational $H$-spectrum $X$ to the functor 
\begin{equation*}
    \Oo_{H}^{\simeq}\xrightarrow{}\Sp_{\RatQ}, \quad H/K\mapsto \Phi^{K}(X),
\end{equation*}
given by the geometric fixed points. At the end of Section \ref{Section4} we show the following result.
\begin{introthm}[Theorem \ref{SMEquivalence}]\label{IntroTheoB}
    The $G$-symmetric monoidal $G$-functor
    \begin{equation*}
        \phi^{\otimes}_{\RatQ}\colon \PSp_{G,\RatQ}^{\otimes}\xrightarrow{}\CorCofMul(\Sp_{\RatQ},\otimes))
    \end{equation*}
    is an equivalence of $G$-symmetric monoidal $G$-$\infty$-categories.
\end{introthm}

Theorem \ref{IntroTheoA} is then obtained using the simplified description of the $G$-symmetric monoidal $G$-$\infty$-category of rational $G$-spectra from Theorem \ref{IntroTheoB}.

\subsection{Outline}
In Section \ref{Section2}, we recall some of the necessary preliminaries. In Section \ref{Section3}, we begin by constructing a model for the $G$-$\infty$-category of rational $G$-spectra. This model is a part of a general construction associating to an $\infty$-category $\Cc$ a $G$-$\infty$-category $\CorCof\Cc)$. We establish some properties of this construction, and define a comparison functor $\phi_{\RatQ}\colon \PSp_{G,\RatQ}\xrightarrow{}\CorCof\Sp_{\RatQ})$. The comparison functor is a parametrized version of the comparison functor from \cite{StratCat}. At the end of this section, we show that $\phi_{\RatQ}$ is an equivalence of $G$-$\infty$-categories, by reducing to the nonparametrized setting of \cite[Theorem 4.2]{StratCat}. In Section \ref{Section4}, we start by associating to a symmetric monoidal $\infty$-category $(\Cc,\otimes)$ a $G$-symmetric monoidal $G$-$\infty$-category $\CorCofMul(\Cc,\otimes))$, which refines the construction of $\CorCof\Cc)$. We show that this assignment is a right adjoint in an adjunction between the $\infty$-categories of symmetric monoidal $\infty$-categories and $G$-symmetric monoidal $G$-$\infty$-categories. Combining this adjunction with the universal property of parametrized equivariant spectra, see \cite[Theorem 10.7]{Partial}, we lift $\phi_{\RatQ}$ to a $G$-symmetric monoidal comparison $G$-functor $\phi^{\otimes}_{\RatQ}\colon \PSp_{G,\RatQ}^{\otimes}\xrightarrow{}\CorCofMul(\Sp_{\RatQ},\otimes))$. The functor $\phi^{\otimes}_{\RatQ}$ is then automatically an equivalence, proving Theorem \ref{IntroTheoB}. In Section \ref{Section5}, we show that the normed algebras in the model of $\PSp_{G,\RatQ}^{\otimes}$ admit a simplified description, proving Theorem \ref{IntroTheoA}.

\subsection*{Acknowledgements}
I would like to  thank my PhD supervisor, Magdalena K\k{e}dziorek, for suggesting the project and for her support throughout. I would like to thank Julius Groenjes, Levan Iashvili, Gr\'egoire Marc, and Irakli Patchkoria for helpful conversations.
I would like to thank the Isaac Newton Institute for Mathematical Sciences, Cambridge,
for support and hospitality during the programme Equivariant Homotopy Theory in Context where
work on this paper was undertaken. This work was supported by EPSRC grant no EP/Z000580/1. The author was supported by NWO grant number OCENW.M.21.231.

\subsection{Notation and conventions}
Throughout the text, we fix a finite group $G$.

We fix a triple $\Uu\in \Vv \in \Ww$ of Grothendieck universes containing the set of natural numbers, and refer to their objects by small, large, and very large, respectively. We refer to $(\infty,1)$-categories simply as categories, and write $\Cat$ for the very large $(\infty,1)$-category of large $(\infty,1)$-categories, $\Spc$ for the large category of small groupoids, $\Spc_{G}$ for the category $\Fun(\OGo,\Spc)$, and $\Spc_{G,*}$ for $\Fun(\OGo,\Spc_{*})$. We refer to colimits indexed by small categories simply by colimits. In other cases, we will specify the size. We refer to large $\Uu$-presentable categories simply by presentable categories. We refer to $(\infty,2)$-categories simply as 2-categories.

Let $\Cc$ be a category. We will denote its finite coproduct completion by $\FCopC[\Cc]$, the category of finite sets $\FCopC[*]$ by $\FCopC$. We denote the orbit category of the finite group $G$ by $\OG$, and the category of finite $G$-sets by $\FinG \coloneqq \FCopC[\OG]$. We denote the underlying groupoid of $\Cc$ by $\Cc^{\simeq}$. Let $F\colon \Cc \xrightarrow{} \Dd$ be a functor, and let $d\in \Dd$. We denote the comma category $\Cc\times_{\Dd}\Dd_{/d}$ simply by $\Cc_{/d}$.

\section{Preliminaries}\label{Section2}

In this section, we review the main preliminaries. We will freely use the language of ($\infty$,1)-categories. Our main references are \cite{HTT} and \cite{HA}.

\subsection{Span categories}

Span categories will play a major role in the paper. We recall basic definitions and fix notation. Span $(\infty,1)$-categories were introduced in \cite{BarwickSpectral}, where they were called effective Burnside categories. We will use \cite{TwoVariable} as our main reference.

\begin{defi}{\cite[Definition 2.1]{TwoVariable}}
    An \textit{adequate triple} $(\Cc,\Cc_B,\Cc_F)$ consists of a category $\Cc$ together with two wide subcategories $\Cc_B$ and $\Cc_F$, such that for any morphism $b\colon X\xrightarrow{}Z$ in $\Cc_B$ and any morphism $f\colon Y\xrightarrow{}Z$ in $\Cc_F$, the pullback
    \[
    \begin{tikzcd}
        X\times_Z Y \arrow[r, "{b'}"] \arrow[d, "{f'}"] &Y \arrow[d, "{f}"]
        \\
        X \arrow[r, "{b}"] &Z
    \end{tikzcd}
    \]
    exists, and the projections $b'$ and $f'$ are morphisms in $\Cc_B$ and $\Cc_F$, respectively.
    
    A morphism of adequate triples $F\colon(\Cc,\Cc_B,\Cc_F) \xrightarrow{} (\Dd,\Dd_B,\Dd_F)$ is a functor $F\colon\Cc \xrightarrow{} \Dd$ such that $F(\Cc_B)\subseteq \Dd_B$, $F(\Cc_F)\subseteq \Dd_F$, and $F$ preserves pullbacks of morphisms in $\Cc_B$ along morphisms in $\Cc_F$. 
    
    We denote the category of adequate triples by $\AdTrip$.
\end{defi}

In \cite[Definition 2.12]{TwoVariable}, the authors define a functor
\begin{equation*}
    \Span \colon \AdTrip \xrightarrow{} \Cat,
\end{equation*}
which associates to each adequate triple its span category. Namely, the objects of $\Span (\Cc,\Cc_B,\Cc_F)$ are the same as objects of $\Cc$, and morphisms are given by spans

\[
\begin{tikzcd}[sep=tiny]
    &Z \arrow[dl, "{b}"'] \arrow[dr, "{f}"]
    \\
    X &&Y
\end{tikzcd}
\]
with $b \in \Cc_B$ and $f \in \Cc_F$. Composition is given by taking pullbacks.

Let $(\Cc,\Cc_B,\Cc_F)$ be an adequate triple. We denote the span category $\Span(\Cc,\Cc_B,\Cc_F)$ by $\Span_{B,F}(\Cc)$. If $\Cc$ admits pullbacks, we denote the span category $\Span_{\text{all},\text{all}}(\Cc)$ simply by $\Span(\Cc)$. It follows from \cite[Proposition 2.15]{TwoVariable} that for any wide subcategory $\Ee\subseteq\Cc$ the triples $(\Cc,\Cc^{\simeq},\Ee)$ and $(\Cc,\Ee,\Cc^{\simeq})$ are adequate. Moreover, there are equivalences $\Span_{\simeq,\Ee}(\Cc)\simeq \Ee$ and $\Span_{\Ee,\simeq}(\Cc)\simeq \Ee^{\op}$. We will often use these equivalences implicitly to identify $\Cc_{B}^{\op}$ and $\Cc_F$ with subcategories of the span category $\Span_{B,F}(\Cc)$.

\begin{rem}\label{SpanPresLimits}
    By \cite[Lemma 2.4, Theorem 2.18]{TwoVariable}, the category $\AdTrip$ admits limits, they are computed componentwise, and the functor $\Span\colon\AdTrip\xrightarrow{}\Cat$ preserves limits.
\end{rem}

\subsection{Mackey functors and \texorpdfstring{$G$}{G}-spectra}

Mackey functors will play a major role in Section 3, in addition, spectral Mackey functors will serve as our main model for $G$-spectra. We will briefly recall the notion of Mackey functors and how they relate to $G$-spectra.

\begin{defi}\cite[Definition 4.2]{BarwickSpectral}
    Let $\Cc$ be a category. We call $\Cc$ \textit{disjunctive} if
    \begin{enumerate}
        \item $\Cc$ admits pullbacks and finite coproducts,
        \item and for a finite family $(X_i)_{i\in I}$ of objects of $\Cc$, the functor
        \begin{equation*}
            \coprod\colon \prod_{i\in I}\Cc_{/X_{i}} \xrightarrow{} \Cc_{/\coprod_{i\in I}X_{i}} 
        \end{equation*}
        is an equivalence.
    \end{enumerate}
\end{defi}

In particular, for any disjunctive category $\Cc$, the triple $(\Cc,\Cc,\Cc)$ is adequate.

\begin{ex}
    The category of finite $G$-sets $\FCopC_{G}$ is disjunctive.
\end{ex}

\begin{prop}\cite[Proposition 4.3]{BarwickSpectral}
    Let $\Cc$ be a disjunctive category. The category $\Span(\Cc)$ is semiadditive.
\end{prop}

Let $\FunProd(\Aa,\Bb)$ be the full subcategory of $\Fun(\Aa,\Bb)$ spanned by finite product preserving functors.

\begin{defi}\cite[Definition 6.1]{BarwickSpectral}
    Let $\Cc$ be a disjunctive category and let $\Aa$ be a category admitting finite products. We define the category of \textit{$\Aa$-valued Mackey functors on $\Cc$} as
    \begin{equation*}
        \Mack_{\Cc}(\Aa)\coloneq \FunProd(\Span(\Cc),\Aa).
    \end{equation*}
\end{defi}

In the case $\Cc = \FCopC_{G}$, we will write $\Mack_{G}(\Aa)$ instead of $\Mack_{\FCopC_{G}}(\Aa)$. We use spectral Mackey functors as a model for $G$-spectra. We will denote $\Mack_{G}(\Sp)$ by $\Sp_{G}$. The equivalence between the category of spectral Mackey functors and the underlying $\infty$-category of orthogonal $G$-spectra was first shown in \cite{GuiMay}. A purely $\infty$-categorical model of spectral Mackey functors was developed in \cite{BarwickSpectral}.

The cartesian product in the category of finite $G$-sets extends to a symmetric monoidal structure on $\Span(\FCopC_{G})$. The category $\Sp_{G}$ admits a symmetric monoidal structure, which is given by the Day convolution product, with respect to the aforementioned monoidal structure on $\Span(\FCopC_{G})$ and the tensor product on spectra. We denote the monoidal unit by $\Sphere$. There exists a colimit preserving, symmetric monoidal functor
\begin{equation*}
    \Sigma^{\infty}\colon \Spc_{G,*} \xrightarrow{} \Sp_{G},
\end{equation*}
from the category of pointed $G$-spaces equipped with the smash product symmetric monoidal structure. We refer the reader to Appendix A of \cite{Descent} for further details. We will recall the universal property of $G$-spectra.

\begin{theo}\cite[Theorem A.2]{Descent}\label{UnivPropGSp}
    The suspension functor 
    \begin{equation*}
        \Sigma^{\infty}\colon \Spc_{G,*} \xrightarrow{} \Sp_{G}
    \end{equation*}
    is the initial functor among presentably symmetric monoidal functors from $\Spc_{G,*}$ that send representation spheres for all finite orthogonal $G$-representations to invertible objects.
\end{theo}

\begin{ex}
    Let $H$ be a subgroup of $G$. The $H$-geometric fixed point functor 
    \begin{equation*}
        \Phi^{H}\colon \Sp_{G} \xrightarrow{} \Sp 
    \end{equation*}
    is defined, using Theorem \ref{UnivPropGSp}, as the unique presentably symmetric monoidal functor making the square
    \[
        \begin{tikzcd}
            \Spc_{G,*} \arrow[d,"{(-)^{H}}"'] \arrow[r,"{\Sigma^{\infty}}"] &\Sp_{G} \arrow[d,"{\Phi^{H}}"']
            \\
            \Spc_{*} \arrow[r,"{\Sigma^{\infty}}"] &\Sp,
        \end{tikzcd}
    \]
    commute.
\end{ex}

\subsection{Rational \texorpdfstring{$G$}{G}-spectra}

Rational $G$-spectra are the central objects of this paper, as they provide a simplified setting in which we can study normed algebras. While the results in this subsection are standard, we include them for completeness and to fix notation. The exposition is inspired by the ongoing book project \cite{BookProjStable}. 

The category of $G$-spectra is additive. For a spectrum $X$, we denote by $n\colon X\xrightarrow{} X$ the composite
\begin{equation*}
    X \xrightarrow{\Delta} \bigoplus_{n} X \xrightarrow{\nabla} X.
\end{equation*}

\begin{defi}
    Let $\{ p_{1},p_{2},... \}$ be the collection of all prime numbers. The \textit{rational sphere $G$-spectrum $\Sphere_{\RatQ}$} is defined as the following colimit
    \begin{equation*}
        \Sphere_{\RatQ} \coloneq \colim (\Sphere \xrightarrow{p_{1}} \Sphere \xrightarrow{p_{1}p_{2}} \Sphere \xrightarrow{p_{1}p_{2}p_{3}} ...),
    \end{equation*}
    in $\Sp_{G}$. It comes equipped with a morphism $\eta\colon \Sphere \xrightarrow{} \Sphere_{\RatQ}$, defined as the inclusion of the first component of the colimit diagram.
\end{defi}

Next, we recall the definition of the equivariant homotopy groups of Mackey functors. Postcomposition with $\Omega^{\infty}\colon \Sp \xrightarrow{} \Spc$ induces the functor
\begin{equation*}
    \Omega^{\infty}\circ -\colon \Sp_{G} \xrightarrow{} \Mack_{G}(\Spc).
\end{equation*}
In \cite[Proposition 7.4]{BarwickSpectral}, it is shown that $\Omega^{\infty}\circ -$ admits a left adjoint
\begin{equation*}
    \Stab \colon \Mack_{G}(\Spc) \xrightarrow{} \Sp_{G}.
\end{equation*}
Let $\Xx\colon \Span(\FCopC_{G})\xrightarrow{} \Sp$ be a Mackey functor, $H$ a subgroup of $G$, and $n\in \IntZ$. One defines the $n$th $H$-homotopy group of $\Xx$ as 
\begin{equation*}
    \pi_{n}^{H}(\Xx)\coloneq \pi_{0}\Hom_{\Sp_{G}}(\Stab(\Hom_{\Span(\FCopC_{G})}(G/H,-))[n],\Xx),
\end{equation*}
where $\Yy[n]$ denotes the n-fold suspension $\Sigma^{n}\Yy$. This construction is functorial in $\Xx$. A priori, this functor lands in the category of sets. Using the fact that the category $\Sp_{G}$ is additive and that functors $\pi_{n}^{H}(-)$ preserve finite products, one obtains a unique lift of these functors to abelian groups
\begin{equation*}
    \pi_{n}^{H}(-)\colon \Sp_{G} \xrightarrow{} \Ab.
\end{equation*}
The last statement is a consequence of \cite[Remark 2.11]{MultInfLoopSpc}. We obtain a chain of natural equivalences
\begin{align*}
    \pi_{n}^{H}(\Xx) & = \pi_{0}\Hom_{\Sp_{G}}(\Stab(\Hom_{\Span(\FCopC_{G})}(G/H,-))[n],\Xx)
    \\
    & \simeq \pi_{0}\Hom_{\Mack_{G}(\Spc)}(\Hom_{\Span(\FCopC_{G})}(G/H,-),\Omega^{\infty}(\Xx[-n]))
    \\
    & \simeq \pi_{0}(\Omega^{\infty}(\Xx[-n])(G/H))
    \\
    & \simeq \pi_{0}(\Omega^{\infty}(\Xx(G/H)[-n]))
    \\
    & \simeq \pi_n(\Xx(G/H)).
\end{align*}
The first equivalence comes from the adjunction $\Stab \dashv \Omega^{\infty}\circ-$, the second one is the Yoneda Lemma, and the third one follows from the fact that limits and colimits in spectral Mackey functors are computed componentwise. 

The collection of functors $\{\pi_{n}^{H}(-)\}_{H\leq G,n\in \IntZ}$ is jointly conservative. This can be seen directly from the fact that equivalences of spectral Mackey functors are componentwise and that $\{\pi_n\}_{n\in \IntZ}$ are jointly conservative on the category of spectra. In addition, the homotopy groups $\pi_{n}^{H}(-)$ preserve sifted colimits by \cite[Proposition 1.4.3.9]{HA}, so we get
\begin{align*}
    \pi_{n}^{H}(\Xx\otimes \Sphere_{\RatQ}) &\simeq \pi_{n}^{H}(\colim(\Xx \xrightarrow{p_{1}} \Xx \xrightarrow{p_{1}p_{2}} \Xx \xrightarrow{p_{1}p_{2}p_{3}}...))
    \\
    & \simeq \colim(\pi_{n}^{H}(\Xx) \xrightarrow{p_{1}} \pi_{n}^{H}(\Xx) \xrightarrow{p_{1}p_{2}} \pi_{n}^{H}(\Xx) \xrightarrow{p_{1}p_{2}p_{3}}...)
    \\
    & \simeq \pi_{n}^{H}(\Xx)\otimes \RatQ.
\end{align*}

\begin{rem}\label{RatLoc}
    Using the equivalence 
    \begin{equation*}
        \pi_{n}^{H}(\Xx\otimes \Sphere_{\RatQ})\simeq \pi_{n}^{H}(\Xx)\otimes \RatQ,
    \end{equation*}
    one can deduce the following statements by reducing to abelian groups:
    \begin{enumerate}
        \item The morphism 
        \begin{equation*}
            \id_{\Sphere_{\RatQ}} \otimes \eta \colon \Sphere_{\RatQ} \otimes \Sphere \xrightarrow{} \Sphere_{\RatQ}\otimes \Sphere_{\RatQ}
        \end{equation*}
            is an equivalence.
        \item For an object $\Xx\in \Sp_{G}$, the morphism
        \begin{equation*}
            \id_{\Xx} \otimes \eta \colon \Xx\otimes \Sphere \xrightarrow{} \Xx\otimes \Sphere_{\RatQ} 
        \end{equation*}
        is an equivalence if and only if for all natural numbers $n$, the morphism $n\colon \Xx\xrightarrow{} \Xx$ is an equivalence.
        \item In the case of $\Sp$, there is an equivalence $\Sphere_{\RatQ}\simeq H\RatQ$. In particular, one additionally obtains an equivalence with the 0th Morava $K$-theory, $\Sphere_{\RatQ}\simeq K(0)$.
    \end{enumerate}
\end{rem}

Part (1) of Remark \ref{RatLoc}, combined with \cite[Proposition 5.2.7.4]{HTT} and \cite[Proposition 2.2.1.9]{HA}, implies that the essential image of $-\otimes \Sphere_{\RatQ}$, which we denote by $\Sp_{G,\RatQ}$, is a symmetric monoidal category, and that 
\begin{equation*}
    L_{\RatQ} \coloneq -\otimes \Sphere_{\RatQ} \colon \Sp_{G} \xrightarrow{} \Sp_{G,\RatQ}
\end{equation*}
is a symmetric monoidal localization functor. Then part (2) of Remark \ref{RatLoc} implies that $\Sp_{G,\RatQ}$ is the full subcategory of $\Sp_{G}$ on objects $\Xx$ for which all morphisms $n\colon \Xx\xrightarrow{}\Xx$ are invertible.

\begin{rem}\label{RatvsMackRat}
    Consider the adjunction 
    \begin{equation*}
        L_{\RatQ}\colon \Sp \rightleftarrows \Sp_{\RatQ} \colon \mkern-3mu \incl.
    \end{equation*}
    Since the categories $\Span(\FinG)$, $\Sp$, and $\Sp_{\RatQ}$ are semiadditive, postcomposition induces an adjunction 
    \begin{equation*}
        L_{\RatQ}\colon\Mack_{G}(\Sp) \rightleftarrows \Mack_{G}(\Sp_{\RatQ})\colon\mkern-3mu \incl,
    \end{equation*}
    with a fully faithful right adjoint. It follows from part (2) of Remark \ref{RatLoc} that the right adjoint is given by the inclusion of $L_{\RatQ}$-local objects. Hence, we obtain an equivalence 
    \begin{equation*}
        \Mack_{G}(\Sp)_{\RatQ}\simeq \Mack_{G}(\Sp_{\RatQ}).
    \end{equation*}
\end{rem}

\subsection{Parametrized categories}

This paper studies normed algebras and therefore makes use of parametrized $\infty$-category theory. We briefly recall the definitions and results that will be needed. The theory was first developed in \cite{Expose}. A treatment in terms of categories internal to an $\infty$-topos was developed by Martini and Wolf in \cite{martini2021yoneda, martini2021colimits}. Our primary references will be \cite{ParStab}, and \cite{Partial}.

\begin{defi}{\cite[Definition 2.1.1]{ParStab}}
    Let $\Tt$ be a small category. A \textit{$\Tt$-category} is a functor $\underline{\Cc}\colon\Tt^{\op} \xrightarrow{} \Cat$.
    The category of $\Tt$-categories is defined to be the functor category $\Cat_{\Tt}\coloneq \Fun(\Tt^{\op},\Cat)$.
\end{defi}

\begin{notation}
    When $\Tt=\OG$, we refer to $\OG$-categories as $G$-categories and write $\Cat_{G}$ for $\Cat_{\Tt}$.
\end{notation}

Throughout the rest of the section, we assume that $\Tt$ is a small category.

\begin{defi}
    Let $\underline{\Cc}$ be a $\Tt$-category. we define \textit{underlying category of $\underline{\Cc}$} to be the limit of the functor $\underline{\Cc}\colon \Tt^{\op}\xrightarrow{} \Cat$. We denote it by
    \begin{equation*}
        \Gamma(\underline{\Cc})\coloneq \lim_{\Tt^{\op}}\underline{\Cc}.
    \end{equation*}
\end{defi}

\begin{rem}\label{AltDefTCat}
    The inclusion of a small category $\Tt$ into its finite coproduct completion $\FCopC(\Tt)$ induces an equivalence of categories 
    \begin{equation*}
        \Cat_{\Tt}\simeq\Fun^{\times}(\FCopC(\Tt)^{\op},\Cat),
    \end{equation*}
    between the category of $\Tt$-categories and the category of finite product preserving functors $\FCopC(\Tt)^{\op}\xrightarrow{}\Cat$.
    \\
    There is another equivalent description of $\Tt$-categories using cocartesian fibrations. The straightening-unstraightening equivalence \cite[Theorem 3.2.0.1]{HTT} provides an equivalence
    \begin{equation*}
        \Cat_{\Tt}\simeq \Cat^{\text{cc}}_{/\Tt^{\op}},
    \end{equation*}
    between the category of $\Tt$-categories and the subcategory of $\Cat_{/\Tt^{\op}}$ on cocartesian fibrations and morphisms preserving cocartesian maps.
    
    Finally, $\Cat_{\Tt}$ is equivalent to the category of categories internal to the $\infty$-topos $\Pp(\Tt) \coloneq \Fun(\Tt^{\op},\Spc)$. In more detail
    \begin{equation*}
        \Cat_{\Tt} \simeq \Fun ^{R}(\Pp(\Tt)^{\op},\Cat) \simeq \Cat(\Pp(\Tt)),
    \end{equation*}
    where $\Fun^R(\Aa,\Bb)$ denotes the full subcategory of $\Fun(\Aa,\Bb)$ on limit preserving functors, see Section 3.5 in \cite{martini2021yoneda} for more details on this equivalence.
    
    We will freely move between the above descriptions.
\end{rem}

\begin{rem}{\cite[Construction 5.3.1]{EnvAlgPat}}{\cite[Remark 2.13]{NormsEq}}\label{2-Cats}
    Let $\Cc$ be a small category, and let $\Ee \subseteq \Cc$ be a wide subcategory. We define $\Cat^{\Ee\text{-cc}}_{/\Cc}$ to be the non-full subcategory of the slice $\Cat_{/\Cc}$, whose objects are those functors that admit cocartesian lifts of morphisms in $\Ee$, and whose morphisms preserve cocartesian morphisms over $\Ee$. Then the functor $-\times \Cc \colon \Cat \xrightarrow{} \Cat^{\Ee\text{-cc}}_{/\Cc}$ preserves finite products, exhibiting $\Cat^{\Ee\text{-cc}}_{/\Cc}$ as a $\Cat$-module. For each $p\colon\Dd \xrightarrow{} \Cc$ in $\Cat^{\Ee\text{-cc}}_{/\Cc}$, the functor 
    \begin{equation*}
        (-\times\Cc)\times_{\Cc}\Dd\colon\Cat\xrightarrow{} \Cat^{\Ee\text{-cc}}_{/\Cc}
    \end{equation*}
    admits a right adjoint. 
    
    It then follows from \cite[Example C.1.11]{6-fun} that $\Cat^{\Ee\text{-cc}}_{/\Cc}$ enhances to a 2-category. Example C.2.2 in \cite{6-fun}, states that 2-functors correspond to lax-$\Cat$-linear functors. Furthermore, by \cite[Corollary 7.3.2.7]{HA}, an adjunction whose left adjoint is $\Cat$-linear upgrades to an adjunction of 2-categories.
\end{rem}

\begin{ex}
     By Remark \ref{2-Cats} and description from Remark \ref{AltDefTCat}, $\Cat_{\Tt}$ admits an enhancement to a 2-category. We denote its $\Cat$-enrichment by
     \begin{equation*}
         \Fun_{\Tt}(-,-)\colon\Cat_{\Tt}^{\op}\times \Cat_{\Tt} \xrightarrow{} \Cat.
     \end{equation*}
\end{ex}

\begin{cons}\label{evidcons}
    Let $\underline{\Cc}$ be a $\Tt$-category, and let $A$ be an object of $\Tt$. We denote the representable functor $\Hom_{\Tt}(-,A)\colon \Tt^{\op}\xrightarrow{} \Cat$ by $\underline{A}$. One obtains a functor
    \begin{equation*}
        \ev_{\id_{A}}\colon \Fun_{\Tt}(\underline{A},\underline{\Cc}) \xrightarrow{} \underline{\Cc}(A),
    \end{equation*}
    given by evaluating on $\id_{A}\in \underline{A}(A)$. More precisely, using the counit $\ev \colon \const(\Fun_{\Tt}(\underline{A},\underline{\Cc})) \times \underline{A} \xrightarrow{} \underline{\Cc}$, we define $\ev_{\id_{A}}$ to be the following composite
    \begin{equation*}
        \Fun_{\Tt}(\underline{A},\underline{\Cc}) \xrightarrow{(\id,\const_{\id_{A})}} \Fun_{\Tt}(\underline{A},\underline{\Cc})\times \underline{A}(A) \xrightarrow{\ev(A)} \underline{\Cc}(A).
    \end{equation*}
\end{cons}

There is an analogue of the Yoneda Lemma for $\Tt$-categories.

\begin{lem}\cite[Lemma 2.2.7]{ParStab}\label{Yoneda}
    Let $\underline{\Cc}$ be a $\Tt$-category, and let $A$ be an object of $\Tt$. The functor
    \begin{equation*}
        \ev_{\id_{A}}\colon \Fun_{\Tt}(\underline{A},\underline{\Cc}) \xrightarrow{} \underline{\Cc}(A)
    \end{equation*}
    is an equivalence of categories.
\end{lem}

\begin{cons}{\cite[Example 2.1.11]{ParStab}}\label{cofree}
    We consider the cocartesian fibration 
    \begin{equation*}
        \ev_1\colon\Fun([1],\Tt) \xrightarrow{} \Tt.
    \end{equation*}
    After straightening, we obtain a slice functor
    \begin{equation*}
        \Tt_{/-}\colon\Tt \xrightarrow{} \Cat,
    \end{equation*}
    with postcomposition functoriality. We then define a functor
    \begin{equation*}
        \Cofree \bullet) \colon \Cat \xrightarrow{\Fun(-,\bullet)} \Fun(\Cat^{\op},\Cat) \xrightarrow{-\circ (\Tt_{/-})^{\op}} \Fun(\Tt^{\op},\Cat)=\Cat_{\Tt},
    \end{equation*}
    and call it the $\Tt$-category associated to $\Cc$. This functor sends a category $\Cc$ to the functor 
    \begin{equation*}
        \Cofree \Cc)=\Fun((\Tt_{/-})^{\op},\Cc)\colon\Tt^{\op}\xrightarrow{} \Cat.
    \end{equation*}
\end{cons}

\begin{ex}
    Using Construction \ref{cofree} one obtains following $\Tt$-categories:
    \begin{itemize}
        \item[-] The $\Tt$-category of small $\Tt$-spaces $\PSpc_{\Tt} \coloneq \Cofree \Spc)$.
        \item[-] The $\Tt$-category of small pointed $\Tt$-spaces $\PSpc_{\Tt,*} \coloneq \Cofree\Spc)$.
        \item[-] The $\Tt$-category of small $\Tt$-categories $\Pcat_{\Tt} \coloneq \Cofree\textup{cat})$.
    \end{itemize}
\end{ex}

\begin{rem}\label{ParAdj}
    Since $\Cat_{\Tt}$ is a 2-category, one obtains a natural notion of $\Tt$-adjunction. Let $L\colon\underline{\Cc} \xrightarrow{} \underline{\Dd}$ be a $\Tt$-functor. By combining \cite[Proposition 3.2.9]{martini2021colimits} and \cite[Corollary 3.2.11]{martini2021colimits}, $L$ admits a $\Tt$-right adjoint if and only if the following conditions hold:
    \begin{enumerate}
        \item for each $A\in\Tt$, $L_A\colon \underline{\Cc}(A) \xrightarrow{} \underline{\Dd}(A)$ admits a right adjoint $R_A$,
        \item and for each morphism $f\colon A\xrightarrow{} B$ in $\Tt$, the mate of the natural equivalence 
        \begin{equation*}
            L_A \circ \underline{\Cc}(f) \simeq \underline{\Dd}(f) \circ L_B,
        \end{equation*}
        defined as the composite
        \begin{equation*}
            \BC_f\colon \underline{\Cc}(f) \circ R_B \xrightarrow{\eta} R_A\circ L_A\circ\underline{\Cc}(f) \circ R_B \simeq R_A\circ \underline{\Dd}(f)\circ L_B \circ R_B \xrightarrow{\epsilon} R_A\circ\underline{\Dd}(f)
        \end{equation*}
        is an equivalence.
    \end{enumerate}
    In this case, the $\Tt$-right adjoint $R$ is given on an object $A$ by $R_A$, and for a morphism $f\colon A\xrightarrow{} B$, the mate $\BC_f$ makes the naturality square commute. There is a similar description of $\Tt$-left adjoints.
\end{rem}

\begin{rem}\label{ParColim}
    The category $\Cat_{\Tt}$ is cartesian closed. Using the inner hom and the notion of $\Tt$-adjunction, one defines $\Tt$-limits and $\Tt$-colimits. Then, it is natural to consider the notion of a $\Tt$-cocomplete $\Tt$-category. By combining \cite[Corollary 5.4.7]{martini2021colimits} and \cite[Remark 5.4.9]{martini2021colimits}, one obtains a simplified characterization of $\Tt$-cocompleteness. A $\Tt$-category $\underline{\Cc}$ is $\Tt$-cocomplete if and only if the following conditions hold:
    \begin{enumerate}
        \item for each $A\in \Tt$, the category $\underline{\Cc}(A)$ is cocomplete,
        \item for each morphism $f\colon A \xrightarrow{} B$ in $\Tt$, the functor $\underline{\Cc}(f)$ preserves small colimits,
        \item for each morphism $f\colon A \xrightarrow{} B$ in $\Tt$, the functor $\underline{\Cc}(f)$ admits a left adjoint $f_{!}$,
        \item and for each pullback square
        \[
        \begin{tikzcd}
            \Xx \arrow[dr, phantom, "\lrcorner", very near start] \arrow[r, "{\alpha'}"] \arrow[d, "{\beta'}"'] & \underline{B} \arrow[d, "{\beta}"']
            \\
            \underline{A} \arrow[r, "{\alpha}"] & \underline{C}
        \end{tikzcd}
        \]
        in $\Pp(\Tt)$, with $A$, $B$, $C$ objects of $\Tt$, the mate
        \begin{equation*}
            \BC_{(\alpha,\beta)} \colon \alpha'_{!} \circ \underline{\Cc}(\beta') \xrightarrow{} \underline{\Cc}(\beta) \circ \alpha_{!}
        \end{equation*}
        is an equivalence.
    \end{enumerate}
\end{rem}

\begin{defi}{\cite[Definition 2.4.1, Definition 2.4.3]{ParStab}}\label{ParStab}
    A $\Tt$-category $\underline{\Cc}$ is called \textit{$\Tt$-presentable} if $\underline{\Cc} \colon \Tt^{\op} \xrightarrow{} \Cat$ factors through $\PrL$ and $\underline{\Cc}$ is $\Tt$-cocomplete. We denote by $\PrLT$ the subcategory of $\Cat_{\Tt}$ spanned by $\Tt$-presentable categories and $\Tt$-left adjoints.
\end{defi}

There is an analogue of the Adjoint Functor Theorem (\cite[Corollary 5.5.2.9]{HTT}) in the context of $\Tt$-categories.

\begin{prop}\cite[Proposition 2.4.8]{ParStab}\label{PAdjFT}
    If $\underline{\Cc}$ and $\underline{\Dd}$ are large $\Tt$-presentable $\Tt$-categories, a $\Tt$-functor $F\colon \underline{\Cc} \xrightarrow{} \underline{\Dd}$ preserves $\Tt$-colimits if and only if it is a $\Tt$-left adjoint.
\end{prop}

\begin{rem}\label{PPrescolim}
    It is an immediate consequence of Remark \ref{ParAdj}, Proposition \ref{PAdjFT}, and the Adjoint Functor Theorem that a $\Tt$-functor $F\colon \underline{\Cc} \xrightarrow{} \underline{\Dd}$ between $\Tt$-presentable $\Tt$-categories, is $\Tt$-cocontinuous if and only if the following hold:
    \begin{enumerate}
        \item for each $A\in \Tt$, the functor $F_{A}\colon \underline{\Cc}(A) \xrightarrow{} \underline{\Dd}(A)$ is cocontinuous,
        \item and for each morphism $f\colon A\xrightarrow{} B$ in $\Tt$, the mate $\BC_f\colon f_{!}\circ F_{A} \xrightarrow{} F_{B}\circ f_{!}$, of the natural equivalence $F_{A}\circ \underline{\Cc}(f)\simeq \underline{\Dd}(f)\circ F_{B}$, is an equivalence.
    \end{enumerate}
\end{rem}

To define a notion of $\Tt$-semiadditivity, one needs to restrict to the case of an atomic orbital category $\Tt$.

\begin{defi}{\cite[Definition 4.3.1]{ParStab}}
    Let $\Tt$ be a small category. We say that $\Tt$ is \textit{orbital} if $\FCopC[\Tt]$ admits pullbacks. We say that $\Tt$ is \textit{atomic} if every section $s\colon A\xrightarrow{} B$ in $\Tt$ is an equivalence.
\end{defi}

\begin{ex}
    For a finite group $G$, the orbit category $\OG$ is an atomic orbital category.
\end{ex}

\begin{rem}\label{Wirth}
    Let $\Tt$ be an atomic orbital category. Assume that $\underline{\Cc}$ is a pointed $\Tt$-category, meaning that $\underline{\Cc}\colon \Tt^{\op}\xrightarrow{} \Cat$ factors through the category $\Cat_*$ of pointed categories and functors preserving the terminal object. Additionally, assume that $\underline{\Cc}$ admits finite $\Tt$-products and $\Tt$-coproducts, see \cite[Definition 4.2.12]{ParStab}. For each morphism $f\colon A\xrightarrow{}B$ in $\Tt$, the functor $\underline{\Cc}(f)$ admits both left and right adjoints, respectively denoted by $f_{!}$ and $f_{*}$. In Construction \cite[Construction 4.3.8]{ParStab}, the authors define a comparison morphism
    \begin{equation*}
        \Nm_{f}\colon f_{!} \xrightarrow{} f_{*}.
    \end{equation*}
\end{rem}

\begin{defi}{\cite[Definition 4.5.1, Corollary 4.5.7]{ParStab}}\label{PSemiadditive}
    Let $\Tt$ be an atomic orbital category. Let $\underline{\Cc}$ be a $\Tt$-category that satisfies the assumptions from Remark \ref{Wirth}. We call $\underline{\Cc}$ \textit{$\Tt$-semiadditive} if it is fiberwise semiadditive and, for every morphism $f\colon A\xrightarrow{}B$ in $\Tt$, the comparison morphism of Remark \ref{Wirth}
    \begin{equation*}
        \Nm_{f}\colon f_{!} \xrightarrow{} f_{*}
    \end{equation*}
    is an equivalence.
\end{defi}

\begin{defi}{\cite[Definition 6.2.3]{ParStab}}\label{PStable}
    Let $\Tt$ be an atomic orbital category, and let $\underline{\Cc}$ be a $\Tt$-category. We call $\underline{\Cc}$ \textit{$\Tt$-stable} if it is $\Tt$-semiadditive and fiberwise stable.
\end{defi}

\begin{notation}
    For a $\Tt$-category $\underline{\Cc}$, and a morphism $f\colon A\xrightarrow{} B$ in $\Tt$, we will denote the functor $\underline{\Cc}(f)$ simply by $f^{*}$. As in Remark \ref{Wirth}, we will denote the left and right adjoints of $\underline{\Cc}(f)$ by $f_{!}$ and $f_{*}$, respectively.
\end{notation}

\subsection{Parametrized higher algebra}
We recall the notions of $G$-symmetric monoidal $G$-categories and $G$-operads. Our main reference will be \cite{NormsEq}, which deals with more generality than we need. Thus, we restrict our attention to the case of finite $G$-sets.

\begin{defi}{\cite[Definition 2.9, Definition 2.53]{NormsEq}}
    We let $\PSMCatG$ denote the functor category $\Fun(\Span(\FinG),\Cat)$, and call its objects \textit{$G$-symmetric monoidal $G$-precategories}. We denote the full subcategory of $\PSMCatG$ spanned by product preserving functors by $\SMCatG \coloneq \FunProd(\Span(\FinG),\Cat)$, and call its objects \textit{$G$-symmetric monoidal $G$-categories}.\!\footnote{In \cite{NormsEq}, they refer to the objects of these categories as $\FinG$-normed $\FinG$-(pre)categories.} When $G$ is a trivial group, we simply write $\PSMCat$ and $\SMCat$ for $\PSMCatG$ and $\SMCatG$.
\end{defi}

\begin{rem}
    Specializing \cite[Definition 2.52]{NormsEq}, to the case $(\Ff,\Nn)=(\FinG,\FinG)$ one obtains a notion of $(\FinG,\FinG)$-operads. We refer to $(\FinG,\FinG)$-operads simply by $G$-operads. In particular, $G$-operads are functors $\Oo\xrightarrow{} \Span(\FinG)$ of categories, satisfying certain conditions. We omit the definition, as it will not be essential for the rest of the paper.
\end{rem}

\begin{defi}\cite[Definition 2.7, Definition 2.52]{NormsEq}
    We let $\POpG$ denote the category $\Cat_{/\Span(\FinG)}^{\FinG^{\op}\text{-cc}}$, and call its objects \textit{$G$-preoperads}. We denote the full subcategory of $\POpG$ spanned by $G$-operads by $\OpG$.\!\footnote{In \cite{NormsEq}, they refer to the objects of these categories as $(\FinG,\FinG)$-(pre)operads.} When $G$ is a trivial group, we simply write $\POp$ and $\Op$ for $\POpG$ and $\OpG$.
\end{defi}

\begin{ex}
    Combining \cite[Proposition C.1]{NormsMotivic} and \cite[Example 2.47]{NormsEq}, pullback along the inclusion
    \begin{equation*}
        \FCopC_{*}\simeq \Span_{\text{inj,all}}(\FCopC)\xrightarrow{} \Span(\FCopC)
    \end{equation*}
    induces equivalences from $\SMCat$ and $\Op$ to the categories of symmetric monoidal $\infty$-categories and $\infty$-operads in the sense of Lurie. 
\end{ex}

\begin{defi}{\cite[Definition 1.2, Theorem 1.4]{BlumbergHillMainDef}}\label{IndexingSystem}
    We call a wide subcategory $\Ii$ of $\FinG$ an \textit{indexing system} if it is pullback stable, $\Ii$ has all finite coproducts, and finite coproducts in $\Ii$ are created in $\FinG$.
\end{defi}

\begin{ex}\label{IndexingSystemOperad}
    Let $\Ii$ be an indexing system. The pair $(\FinG,\Ii)$ is an extensive span pair in the sense of \cite[Definition 2.2.1]{HighTamb}. By \cite[Proposition 2.2.5]{HighTamb}, the span category $\Span_{\text{all},\Ii}(\FinG)$ is semiadditive, and its finite products are given by finite coproducts in $\FinG$. It follows that the inclusion $\Span_{\text{all},\Ii}(\FinG)\xrightarrow{}\Span(\FinG)$ satisfies the conditions of \cite[Definition 2.52]{NormsEq} and thus is a $G$-operad.
\end{ex}

\begin{notation}
    Let $\underline{\Cc}^{\otimes} \colon \Span(\FinG)\xrightarrow{}\Cat$ be a $G$-symmetric monoidal $G$-category. Let
    \[
        \begin{tikzcd}[sep=tiny]
            &X \arrow[dl, "{\id_{X}}"'] \arrow[dr, "{f}"]
            \\
            X &&Y
        \end{tikzcd}
    \]
    be a morphism in the category $\Span(\FinG)$. We denote the functor $\underline{\Cc}^{\otimes}(f)\colon \underline{\Cc}^{\otimes}(X)\xrightarrow{}\underline{\Cc}^{\otimes}(Y)$ simply by $f_{\otimes}$.
\end{notation}

\begin{rem}\label{LaxFuncCat}
    It follows from Remark \ref{2-Cats}, that the categories $\PSMCatG$ and $\POpG$ are 2-categories. We denote their $\Cat$-enrichments by
    \begin{equation*}
        \FunTenG(-,-) \quad \textup{and} \quad \FunLaxG(-,-),
    \end{equation*}
    respectively. When $G$ is a trivial group, we simply write $\FunTen(-,-)$ and $\FunLax(-,-)$. These $\Cat$-enrichments restrict to $\Cat$-enrichments on the full subcategories $\SMCatG$ and $\OpG$.
    
    Let $p\colon\Cc\xrightarrow{}\Span(\FinG)$, and $q\colon\Dd \xrightarrow{}\Span(\FinG)$ be $G$-preoperads. It is easy to verify the equivalence
    \begin{equation*}
        \FunLaxG(\Cc,\Dd)\simeq \Fun^{\FinG^{\op}\text{-cc}}(\Cc,\Dd)\times_{\Fun(\Cc,\Span(\FinG))}\{p\},
    \end{equation*}
    where $\Fun^{\FinG^{\op}\text{-cc}}(\Cc,\Dd)$ denotes the full subcategory of $\Fun(\Cc,\Dd)$ spanned by functors preserving cocartesian morphisms over $\FinG^{\op}$.
\end{rem}

\begin{rem}
    There exist forgetful 2-functors
    \begin{equation*}
        \fgt \colon \PSMCatG \xrightarrow{} \POpG, \quad \text{and} \quad \fgt\colon \POpG \xrightarrow{} \Cat_{G}.
    \end{equation*}
    The first functor is given by the cocartesian unstraightening functor
    \begin{equation*}
        \Uncc \colon \PSMCatG = \Fun(\Span(\FinG),\Cat) \xrightarrow{} \Cat^{\text{cc}}_{/\Span(\FinG)}\xhookrightarrow{} \Cat^{\FinG^{\op}\text{-cc}}_{/\Span(\FinG)}=\POpG.
    \end{equation*}
    The second functor is simply given by pulling back along the inclusion
    \begin{equation*}
        \OGo\xhookrightarrow{} \Span(\FinG).
    \end{equation*}
    The 2-functoriality follows from both functors being $\Cat$-linear and Remark \ref{2-Cats}. The above functors restrict to 2-functors
    \begin{equation*}
        \fgt \colon \SMCatG \xrightarrow{} \OpG, \quad \text{and} \quad \fgt\colon \OpG \xrightarrow{} \Cat_{G}.
    \end{equation*}
    For the first functor, this follows from \cite[Corollary 2.49]{NormsEq}.
\end{rem}

\begin{rem}\label{Env}
    Let $\Ar_{\FinG}(\Span(\FinG))$ denote the full subcategory of $\Fun([1],\Span(\FinG))$ spanned by the forward morphisms. By Theorem 2.23 in \cite{NormsEq}, the composite functor
    \begin{equation*}
        \Cat_{/\Span(\FinG)}\xrightarrow{\ev_{0}^{*}} \Cat_{/\Ar_{\FinG}(\Span(\FinG))}\xrightarrow{\ev_{1}\circ -}  \Cat_{/\Span(\FinG)},
    \end{equation*}
    restricts to a functor
    \begin{equation*}
        \Env \colon \POp_{G} \xrightarrow{} \PSMCatG,
    \end{equation*}
    which is a 2-left adjoint to the functor
    \begin{equation*}
        \fgt \colon \PSMCatG \xrightarrow{} \POpG.
    \end{equation*}
    It follows from the Theorem 2.57 in \cite{NormsEq}, that the above 2-adjunction restrict to a 2-adjunction
    \begin{equation*}
       \Env\colon\OpG \rightleftarrows \SMCatG\colon\mkern-3mu \fgt.
    \end{equation*}
    By Proposition 2.26 and Theorem 2.57 in \cite{NormsEq}, the functor $\fgt\colon\OpG \xrightarrow{} \Cat_{G}$ also admits a 2-left adjoint
    \begin{equation*}
        \Triv\colon \Cat_{G} \xrightarrow{} \OpG.
    \end{equation*}
    We, by abuse of notation, also denote the 2-left adjoint of $\fgt\colon\SMCatG\xrightarrow{} \Cat_{G}$ by
    \begin{equation*}
        \Env\colon\Cat_{G}\xrightarrow{}\SMCatG.
    \end{equation*}
\end{rem}

\section{A model for the \texorpdfstring{$G$}{G}-category of rational \texorpdfstring{$G$}{G}-spectra}\label{Section3}

This section is devoted to upgrading the equivalence from Theorem 4.2 in \cite{StratCat} to the setting of $G$-categories. In \cite{StratCat}, the author works in the general context of epiorbital categories. We specialize to the case of the orbit category $\OG$. We begin by defining a $G$-category that will model the $G$-category of rational $G$-spectra. To do so, we first present a general construction that works for an arbitrary category and establish some of its basic properties. Using this, we construct a comparison $G$-functor and show that it induces an equivalence between the $G$-categories of rational $G$-spectra and the constructed model, by reducing to the nonparametrized setting.

\subsection{An Algebraic model for the \texorpdfstring{$G$}{G}-category of rational \texorpdfstring{$G$}{G}-spectra}

In this subsection, we define a functor that associates a $G$-category to any category. In the case of rational spectra, this recovers a model for the $G$-category of rational $G$-spectra. This functor is similar to the functor of the associated $G$-category from Construction \ref{cofree}.

\begin{cons}\label{CorOG}
    Let $({\OG}_{/-})^{\simeq}\colon \OG \xrightarrow{}\Cat$ denote the composite 
    \begin{equation*}
        \OG \xrightarrow{{\OG}_{/-}} \Cat \xrightarrow{(-)^{\simeq}} \Cat,
    \end{equation*}
    where the first functor is the slice functor with the postcomposition functoriality.
\end{cons}

\begin{defi}\label{CorCof}
    We denote by $\CorCof \bullet)$ the composite
    \begin{equation*}
        \Cat \xrightarrow{\Fun(-,\bullet)} \Fun(\Cat^{\op},\Cat) \xrightarrow{-\circ ({\OG}_{/-})^{\simeq}} \Fun(\OGo,\Cat)=\Cat_{G},
    \end{equation*}
    where the second functor is the precomposition with the functor from Construction \ref{CorOG}. This functor sends a category $\Cc$ to the functor
    \begin{equation*}
        \CorCof\Cc)= \Fun(({\OG}_{/-})^{\simeq},\Cc)\colon \OGo \xrightarrow{} \Cat.
    \end{equation*}
\end{defi}

\begin{rem}
    Consider the commutative diagram
    \[
    \begin{tikzcd}
        \Fun([1],\OG)_{cc} \arrow[rr,"{\incl}"] \arrow[dr,"{\ev_1}"'] && \Fun([1],\OG) \arrow[dl,"{\ev_1}"]
        \\
        & \OG,
    \end{tikzcd}
    \]
    where $\Fun([1],\OG)_{cc}$ denotes the wide subcategory of $\Fun([1],\OG)$ on cocartesian morphisms. Cocartesian morphisms, in this case, are those commutative squares for which evaluation at 0 is an equivalence. Straightening, we obtain a natural transformation
    \begin{equation*}
        ({\OG}_{/-})^{\simeq} \xrightarrow{} {\OG}_{/-},
    \end{equation*}
    of functors from $\OG$ to $\Cat$, which is given degreewise, by the inclusion of the core.
    
    By 2-functoriality of $\Fun(-,-)$, one obtains a natural transformation
    \begin{equation*}
        \CofreeG -) \xrightarrow{} \CorCof -),
    \end{equation*}
    between the associated $G$-category functor and the functor from Definition \ref{CorCof}.
\end{rem}

The $G$-category $\CorCof\Sp_{\RatQ})$ will serve as our model for the $G$-category of rational $G$-spectra. In the next section, we introduce an enhancement of this model to the $G$-symmetric monoidal $G$-categories. The following two results will be needed in the construction of this enhancement.

\begin{lem}\label{CorCofPres}
    The functor
    \begin{equation*}
        \CorCof -) \colon \Cat \xrightarrow{} \Cat_G
    \end{equation*}
    sends presentable categories to $G$-presentable $G$-categories. It restricts to a functor 
    \begin{equation*}
        \CorCof -) \colon \PrL \xrightarrow{} \PrLG.
    \end{equation*}
\end{lem}
\begin{proof}
    Let $\Cc$ be a presentable category. We first show that $\CorCof \Cc)$ is $G$-presentable. We check the conditions of Definition \ref{ParStab}. For an object $A\in \OG$, the value of this $G$-category at $A$ is $\Fun(({\OG}_{/A})^{\simeq},\Cc)$, which is presentable, since it is a category of functors from a small category to a presentable category. For a morphism $f\colon A\xrightarrow{} B$ in $\OG$, we have that
    \begin{equation*}
        f^{*} \simeq (-\circ (f \circ-))\colon \Fun(({\OG}_{/B})^{\simeq},\Cc) \xrightarrow{} \Fun(({\OG}_{/A})^{\simeq},\Cc),
    \end{equation*}
    preserves all colimits, since colimits in functor categories are computed pointwise. Therefore, $\CorCof \Cc)$ factors through $\PrL$.
    
    It remains to check the conditions of Remark $\ref{ParColim}$. Conditions (1), (2), and (3) hold, since $\CorCof \Cc)$ factors through $\PrL$ and the restriction functors preserve limits. For condition (4), let 
    \[
        \begin{tikzcd}
            \Xx \arrow[dr, phantom, "\lrcorner", very near start] \arrow[r, "{f}"] \arrow[d, "{g}"'] & \underline{B} \arrow[d, "{\beta}"']
            \\
            \underline{A} \arrow[r, "{\alpha}"] & \underline{C}
        \end{tikzcd}
    \]
    be a pullback in $\Pp(\OG)$. The orbit category $\OG$ is orbital, hence there is a decomposition $\Xx \simeq \coprod_{i\in I} \underline{D}_i$. We show that the natural transformation $\BC_{(\alpha,\beta)}$, from Remark \ref{ParColim}, makes the diagram
    \[
        \begin{tikzcd}
            \prod_{i\in I}\Fun(({\OG}_{/D_{i}})^{\simeq},\Cc) \arrow[r, "{\prod(f_{i})_{!}}"]  & \prod_{i\in I}\Fun(({\OG}_{/B})^{\simeq},\Cc)  \arrow[r, "{\coprod}"]  & \Fun(({\OG}_{/B})^{\simeq},\Cc) 
            \\
            \Fun(({\OG}_{/A})^{\simeq},\Cc) \arrow[rr, "{\alpha_{!}}"] \arrow[u, "{(g_{i}^{*})_{i \in I}}"] && \Fun(({\OG}_{/C})^{\simeq},\Cc) \arrow[u, "{\beta^{*}}"]
        \end{tikzcd}
    \]
    commute. We use the pointwise formula for Kan extensions. Let $\Yy\colon ({\OG}_{/A})^{\simeq} \xrightarrow{} \Cc$ be a functor, and let $\gamma \colon E\xrightarrow{} B$ be a morphism in $\OG$. We have equivalences
    \begin{equation*}
        (\beta^{*} \alpha_{!} (\Yy))(\gamma) \simeq \alpha_{!} (\Yy)(\beta \gamma) \simeq \colim_{(({\OG}_{/A})^{\simeq})_{/\beta \gamma}}\Yy,
    \end{equation*}
    and
    \begin{equation*}
        \coprod_{i\in I} ((f_{i})_{!} g_{i}^{*}(\Yy))(\gamma) \simeq \coprod_{i\in I} \colim_{(({\OG}_{/D_{i}})^{\simeq})_{/\gamma}}g_{i}^{*}(\Yy).
    \end{equation*}
    It suffices to check that the functor
    \begin{equation*}
        \coprod_{i\in I} (({\OG}_{/D_{i}})^{\simeq})_{/\gamma} \xrightarrow{} (({\OG}_{/A})^{\simeq})_{/\beta \gamma}, \quad (X,X\xrightarrow{\lambda} D_{i},X\xrightarrow{\simeq}E)\mapsto (X,X\xrightarrow{g_{i}\lambda}A,X\xrightarrow{\simeq}E).
    \end{equation*}
    obtained using $\beta$ and $(g_i)_{i\in I}$ is cofinal. We check the stronger statement that this functor is, in fact, an equivalence. We note that for a functor $F\colon \Bb \xrightarrow{} \Cc$ between (1,1)-categories and an object $c\in\Cc$, the slice category $\Bb_{/c}=\Bb\times_{\Cc}\Cc_{/c}$ is given by a strict (1,1)-categorical pullback, this follows from the forgetful functor $\Cc_{/c}\xrightarrow{}\Cc$ being an isofibration. First, we check fully faithfulness. Injectivity on Hom sets is immediate. For surjectivity, consider a morphism
    \begin{equation*}
        (X,X\xrightarrow{g_{i}\lambda}A,X\xrightarrow{\simeq}E) \xrightarrow{} (Y,Y\xrightarrow{g_{j}\nu}A,Y\xrightarrow{\simeq}E).
    \end{equation*}
    The invertibility of $X\xrightarrow{\simeq}E$ and $Y\xrightarrow{\simeq}E$, together with the fact that $\Xx \simeq \coprod_{i\in I} \underline{D}_i$ is a pullback, shows that $i=j$ and that the morphism $X\xrightarrow{} Y$ is also a morphism between $(X,X\xrightarrow{\lambda}D_{i},X\xrightarrow{\simeq}E)$ and $(X,X\xrightarrow{\mu}D_{j},Y\xrightarrow{\simeq}E)$.
    
    Now let $(Z,Z\xrightarrow{\lambda}A,Z\xrightarrow{\simeq}E)$ be an object of $(({\OG}_{/A})^{\simeq})_{/\beta \gamma}$. This object can be visualized as a commutative diagram
    \[
        \begin{tikzcd}
            Z  \arrow[d, "{\lambda}"'] \arrow[r, "{\simeq}"] & E  \arrow[r, "{\gamma}"] & B \arrow[d, "{\beta}"']
            \\
            A \arrow[rr, "{\alpha}"] && C,
        \end{tikzcd}
    \]
    from which we obtain a morphism $\underline{Z} \xrightarrow{} \Xx$, which necessarily factors through a unique summand $\underline{D}_{i}$. We thus obtain an object $(Z,Z\xrightarrow{\hat{\lambda}}D_{i},Z\xrightarrow{\simeq}E)\in \coprod_{i\in I} (({\OG}_{/D_{i}})^{\simeq})_{/\gamma}$, mapping to $(Z,Z\xrightarrow{\lambda}A,Z\xrightarrow{\simeq}E)$. This proves essential surjectivity. Hence, the functor
    \begin{equation*}
        \coprod_{i\in I} (({\OG}_{/D_{i}})^{\simeq})_{/\gamma} \xrightarrow{} (({\OG}_{/A})^{\simeq})_{/\beta \gamma},
    \end{equation*}
    is an equivalence, showing that $\BC_{(\alpha,\beta)}$ is an equivalence.
    
    Finally, we show that for a left adjoint functor $F\colon \Cc\xrightarrow{}\Dd$ between presentable categories, the $G$-functor
    \begin{equation*}
        \CorCof F)\colon \CorCof \Cc) \xrightarrow{} \CorCof \Dd)
    \end{equation*}
    is a $G$-left adjoint. By Proposition \ref{PAdjFT}, it suffices to check that $F$ is $\Tt$-cocontinuous. We check the conditions of Remark \ref{PPrescolim}. Condition (1) is immediate. Condition (2) follows from the pointwise formula for left Kan extensions combined with the fact that $F$ preserves colimits. 
\end{proof}

\begin{prop}\label{CorCofGStableGpresentable}
    Let $\Cc$ be a stable, presentable category. The $G$-category 
    \begin{equation*}
        \CorCof \Cc)
    \end{equation*}
    is a $G$-stable, $G$-presentable $G$-category.
\end{prop}
\begin{proof}
    It follows directly from Lemma \ref{CorCofPres} that $\CorCof \Cc)$ is $G$-presentable. The fibers 
    \begin{equation*}
        \Fun(({\OG}_{/X})^{\simeq},\Cc)
    \end{equation*}
    are stable, since $\Cc$ is stable and functor categories from a small category to a stable category are themselves stable. The restriction morphisms are exact, since they preserve arbitrary colimits. All that is left to show is $G$-semiadditivity. Note that fiberwise semiadditivity follows from the previous part. Let $f\colon X\xrightarrow{} Y$ be a morphism in $\OG$. We need to show that the norm morphism $\Nm_{f}$ from Remark \ref{Wirth} is an equivalence. Note that the fibers of the functor $f\circ-\colon ({\OG}_{/X})^{\simeq}\xrightarrow{} ({\OG}_{/Y})^{\simeq}$ are given by finite coproducts of contractible spaces. Because of this, the norm morphisms coincide with the ones defined in Construction 6.1.6.4 in \cite{HA}. The Proposition 6.1.6.12 in \cite{HA}, reduces the statement to $\Cc$ being semiadditive, finishing the proof.
\end{proof}

\subsection{Parametrized comparison functor}
As stated before, the $G$-category $\CorCof\Sp_{\RatQ})$ will be our model for the $G$-category of rational $G$-spectra. Having defined this model, we now construct a comparison $G$-functor from the $G$-category of rational $G$-spectra to our model. Componentwise, this $G$-functor is given by the functor from Definition 2.31 in \cite{StratCat}. In the nonparametrized setting, the content of this section is contained in section 2 of \cite{StratCat}. 

We begin by defining $G$-categories of Mackey functors, which will play an essential role in the construction of the comparison $G$-functor.

\begin{cons}\label{PMackSpan}
    Let $F\colon \OG \xrightarrow{} \Cat$ be a functor. Assume that for each $X\in \OG$, the category $F(X)$ is disjunctive. Furthermore, assume that for every morphism $f\colon X \xrightarrow{} Y$ the functor $F(f)\colon F(X) \xrightarrow{} F(Y)$ preserves pullbacks and finite coproducts. Then the functor $F$ factors as
    \begin{equation*}
        \OG \xrightarrow{(F(-),\text{all},\text{all})} \AdTrip \xrightarrow{\fgt} \Cat.
    \end{equation*}
    More precisely, for an object $X\in \OG$ we have
    \begin{equation*}
        (F(-),\text{all},\text{all})(X) = (F(X),F(X),F(X)).
    \end{equation*}
    Postcomposing with the $\Span$ functor, one obtains a functor
    \begin{equation*}
        \Span(F(-),\text{all},\text{all})\colon \OG \xrightarrow{} \Cat.
    \end{equation*}
    Since each $F(X)$ is disjunctive and each $F(f)$ preserves finite products, $\Span(F(-),\text{all},\text{all})$ factors through $\Cat^{\oplus}$, the subcategory of $\Cat$ spanned by semiadditive categories and direct sum preserving functors. We denote this functor again by
    \begin{equation*}
        \Span(F(-),\text{all},\text{all}) \colon \OG \xrightarrow{} \Cat^{\oplus}.
    \end{equation*}
\end{cons}

\begin{defi}
    Let $F \colon \OG \xrightarrow{} \Cat$ be a functor satisfying the conditions of Construction \ref{PMackSpan}, and let $\Aa$ be a category admitting finite products. We define the \textit{$G$-category of $\Aa$-valued Mackey functors on $F$} as the composite
    \begin{equation*}
        \OGo \xrightarrow{\Span(F(-),\text{all},\text{all})} (\Cat^{\times})^{\op} \xrightarrow{\FunProd(-,\Aa)} \Cat,
    \end{equation*}
    where the first functor is the one from Construction \ref{PMackSpan}. We denote this composite by
    \begin{equation*}
        \PMack_{F}(\Aa) \coloneq \FunProd(\Span(F(-),\text{all},\text{all}),\Aa).
    \end{equation*}
\end{defi}

We note that $\PMack_{F}(\Aa)$ is functorial in both $F$ and $\Aa$. To be more precise, let $\Cat^{disj}$ denote the subcategory of $\Cat$ spanned by disjunctive categories and functors preserving pullbacks and finite coproducts. The $G$-category of Mackey functors construction defines a functor
\begin{equation*}
    \PMack_{\bullet}(-)\colon \Fun(\OG,\Cat^{disj})^{\op}\times \Cat^{\times} \xrightarrow{} \Cat.
\end{equation*}
We now introduce several examples of functors $F\colon\OG\xrightarrow{}\Cat$ satisfying the conditions of Construction \ref{PMackSpan}.

\begin{ex}
    Consider the slice functor 
    \begin{equation*}
        {\OG}_{/-} \colon \OG \xrightarrow{} \Cat
    \end{equation*}
    with the postcomposition functoriality. Composing with the finite coproduct completion $\FCopC[-]\colon \Cat \xrightarrow{} \Cat^{\amalg}$ we obtain a functor
    \begin{equation*}
        \FCopC[{\OG}_{/-}] \colon \OG \xrightarrow{} \Cat.
    \end{equation*}
    The inclusion $\OG\xrightarrow{} \FinG$ induces a natural equivalence $\FCopC[{\OG}_{/-}]\simeq {\FinG}_{/-}$, between the functor constructed above and the restriction of the slice functor ${\FinG}_{/-}$ to the orbit category. In particular, ${\FinG}_{/-}$ satisfies the conditions of Construction \ref{PMackSpan}. Let $\Aa$ be a category admitting finite products. We denote the $G$-category of $\Aa$-valued Mackey functors on ${\FinG}_{/-}$, by 
    \begin{equation*}
        \PMack_{G}(\Aa)\coloneq \PMack_{{\FinG}_{/-}}(\Aa).
    \end{equation*}
\end{ex}

\begin{ex}\cite[Construction 5.4.1]{HighTamb}
    The $G$-category
    \begin{equation*}
        \PSp_{G}\coloneq \PMack_{G}(\Sp)
    \end{equation*}
    will be our model for the $G$-category of $G$-spectra.
\end{ex}

\begin{cons}\label{leqGH}
    For each $H\leq G$, we define a subfunctor 
    \begin{equation*}
        ({\OG}_{/-})^{\leq G/H} \colon \OG \xrightarrow{} \Cat
    \end{equation*}
    of 
    \begin{equation*}
        {\OG}_{/-} \colon \OG \xrightarrow{} \Cat.
    \end{equation*}
    For $A\in \OG$, the $({\OG}_{/A})^{\leq G/H}$ is defined to be the full subcategory of ${\OG}_{/A}$ spanned by morphisms $B \xrightarrow{} A$ in $\OG$, that admit a morphism from $G/g^{\inv}Hg \xrightarrow{\alpha} A$ for some choice of $g \in G$ and the morphism $\alpha$ induced by the inclusion of subgroups of $G$.\!\footnote{Note that we could have instead picked an arbitrary morphism $\alpha$ and have avoided conjugating the subgroup $H$.} It is immediate that $({\OG}_{/-})^{\leq G/H}$ is a subfunctor of ${\OG}_{/-}$. We denote by $({\FinG}_{/-})^{\leq G/H}$ the following composite
    \begin{equation*}
        ({\FinG}_{/-})^{\leq G/H} \colon \OG \xrightarrow{ ({\OG}_{/-})^{\leq G/H}} \Cat \xrightarrow{\FCopC[-]} \Cat.
    \end{equation*}
    Applying $\FCopC[-]$ to the inclusion $({\OG}_{/-})^{\leq G/H} \xhookrightarrow{} {\OG}_{/-}$ we obtain a natural inclusion 
    \begin{equation*}
        i^{\leq G/H}\colon ({\FinG}_{/-})^{\leq G/H} \xhookrightarrow{} {\FinG}_{/-}.
    \end{equation*}
    Note that for each $A\in \OG$ the functor $i^{\leq G/H}_{A}\colon ({\FinG}_{/A})^{\leq G/H} \xhookrightarrow{} {\FinG}_{/A}$ is fully faithful, as it is a left Kan extension of the fully faithful functor.
\end{cons}

\begin{rem}\cite[Lemma 2.21]{StratCat}\label{leqGHAdj}
    In the setting of Construction \ref{leqGH}, for each $A \in \OG$, the functor 
    \begin{equation*}
        i^{\leq G/H}_{A}\colon ({\FinG}_{/A})^{\leq G/H} \xhookrightarrow{} {\FinG}_{/A}
    \end{equation*}
    admits a right adjoint 
    \begin{equation*}
        j^{\leq G/H}_{A} \colon {\FinG}_{/A} \xrightarrow{} ({\FinG}_{/A})^{\leq G/H}.
    \end{equation*}
    In more detail, $\FCopC[\Cc]$ is a full subcategory of $\Pp(\Cc)$ spanned by finite coproducts of representables. The functor $i^{\leq G/H}_{A}$ is then defined as the restriction of the left adjoint of precomposing with the inclusion $({\OG}_{/A})^{\leq G/H} \xhookrightarrow{} {\OG}_{/A}$. We define $j^{\leq G/H}_{A} \colon {\FinG}_{/A} \xrightarrow{} ({\FinG}_{/A})^{\leq G/H}$ as the restriction of precomposing with the inclusion $({\OG}_{/A})^{\leq G/H} \xhookrightarrow{} {\OG}_{/A}$. For precomposition to restrict properly, it is enough to note that for any morphism $\alpha \colon B\xrightarrow{} A$ in ${\OG}_{/A}$, we have
    \[ \Hom_{{\OG}_{/A}}(-,\alpha)|_{({\OG}_{/A})^{\leq G/H}} \simeq \begin{cases} 
          \Hom_{({\OG}_{/A})^{\leq G/H}}(-,\alpha) & \alpha \in ({\OG}_{/A})^{\leq G/H} \\
          \emptyset & else.
       \end{cases}
    \]
    In particular, the adjunction between the left Kan extension and precomposition restricts to the adjunction
    \begin{equation*}
        i^{\leq G/H}_{A} \colon ({\FinG}_{/A})^{\leq G/H} \rightleftarrows {\FinG}_{/A} \colon \mkern-3mu j^{\leq G/H}_{A}.
    \end{equation*}
    Furthermore, the functors $j^{\leq G/H}_{A}$, for various $A \in \OG$, assemble into a natural transformation
    \begin{equation*}
        j^{\leq G/H} \colon {\FinG}_{/-} \xrightarrow{} ({\FinG}_{/-})^{\leq G/H}.
    \end{equation*}
    To see this it is enough to observe that each $j^{\leq G/H}_{A}$ preserves finite coproducts and that for each morphism $\gamma \colon A\xrightarrow{} C$ in $\OG$ we have
    \begin{equation*}
        (\gamma \circ -)_{!}(\Hom_{\Pp({\OG}_{/A})}(-,\alpha)|_{({\OG}_{/A})^{\leq G/H}})  \simeq \Hom_{\Pp({\OG}_{/C})}(-,\gamma \circ \alpha)|_{({\OG}_{/C})^{\leq G/H}}.
    \end{equation*}
\end{rem}

\begin{lem}
    The functor 
    \begin{equation*}
        ({\FinG}_{/-})^{\leq G/H}\colon \OG \xrightarrow{} \Cat
    \end{equation*}
    from Construction \ref{leqGH} satisfies the conditions of Construction \ref{PMackSpan}.
\end{lem}

\begin{proof}
    To see this, first note that for each $A\in \OG$, the category $({\FinG}_{/A})^{\leq G/H}$ is disjunctive by combining Lemma 2.14, Example 2.3, and the discussion before Lemma 2.21 in \cite{StratCat}. Now we check that for a morphism $f\colon A\xrightarrow{}B$ in $\OG$, the functor $\FCopC[f\circ-] \colon ({\FinG}_{/A})^{\leq G/H} \xrightarrow{} ({\FinG}_{/B})^{\leq G/H}$ preserves coproducts and pullbacks. It preserves coproducts directly from the definition. For pullback preservation, we note that, pullbacks in $({\FinG}_{/A})^{\leq G/H}$ are computed by first applying the functor $i^{\leq G/H}_{A}$, then taking pullbacks, and finally applying the functor $j^{\leq G/H}_{A}$. This is a consequence of the equivalence $\id_{({\FinG}_{/A})^{\leq G/H}} \simeq j^{\leq G/H}_{A}i^{\leq G/H}_{A}$, given by the unit of the adjunction. We have a commutative diagram
    \[
        \begin{tikzcd}
            ({\FinG}_{/A})^{\leq G/H}  \arrow[r, "{i^{\leq G/H}_{A}}"] \arrow[d, "{\FCopC[f\circ-]}"'] & {\FinG}_{/A} \arrow[r, "{j^{\leq G/H}_{A}}"] \arrow[d, "{f\circ-}"'] & ({\FinG}_{/A})^{\leq G/H} \arrow[d,"{\FCopC[f\circ-]}"']
            \\
            ({\FinG}_{/B})^{\leq G/H} \arrow[r, "{i^{\leq G/H}_{B}}"] & {\FinG}_{/B} \arrow[r, "{j^{\leq G/H}_{B}}"] & ({\FinG}_{/B})^{\leq G/H}.
        \end{tikzcd}
    \]
    Now, for $X,Y,Z\in ({\FinG}_{/A})^{\leq G/H}$ one computes 
    \begin{align*}
         \FCopC[f\circ-](X\times_{Y}Z) & \simeq \FCopC[f\circ-]j^{\leq G/H}_{A}(i^{\leq G/H}_{A}(X)\times_{i^{\leq G/H}_{A}(Y)}i^{\leq G/H}_{A}(Z))
         \\
         & \simeq j^{\leq G/H}_{B}(f\circ-)(i^{\leq G/H}_{A}(X)\times_{i^{\leq G/H}_{A}(Y)}i^{\leq G/H}_{A}(Z))
         \\
         & \simeq j^{\leq G/H}_{B}((f\circ-)i^{\leq G/H}_{A}(X)\times_{(f\circ-)i^{\leq G/H}_{A}(Y)}(f\circ-)i^{\leq G/H}_{A}(Z))
         \\
         & \simeq j^{\leq G/H}_{B}(i^{\leq G/H}_{B}\FCopC[f\circ-](X)\times_{i^{\leq G/H}_{B}\FCopC[f\circ-](Y)}i^{\leq G/H}_{B}\FCopC[f\circ-](Z))
         \\
         & \simeq \FCopC[f\circ-](X)\times_{\FCopC[f\circ-](Y)}\FCopC[f\circ-](Z).
    \end{align*}
    The first and the last equivalences follow from the remark above about pullbacks in the category $({\FinG}_{/A})^{\leq G/H}$. The second and the fourth equivalences follow from the commutative diagram above. Finally, the third equivalence is the consequence of the fact that $f\circ -\colon {\FinG}_{/A} \xrightarrow{} {\FinG}_{/B}$ preserves pullbacks.
\end{proof}

\begin{ex}
    We have shown in the lemma above that the functor $({\FinG}_{/-})^{\leq G/H}$ satisfies conditions of Construction \ref{PMackSpan}. Let $\Aa$ be a category admitting finite products. We denote the $G$-category of $\Aa$-valued Mackey functors on $({\FinG}_{/-})^{\leq G/H}$, by 
    \begin{equation*}
        \PMack_{\leq G/H}(\Aa) \coloneq \PMack_{({\FinG}_{/-})^{\leq G/H}}(\Aa).
    \end{equation*}
\end{ex}

\begin{ex}\label{GroupoidGH}
    Let $G/H$ be an object in $\OG$. We denote by 
    \begin{equation*}
        \{G/H\}^{\simeq}_{/-} \colon \OG \xrightarrow{} \Cat
    \end{equation*}
    the subfunctor of ${\OG}_{/-} \colon \OG \xrightarrow{} \Cat$, for $A\in \OG$ given by a full subcategory of ${\OG}_{/A}$ spanned by the objects equivalent to $G/g^{\inv}Hg\xrightarrow{\alpha} A$ for some choice of $g\in G$ and the morphism $\alpha$ induced by inclusion of subgroups of $G$. The composite
    \begin{equation*}
        \FCopC[\{G/H\}^{\simeq}_{/-}] \colon \OG \xrightarrow{\{G/H\}^{\simeq}_{/-}} \Cat \xrightarrow{\FCopC[-]} \Cat
    \end{equation*}
    satisfies conditions of Construction \ref{PMackSpan}. For each $A\in \OG$ the category $\FCopC[\{G/H\}^{\simeq}_{/A}]$ admits pullbacks, since already $\{G/H\}^{\simeq}_{/A}$ admits pullbacks, which are the same as commutative squares. Then $\FCopC[\{G/H\}^{\simeq}_{/A}]$ is automatically disjunctive, as the finite coproduct completion is always disjunctive if it admits pullbacks. For any morphism $f\colon A\xrightarrow{} B$ in $\OG$, the functor $\FCopC[f\circ-] \colon \FCopC[\{G/H\}^{\simeq}_{/A}]\xrightarrow{} \FCopC[\{G/H\}^{\simeq}_{/B}]$ preserves finite coproducts by definition. It also preserves pullbacks, since it preserves commutative squares before passing to finite coproduct completions.
    
    Let $\Aa$ be a category admitting finite products. We denote the $G$-category of $\Aa$-valued Mackey functors on $\FCopC[\{G/H\}^{\simeq}_{/-}]$, by 
    \begin{equation*}
        \PMack_{G/H}(\Aa)\coloneq \PMack_{\FCopC[\{G/H\}^{\simeq}_{/-}]}(\Aa).
    \end{equation*}
\end{ex}

\begin{cons}\label{GroupoidGHcomparison}
    In the context of Example \ref{GroupoidGH}, the inclusions $\{G/H\}^{\simeq}_{/A}\xhookrightarrow{} \FCopC[\{G/H\}^{\simeq}_{/A}]$ for various $A\in \OG$ assemble into a natural transformation
    \begin{equation*}
        \{G/H\}^{\simeq}_{/-} \xrightarrow{} \FCopC[\{G/H\}^{\simeq}_{/-}].
    \end{equation*}
    To be more precise, this natural transformation is the same as the unit of the adjunction
    \begin{equation*}
        \FCopC[-]_{*} \colon \Fun(\OG,\Cat) \rightleftarrows \Fun(\OG,\Cat^{\amalg}) \colon \mkern-3mu \fgt_{*}.
    \end{equation*}
    This natural transformation can be lifted to the natural transformation
    \begin{equation*}
        (\{G/H\}^{\simeq}_{/-},\text{all},\text{all})\xrightarrow{} (\FCopC[\{G/H\}^{\simeq}_{/-}],\text{all},\text{all})
    \end{equation*}
    of functors into the category $\AdTrip$. This is a direct consequence of the inclusion $\{G/H\}^{\simeq}_{/-} \xrightarrow{} \FCopC[\{G/H\}^{\simeq}_{/-}]$ preserving pullbacks in each degree. We apply the functor $\Span$ to obtain a natural transformation
    \begin{equation*}
        \Span(\{G/H\}^{\simeq}_{/-})\xrightarrow{} \Span(\FCopC[\{G/H\}^{\simeq}_{/-}]).
    \end{equation*}
    Observe that the functors
    \begin{equation*}
        \{G/H\}^{\simeq}_{/A} \xrightarrow{} \Span(\{G/H\}^{\simeq}_{/A}), \quad X \mapsto X, \quad X\xrightarrow{f} Y \mapsto
        \begin{tikzcd}[sep=tiny]
            &X \arrow[dl, "{\id_{X}}"'] \arrow[dr, "{f}"]
            \\
            X && Y,
        \end{tikzcd}
    \end{equation*}
    between (2,1)-categories, are equivalences. They assemble into a natural equivalence
    \begin{equation*}
        \{G/H\}^{\simeq}_{/-} \xrightarrow{\simeq} \Span(\{G/H\}^{\simeq}_{/-}).
    \end{equation*}
    Precomposing with the composite natural transformation
    \begin{equation*}
        \{G/H\}^{\simeq}_{/-} \xrightarrow{\simeq} \Span(\{G/H\}^{\simeq}_{/-})\xrightarrow{} \Span(\FCopC[\{G/H\}^{\simeq}_{/-}]),
    \end{equation*}
    one obtains a $G$-functor
    \begin{equation*}
        \theta^{G/H}\colon \PMack_{G/H}(\Aa) \xrightarrow{} \Fun(\{G/H\}^{\simeq}_{/-},\Aa).
    \end{equation*}
\end{cons}

\begin{lem}\label{thetaequiv}
    The $G$-functor 
    \begin{equation*}
        \theta^{G/H}\colon \PMack_{G/H}(\Aa) \xrightarrow{} \Fun(\{G/H\}^{\simeq}_{/-},\Aa),
    \end{equation*}
    from Construction \ref{GroupoidGHcomparison}, is an equivalence.
\end{lem}
\begin{proof}
    The statement can be checked objectwise, where it is a content of Theorem 2.27 in \cite{StratCat}.
\end{proof}

\begin{cons}\label{leqGHcomparison}
    Let $\Aa$ be a presentable, semiadditive category. We consider the natural transformation
    \begin{equation*}
        j^{\leq G/H} \colon {\FinG}_{/-} \xrightarrow{} ({\FinG}_{/-})^{\leq G/H},
    \end{equation*}
    from Remark \ref{leqGHAdj}. For any object $A$ of $\OG$, the functor $j^{\leq G/H}_{A}$ preserves pullbacks and finite coproducts. Therefore, the natural transformation $j^{\leq G/H}$ lifts to a natural transformation
    \begin{equation*}
        j^{\leq G/H} \colon ({\FinG}_{/-},\text{all},\text{all}) \xrightarrow{} (({\FinG}_{/-})^{\leq G/H},\text{all},\text{all})
    \end{equation*}
    between functors into the category $\AdTrip$. The induced natural transformation $\Span(j^{\leq G/H})$ objectwise preserves direct sums. Precomposing with the natural transformation $\Span(j^{\leq G/H})$ we obtain a $G$-functor
    \begin{equation*}
        (j^{\leq G/H})^{*} \colon \PMack_{\leq G/H}(\Aa) \xrightarrow{} \PMack_{G}(\Aa).
    \end{equation*}
    For each $B \in \OG$, the functor $(j^{\leq G/H})^{*}_{B}\simeq -\circ \Span(j^{\leq G/H}_{B})$ admits a left adjoint given by the left Kan extension along $\Span(j^{\leq G/H}_{B})$, see \cite[Lemma 2.20]{StratCat}. We denote this left adjoint by
    \begin{equation*}
        \phi^{\leq G/H}_{B}\colon \PMack_{G}(\Aa)(B) \xrightarrow{} \PMack_{\leq G/H}(\Aa)(B).
    \end{equation*}
    We note that the $G$-functor $(j^{\leq G/H})^{*}$ is the one we get by the functoriality of $\PMack_{\bullet}(\Aa)$.
\end{cons}

We check that the objectwise adjunction in the above construction lifts to a $G$-adjunction.

\begin{lem}\label{leqGhGrpGHCompParametrizedAdj}
    Let $\Aa$ be a presentable, semiadditive category. For various $B\in \OG$, the functors $\phi^{\leq G/H}_{B}\colon \PMack_{G}(\Aa)(B) \xrightarrow{} \PMack_{\leq G/H}(\Aa)(B)$ assemble into a $G$-functor 
    \begin{equation*}
        \phi^{\leq G/H}\colon \PMack_{G}(\Aa) \xrightarrow{} \PMack_{\leq G/H}(\Aa),
    \end{equation*}
    which is a $G$-left adjoint to the $G$-functor
    \begin{equation*}
        (j^{\leq G/H})^{*} \colon \PMack_{\leq G/H}(\Aa) \xrightarrow{} \PMack_{G}(\Aa).
    \end{equation*}
\end{lem}

\begin{proof}
    We check the conditions dual to those of Remark \ref{ParAdj}. Condition (1) follows from Construction \ref{leqGHcomparison}. Fix a morphism $f\colon B\xrightarrow{} C$ in $\OG$. For condition (2), we need to check that the Beck-Chevalley map
    \begin{equation*}
        \BC_{f}\colon \phi^{\leq G/H}_{B}\circ f^{*} \xrightarrow{} f^{*}\circ \phi^{\leq G/H}_{C}
    \end{equation*}
    is an equivalence. The natural transformation $\BC_f$ is obtained by passing to the left adjoints of the horizontal functor in the diagram 
    \[
        \begin{tikzcd}[column sep=huge]
            \PMack_{\leq G/H}(\Aa)(C) \arrow[r, "{(j^{\leq G/H})^{*}_{C}}"] \arrow[d, "{f^{*}}"'] & \PMack_{G}(\Aa)(C) \arrow[d, "{f^{*}}"']
            \\
            \PMack_{\leq G/H}(\Aa)(B) \arrow[r, "{(j^{\leq G/H})^{*}_{B}}"] & \PMack_{G}(\Aa)(B).
        \end{tikzcd}
    \]
    The map $\BC_f$ is equivalently given by first passing to right adjoints of the vertical functors in the above diagram and then passing to a total mate. Therefore, it is enough to check that the natural transformation
    \begin{equation*}
        (j^{\leq G/H})^{*}_{C}\circ f_{*} \xrightarrow{} f_{*}\circ (j^{\leq G/H})^{*}_{B},
    \end{equation*}
    obtained by passing to the right adjoints of the vertical functors in the above diagram is an equivalence. For this, we first study the functors $f_{*}$ for both $G$-categories.
    
    Let's first consider the functor $f_{*}$ for the $G$-category $\PMack_{G}(\Aa)$. For finite $G$-sets, we have an adjunction
    \begin{equation}\label{EqPostCompAdjPullb}
        f\circ - \colon {\FinG}_{/B} \rightleftarrows {\FinG}_{/C} \colon \mkern-3mu B\times_{C}-,
    \end{equation}
    with the right adjoint given by pulling back along $f\colon B\xrightarrow{} C$. Consider arbitrary morphisms
    \[
        \begin{tikzcd}[sep=tiny]
            X \arrow[rr, "{\alpha}"] \arrow[rd] && Y \arrow[dl] && X' \arrow[rr, "{\alpha'}"] \arrow[rd] && Y' \arrow[dl]
            \\
            & B, &&&& C
        \end{tikzcd}
    \]
    in ${\FinG}_{/B}$ and ${\FinG}_{/C}$, respectively. The commutative squares
    \begin{equation}\label{EqSpanAdjLiftSquare}
        \begin{tikzcd}
            X \arrow[dr, phantom, "\lrcorner", very near start] \arrow[r, "{\eta}"] \arrow[d, "{\alpha}"'] & B\times_{C}X \arrow[d, "{B\times_{C}\alpha}"'] & B\times_{C}X' \arrow[dr, phantom, "\lrcorner", very near start] \arrow[r, "{\epsilon}"] \arrow[d, "{B\times_{C}\alpha'}"'] & X' \arrow[d, "{\alpha'}"']
            \\
            Y \arrow[r, "{\eta}"] & B\times_{C}Y, & B\times_{C}Y' \arrow[r, "{\epsilon}"] & Y',
        \end{tikzcd}
    \end{equation}
    coming from the naturality of the unit and the counit of the above adjunction, are pullbacks in ${\FinG}_{/B}$ and ${\FinG}_{/C}$, respectively. This can be checked directly after forgetting to $\FinG$, where it is clear. We can now apply Corollary C.21 in \cite{NormsMotivic}. It follows that the functor
    \begin{equation*}
        \Span(B\times_{C}-)\colon \Span({\FinG}_{/C}) \xrightarrow{} \Span({\FinG}_{/B})
    \end{equation*}
    is both left and right adjoint to
    \begin{equation*}
        \Span(f\circ -)\colon \Span({\FinG}_{/B}) \xrightarrow{} \Span({\FinG}_{/C}).
    \end{equation*}
    By 2-functoriality of $\FunProd(-,\Aa)$ we get that
    \begin{equation*}
        f_{!} \simeq \Span(B\times_{C}-)^{*} \simeq f_{*}.
    \end{equation*}
    In particular. the adjunction $f_{!} \dashv f^{*} \dashv f_{*}$ for the $G$-category $\PMack_{G}(\Aa)$ is ambidextrous.
    
    Now, let's consider the map $f_{*}$ for the $G$-category $\PMack_{\leq G/H}(\Aa)$. We start by considering a right adjoint of the functor
    \begin{equation*}
        \FCopC[f\circ -] \colon ({\FinG}_{/B})^{\leq G/H} \xrightarrow{} ({\FinG}_{/C})^{\leq G/H}.
    \end{equation*}
    It is given by the composite
    \begin{equation*}
        R_{\FCopC[f\circ -]}(-)\coloneq j^{\leq G/H}_{B}(B\times_{C}i^{\leq G/H}_{C}(-))\colon ({\FinG}_{/C})^{\leq G/H} \xrightarrow{}({\FinG}_{/B})^{\leq G/H}.
    \end{equation*}
    Indeed, for $U\in ({\FinG}_{/B})^{\leq G/H}$ and for $V \in ({\FinG}_{/C})^{\leq G/H}$ one computes
    \begin{align*}
        \Hom_{({\FinG}_{/C})^{\leq G/H}}(\FCopC[f\circ -](U),V) & \simeq \Hom_{{\FinG}_{/C}}(i^{\leq G/H}_{C}\FCopC[f\circ -](U),i^{\leq G/H}_{C}(V))
        \\
        & \simeq \Hom_{{\FinG}_{/C}}((f\circ -)i^{\leq G/H}_{B}(U),i^{\leq G/H}_{C}(V))
        \\
        & \simeq \Hom_{{\FinG}_{/B}}(i^{\leq G/H}_{B}(U),B\times_{C}i^{\leq G/H}_{C}(V))
        \\
        & \simeq \Hom_{({\FinG}_{/B})^{\leq G/H}}(U,j^{\leq G/H}_{B}(B\times_{C}i^{\leq G/H}_{C}(V)))
        \\
        & = \Hom_{({\FinG}_{/B})^{\leq G/H}}(U,R_{\FCopC[f\circ -]}(V)).
    \end{align*}
    The first equivalence follows from $i^{\leq G/H}_{C}$ being fully faithful. The second equivalence follows from the fact that $i^{\leq G/H}$ is a natural transformation. The third equivalence is the adjunction (\ref{EqPostCompAdjPullb}). Finally, the last equivalence follows from the adjunction $i^{\leq G/H}_{B} \dashv j^{\leq G/H}_{B}$. These equivalences are natural in $U$ and $V$, verifying the adjunction. Furthermore, the unit and the counit are given by
    \begin{equation*}
        U\simeq j^{\leq G/H}_{B}i^{\leq G/H}_{B}(U) \xrightarrow{j^{\leq G/H}_{B}\eta_{i^{\leq G/H}_{B}(U)}} j^{\leq G/H}_{B}(B\times_{C}(f\circ -)(i^{\leq G/H}_{B}(U)))\simeq R_{\FCopC[f\circ -]}\FCopC[f\circ-](U),
    \end{equation*}
    and
    \begin{equation*}
        \FCopC[f\circ-]R_{\FCopC[f\circ -]}(V) \simeq j^{\leq G/H}_{C}(f\circ -)(B\times_{C}i^{\leq G/H}_{C}(V)) \xrightarrow{j^{\leq G/H}_{C}\epsilon_{i^{\leq G/H}_{C}(V)}} j^{\leq G/H}_{C}i^{\leq G/H}_{C}(V)\simeq V.
    \end{equation*}
    To see this, one can check the triangle identities, but we omit the details here. Consider arbitrary morphisms $U\xrightarrow{\beta} V$ and $U'\xrightarrow{\beta'}V'$ in $({\FinG}_{/B})^{\leq G/H}$ and $({\FinG}_{/C})^{\leq G/H}$, respectively. The commutative squares
    \[
        \begin{tikzcd}
            U \arrow[dr, phantom, "\lrcorner", very near start] \arrow[r, "{\eta}"] \arrow[d, "{\beta}"'] & R_{\FCopC[f\circ -]}\FCopC[f\circ-](U) \arrow[d, "{R_{\FCopC[f\circ -]}\FCopC[f\circ-](\beta)}"'] & \FCopC[f\circ-]R_{\FCopC[f\circ -]}(U') \arrow[dr, phantom, "\lrcorner", very near start] \arrow[r, "{\epsilon}"] \arrow[d, "{\FCopC[f\circ-]R_{\FCopC[f\circ -]}(\beta')}"'] & U' \arrow[d, "{\beta'}"']
            \\
            V \arrow[r, "{\eta}"] & R_{\FCopC[f\circ -]}\FCopC[f\circ-](V), & \FCopC[f\circ-]R_{\FCopC[f\circ -]}(V') \arrow[r, "{\epsilon}"] & V',
        \end{tikzcd}
    \]
    coming from the naturality of the unit and the counit of the adjunction $\FCopC[f\circ -] \dashv R_{\FCopC[f\circ -]}$, are pullbacks in $({\FinG}_{/B})^{\leq G/H}$ and $({\FinG}_{/C})^{\leq G/H}$, respectively. To see this, one observes that these commutative squares are equivalently given by considering the squares of (\ref{EqSpanAdjLiftSquare}) with $\alpha = i^{\leq G/H}_{B}(\beta)$ and $\alpha' = i^{\leq G/H}_{C}(\beta')$, and then applying $j^{\leq G/H}_{B}$ and $j^{\leq G/H}_{C}$, respectively. The conclusion follows from $j^{\leq G/H}_{A}$ preserving pullbacks for all $A \in \OG$. Once again, we can apply Corollary C.21 in \cite{NormsMotivic}. We get that the functor
    \begin{equation*}
        \Span(R_{\FCopC[f\circ -]})\colon \Span(({\FinG}_{/C})^{\leq G/H}) \xrightarrow{} \Span(({\FinG}_{/B})^{\leq G/H})
    \end{equation*}
    is both left and right adjoint to
    \begin{equation*}
        \Span(\FCopC[f\circ -]) \colon \Span(({\FinG}_{/B})^{\leq G/H}) \xrightarrow{} \Span(({\FinG}_{/C})^{\leq G/H}).
    \end{equation*}
    Once again, by 2-functoriality of $\FunProd(-,\Aa)$ it follows that
    \begin{equation*}
        f_{!} \simeq \Span(R_{\FCopC[f\circ -]})^{*} \simeq f_{*}.
    \end{equation*}
    In particular, the adjunction $f_{!} \dashv f^{*} \dashv f_{*}$ for the $G$-category $\PMack_{\leq G/H}(\Aa)$ is ambidextrous.
    
    We are ready to show that the natural transformation
    \begin{equation*}
        (j^{\leq G/H})^{*}_{C}\circ f_{*} \xrightarrow{} f_{*}\circ (j^{\leq G/H})^{*}_{B},
    \end{equation*}
    is an equivalence. By 2-functoriality of $\FunProd(-,\Aa)$, this reduces to showing that the mate of the commutative square
    \[
        \begin{tikzcd}[column sep=huge]
            \Span(({\FinG}_{/C})^{\leq G/H})   & \Span({\FinG}_{/C}) \arrow[l, "{\Span(j^{\leq G/H}_{C})}"']
            \\
            \Span(({\FinG}_{/B})^{\leq G/H}) \arrow[u, "{\Span(\FCopC[f\circ -])}"]  & \Span({\FinG}_{/B}) \arrow[l, "{\Span(j^{\leq G/H}_{B})}"'] \arrow[u, "{\Span(f\circ -)}"],
        \end{tikzcd}
    \]
    obtained by passing to the left adjoints of vertical morphisms, is an equivalence. Using the above established adjunctions $\Span(B\times_{C}-) \dashv \Span(f\circ -)$ and $\Span(R_{\FCopC[f\circ -]}) \dashv \Span(\FCopC[f\circ -])$, one can check the statement before applying the $\Span$ functor. Hence, it is enough to show an equivalence
    \begin{equation}\label{EqRightAdjIndentification}
        j^{\leq G/H}_{B}(B\times_{C}-) \simeq R_{\FCopC[f\circ -]}(j^{\leq G/H}_{C}(-)).
    \end{equation}
    Expanding the functor $R_{\FCopC[f\circ -]}$ we get 
    \begin{equation*}
        R_{\FCopC[f\circ -]}(j^{\leq G/H}_{B}(-)) \simeq j^{\leq G/H}_{B}(B\times_{C}i^{\leq G/H}_{C}j^{\leq G/H}_{C}(-)).
    \end{equation*}
    Both sides in (\ref{EqRightAdjIndentification}) preserve finite coproducts, as such, we can check the equivalence on the objects of ${\OG}_{/C}$. Let $D\xrightarrow{} C$ be an object of ${\OG}_{/C}$. There are two options, either $i^{\leq G/H}_{C}j^{\leq G/H}_{C}(D) = D$, or $i^{\leq G/H}_{C}j^{\leq G/H}_{C}(D)=\emptyset$. In the first case, equivalence is immediate. In the second case, equivalence follows from the fact that if $B\times_{C}D$ admits a morphism from $G/g^{\inv}Hg\xrightarrow{}B$, then by postcomposing with the counit, one obtains a morphism from $G/g^{\inv}Hg\xrightarrow{}C$ to $D\xrightarrow{} C$. This can't happen since $i^{\leq G/H}_{C}j^{\leq G/H}_{C}(D)=\emptyset$, but then also $j^{\leq G/H}_{B}(B\times_{C}D) = \emptyset$. This finishes the proof.
\end{proof}

\begin{cons}\label{leqGHToGH}
    In Example \ref{GroupoidGH}, we defined a subfunctor 
    \begin{equation*}
        \{G/H\}^{\simeq}_{/-} \xrightarrow{} {\OG}_{/-}.
    \end{equation*}
    This subfunctor factors as a composite
    \begin{equation*}
        \{G/H\}^{\simeq}_{/-} \xrightarrow{} ({\OG}_{/-})^{\leq G/H} \xrightarrow{} {\OG}_{/-},
    \end{equation*}
    through the functor $({\OG}_{/-})^{\leq G/H}$ from Construction \ref{leqGH}. We obtain a natural transformation
    \begin{equation*}
        r^{G/H} \colon \FCopC[\{G/H\}^{\simeq}_{/-}] \xrightarrow{} ({\FinG}_{/-})^{\leq G/H}.
    \end{equation*}
    This natural transformation is objectwise finite coproduct preserving and fully faithful. We claim that $r^{G/H}$ objectwise preserves pullbacks. Indeed, the functor
    \begin{equation*}
        \{G/H\}^{\simeq}_{/A} \xrightarrow{} ({\OG}_{/A})^{\leq G/H}
    \end{equation*}
    preserves pullbacks for all $A\in \OG$, since $\{G/H\}^{\simeq}_{/A}$ is a groupoid. The claim for $r^{G/H}$ follows from the way pullbacks are calculated in the finite coproduct completions.
    
    Let $\Aa$ be a category admitting finite products. By the functoriality of $\PMack_{\bullet}(\Aa)$, we obtain a $G$-functor
    \begin{equation*}
        (r^{G/H})^{*} \colon \PMack_{\leq G/H}(\Aa) \xrightarrow{} \PMack_{G/H}(\Aa).
    \end{equation*}
\end{cons}

\begin{lem}\label{rGHadj}
    Let $\Aa$ be a presentable, semiadditive category. The $G$-functor 
    \begin{equation*}
        (r^{G/H})^{*} \colon \PMack_{\leq G/H}(\Aa) \xrightarrow{} \PMack_{G/H}(\Aa),
    \end{equation*}
    from Construction \ref{leqGHToGH}, is a $G$-left adjoint $G$-functor.
\end{lem}

\begin{proof}
    The idea of the proof is similar to that of Lemma \ref{leqGhGrpGHCompParametrizedAdj}, thus, we skip details. We check the conditions of Remark \ref{ParAdj}. Let $B$ be an object of $\OG$. The right Kan extension functor $\Ran_{\Span(r^{G/H}_{B})}$ along 
    \begin{equation*}
        \Span(r^{G/H}_{B})\colon \Span(\FCopC[\{G/H\}^{\simeq}_{/B}]) \xrightarrow{} \Span(({\FinG}_{/B})^{\leq G/H}),
    \end{equation*}
    restricts to a right adjoint
    \begin{equation*}
        \Ran_{\Span(r^{G/H}_{B})}\colon \PMack_{G/H}(\Aa)(B) \xrightarrow{} \PMack_{\leq G/H}(\Aa)(B)
    \end{equation*}
    of $(r^{G/H})^{*}_{B}$, which follows from the dual of Lemma 2.20 in \cite{StratCat}.
    
    It remains to verify that for every morphism $f\colon B\xrightarrow{} C$ in $\OG$, the Beck-Chevalley transformation
    \begin{equation*}
        f^{*}\circ \Ran_{\Span(r^{G/H}_{C})} \xrightarrow{} \Ran_{\Span(r^{G/H}_{B})} \circ f^{*}
    \end{equation*}
    is an equivalence. By an argument similar to the one in Lemma \ref{leqGhGrpGHCompParametrizedAdj}, this is equivalent to the Beck-Chevalley transformation
    \begin{equation*}
        f_{!}\circ (r^{G/H})^{*}_{B} \xrightarrow{} (r^{G/H})^{*}_{C}\circ f_{!}
    \end{equation*}
    being an equivalence. Our next step is to explicitly describe $f_{!}$ for $\PMack_{G/H}(\Aa)$. We consider a commutative square
    \begin{equation}\label{EqBeckChevalleyForr}
        \begin{tikzcd}[column sep=huge]
            \FCopC[\{G/H\}^{\simeq}_{/B}] \arrow[r, "{r^{G/H}_{B}}"] \arrow[d, "{\FCopC[f\circ-]}"'] & ({\FinG}_{/B})^{\leq G/H} \arrow[d,"{\FCopC[f\circ-]}"']
            \\
            \FCopC[\{G/H\}^{\simeq}_{/C}] \arrow[r, "{r^{G/H}_{C}}"] & ({\FinG}_{/C})^{\leq G/H}.
        \end{tikzcd}
    \end{equation}
    In the proof of Lemma \ref{leqGhGrpGHCompParametrizedAdj}, we showed the existence of a right adjoint $R_{\FCopC[f\circ -]}$ of the right vertical functor. Both horizontal functors are fully faithful and the composite functor $R_{\FCopC[f\circ -]}r^{G/H}_{C}$ factors through $\FCopC[\{G/H\}^{\simeq}_{/B}]$. To see this, consider an arbitrary $X\in \FCopC[\{G/H\}^{\simeq}_{/C}]$. We can view
    \begin{equation*}
        R_{\FCopC[f\circ -]}r^{G/H}_{C}(X) = j^{\leq G/H}_{B}(B\times_{C}i^{\leq G/H}_{C}r^{G/H}_{C}(X))
    \end{equation*}
    as an object of ${\FinG}_{/B}$. We can decompose it as a finite coproduct of orbits
    \begin{equation*}
        R_{\FCopC[f\circ -]}r^{G/H}_{C}(X)\simeq \coprod_{s\in S}Y_{s}.
    \end{equation*}
    Let $g$ be an element of $G$, after forgetting to $\FinG$, the projection 
    \begin{equation*}
        B\times_{C}i^{\leq G/H}_{C}r^{G/H}_{C}(X)\xrightarrow{} i^{\leq G/H}_{C}r^{G/H}_{C}(X)
    \end{equation*}
    provides us with a morphism from $Y_{s}$ to $G/g^{\inv}Hg$, for each $s\in S$. From the definition of the functor $j^{\leq G/H}_{B}$, it is immediate that for each $s\in S$, there exists $g_{s}\in G$ such that there is a morphism $\gamma_{s}$ from $G/g_{s}^{\inv}Hg_{s}\xrightarrow{} B$ to $Y_{s}$, in ${\FinG}_{/B}$. Hence, since $\OG$ is atomic the morphism $\gamma_{s}$ is an equivalence for each $s\in S$ and therefore $R_{\FCopC[f\circ -]}r^{G/H}_{C}(X)\in \FCopC[\{G/H\}^{\simeq}_{/B}]$.
    
    We have established that a right adjoint of the right vertical morphism in diagram (\ref{EqBeckChevalleyForr}) restricts to a right adjoint to the left vertical morphism. Furthermore, for all $A\in \OG$ the full subcategory $\FCopC[\{G/H\}^{\simeq}_{/A}]$ of $({\FinG}_{/A})^{\leq G/H}$ is closed under pullbacks. Therefore, the restricted adjunction also satisfies conditions of Corollary C.21 in \cite{NormsMotivic}. It follows that we obtain adjunctions on the span categories, which induce adjunctions after applying $\FunProd(-,\Aa)$. We note that passing to the right adjoints of the vertical morphisms in diagram (\ref{EqBeckChevalleyForr}) produces a commutative square. Therefore, the Beck-Chevalley transformation
    \begin{equation*}
        f_{!}\circ (r^{G/H})^{*}_{B} \xrightarrow{} (r^{G/H})^{*}_{C}\circ f_{!}
    \end{equation*}
    is an equivalence. This finishes the proof.
\end{proof}

\begin{defi}
    Let $G/H$ be an object of $\OG$. Let $\Aa$ be a presentable, semiadditive category. We will denote the composite $G$-functor
    \begin{equation*}
        \PMack_{G}(\Aa) \xrightarrow{\phi^{\leq G/H}} \PMack_{\leq G/H}(\Aa) \xrightarrow{(r^{G/H})^{*}} \PMack_{G/H}(\Aa) \xrightarrow{\theta^{G/H}} \Fun(\{G/H\}^{\simeq}_{/-},\Aa)
    \end{equation*}
    by
    \begin{equation*}
        \phi^{G/H} \colon \PMack_{G}(\Aa) \xrightarrow{} \Fun(\{G/H\}^{\simeq}_{/-},\Aa).
    \end{equation*}
\end{defi}

For various subgroups $H\leq G$, the $G$-functors $\phi^{G/H}$ assemble into a $G$-functor
\begin{equation*}
    \PSp_{G}\xrightarrow{} \CorCof\Sp).
\end{equation*}
This functor becomes an equivalence after inverting rational equivalences on both sides. In Section \ref{Section4} we will define a $G$-symmetric monoidal enhancement of the above $G$-functor. To show that its underlying $G$-functor coincides with the one from this section we will use the universal property of the $G$-category of $G$-spectra, which we recall in Theorem \ref{ParametrizedUnivPropPresSp}. The universal property states that the $G$-colimit preserving functors out of the $G$-category of $G$-spectra into a $G$-presentable, $G$-stable $G$-category are determined by their value at the sphere spectrum. We have already shown in Proposition \ref{CorCofGStableGpresentable} that for a presentable, stable category $\Cc$ the $G$-category $\CorCof\Cc)$ is $G$-presentable and $G$-stable. Our next goal is to construct the above $G$-functor from the $G$-category of $G$-spectra, and to show that it preserves $G$-colimits and sends the sphere spectrum to the constant diagram at the sphere spectrum.

\begin{prop}\label{phiGHleftadj}
    Let $G/H$ be an object of $\OG$. Let $\Aa$ be a presentable, semiadditive category. The $G$-functor
    \begin{equation*}
        \phi^{G/H} \colon \PMack_{G}(\Aa) \xrightarrow{} \Fun(\{G/H\}^{\simeq}_{/-},\Aa)
    \end{equation*}
    is a $G$-left adjoint.
\end{prop}
\begin{proof}
    This follows immediately by combining Lemmas \ref{thetaequiv}, \ref{leqGhGrpGHCompParametrizedAdj}, and \ref{rGHadj}.
\end{proof}

\begin{prop}\label{SphereGH}
    Let $G/H$ be an object of $\OG$. The functor
    \begin{equation*}
        \phi^{G/H}_{G/G} \colon \PMack_{G}(\Sp)(G/G) \xrightarrow{} \Fun(\{G/H\}^{\simeq}_{/(G/G)},\Sp)
    \end{equation*}
    sends sphere spectrum $\Sphere\in \Sp_{G}$ to the constant functor
    \begin{equation*}
        \const(\Sphere)\colon \{G/H\}^{\simeq}_{/(G/G)} \xrightarrow{} \Sp,
    \end{equation*}
    taking value $\Sphere\in \Sp$.
\end{prop}
\begin{proof}
    Recall that the right adjoint of
    \begin{equation*}
        \phi^{G/H}_{G/G} \colon \PMack_{G}(\Sp)(G/G) \xrightarrow{} \Fun(\{G/H\}^{\simeq}_{/(G/G)},\Sp)
    \end{equation*}
    is given by the composite
    \begin{align*}
        R_{\phi^{G/H}_{G/G}} \colon \Fun(\{G/H\}^{\simeq}_{/(G/G)},\Sp) &\xrightarrow{(\theta^{G/H}_{G/G})^{\inv}} \PMack_{G/H}(\Sp)(G/G)
        \\
        &\xrightarrow{\Ran_{\Span(r^{G/H}_{G/G})}} \PMack_{\leq G/H}(\Sp)(G/G)
        \\
        & \xrightarrow{(j^{\leq G/H})^{*}_{G/G}} \PMack_{G}(\Sp)(G/G).
    \end{align*}
    These adjunctions are the content of Lemma \ref{thetaequiv}, Lemma \ref{rGHadj}, and Construction \ref{leqGHcomparison}, respectively. Let $\Xx\colon \{G/H\}^{\simeq}_{/(G/G)} \xrightarrow{} \Sp$ be a functor. We have a chain of equivalences natural in $\Xx$
    \begin{align}
        \Hom_{\Fun(\{G/H\}^{\simeq}_{/(G/G)},\Sp)}(\phi^{G/H}_{G/G}(\Sphere),\Xx) &\simeq \Hom_{\PMack_{G}(\Sp)(G/G)}(\Sphere,R_{\phi^{G/H}_{G/G}}(\Xx)) \label{eq5}
        \\ 
        &\simeq \Omega^{\infty}R_{\phi^{G/H}_{G/G}}(\Xx)(G/G) \label{eq6}
        \\
        &\simeq \Omega^{\infty}(\Ran_{\Span(r^{G/H}_{G/G})}(\theta^{G/H}_{G/G})^{\inv}(\Xx))(G/G) \label{eq7}
        \\
        &\simeq \Omega^{\infty}\lim_{\Span(\FCopC[\{G/H\}^{\simeq}_{/(G/G)}])_{(G/G)/}}(\theta^{G/H}_{G/G})^{\inv}(\Xx) \label{eq8}
        \\
        &\simeq \Omega^{\infty}\lim_{\FCopC[\{G/H\}^{\simeq}_{/(G/G)}]^{\op}}(\theta^{G/H}_{G/G})^{\inv}(\Xx) \label{eq9}
        \\
        &\simeq \Omega^{\infty}\lim_{(\{G/H\}_{/(G/G)}^{\simeq})^{\op}}(\theta^{G/H}_{G/G})^{\inv}(\Xx) \label{eq10}
        \\
        &\simeq \Omega^{\infty}\lim_{\{G/H\}_{/(G/G)}^{\simeq}}\Xx \label{eq11}
        \\
        &\simeq \Hom_{\Sp}(\Sphere, \lim_{\{G/H\}_{/(G/G)}^{\simeq}}\Xx) \label{eq12}
        \\
        &\simeq \Hom_{\Fun(\{G/H\}^{\simeq}_{/(G/G)},\Sp)}(\const(\Sphere),\Xx). \label{eq13}
    \end{align}
    Before we explain where each equivalence comes from, we first note that this proves the result. Indeed, by the Yoneda Lemma we obtain an equivalence
    \begin{equation*}
        \phi^{G/H}_{G/G}(\Sphere) \simeq \const(\Sphere).
    \end{equation*}
    We now explain the above equivalences. The equivalence (\ref{eq5}) is the adjunction $\phi^{G/H}_{G/G} \dashv R_{\phi^{G/H}_{G/G}}$. The equivalence (\ref{eq6}) follows from the description of the sphere spectrum in $\Sp_{G}$ as $\Stab(\Hom_{\Span(\FinG)}(G/G,-))$. The equivalence (\ref{eq7}) follows from the description of $R_{\phi^{G/H}_{G/G}}$ as a composite, combined with the equivalence $\ev_{G/G}(j^{\leq G/H})^{*}_{G/G}\simeq \ev_{G/G}$. The equivalence (\ref{eq8}) is the pointwise formula for right Kan extensions. The equivalence (\ref{eq9}) is a consequence of the dual of Lemma 2.28 in \cite{StratCat}. The equivalence (\ref{eq10}) follows from the fact that a limit of a finite product preserving extension of a functor is the same as its limit. The equivalence (\ref{eq11}) follows from the fact that $\theta^{G/H}_{G/G}$ is defined as a precomposition with an appropriate functor, see Construction \ref{GroupoidGHcomparison}. The equivalence (\ref{eq12}) follows from the equivalence $\Sphere \simeq \Sigma^{\infty}_{+}*$. The equivalence (\ref{eq13}) is the definition of a limit. This finishes the proof.  
\end{proof}

\begin{rem}
    Let $\Cc$ be a category. Note that we have an equivalence
    \begin{equation*}
        \CorCof \Cc) \simeq \prod_{(H\leq G)}\Fun(\{G/H\}^{\simeq}_{/-},\Cc).
    \end{equation*}
    Here $\prod_{(H\leq G)}$ denotes the product indexed by the conjugacy classes of subgroups of $G$.
\end{rem}

\begin{defi}\label{phi}
    Let $\Aa$ be a presentable, semiadditive category. We will denote the composite $G$-functor
    \begin{equation*}
        \PMack_{G}(\Aa) \xrightarrow{\triangle} \prod_{(H\leq G)} \PMack_{G}(\Aa) \xrightarrow{\prod_{(H\leq G)}\phi^{G/H}} \prod_{(H\leq G)}\Fun(\{G/H\}^{\simeq}_{/-},\Aa) \simeq \CorCof \Aa)
    \end{equation*}
    by
    \begin{equation*}
        \phi \colon \PMack_{G}(\Aa) \xrightarrow{} \CorCof \Aa).
    \end{equation*}
\end{defi}

\begin{prop}\label{phileftadj}
    Let $\Aa$ be a presentable, semiadditive category. The $G$-functor 
    \begin{equation*}
        \phi \colon \PMack_{G}(\Aa) \xrightarrow{} \CorCof \Aa)
    \end{equation*}
    is a $G$-left adjoint.
\end{prop}
\begin{proof}
    The $G$-right adjoint is given by
    \begin{equation*}
        \CorCof \Aa) \simeq \prod_{(H\leq G)}\Fun(\{G/H\}^{\simeq}_{/-},\Aa) \xrightarrow{\prod_{(H\leq G)}R_{\phi^{G/H}}} \prod_{(H\leq G)} \PMack_{G}(\Aa) \xrightarrow{\bigoplus_{(H\leq G)}} \PMack_{G}(\Aa).
    \end{equation*}
    The $G$-functor $R_{\phi^{G/H}}$ is a $G$-right adjoint from Proposition \ref{phiGHleftadj}. The $G$-adjunction $\triangle \dashv \bigoplus_{(H\leq G)}$ follows from the fact that $\PMack_{G}(\Aa)$ admits $G$-limits indexed by constant $G$-categories.
\end{proof}

\begin{prop}\label{ComparFunBeforeRationalizing}
    The $G$-functor 
    \begin{equation*}
        \phi \colon \PMack_{G}(\Sp) \xrightarrow{} \CorCof \Sp)
    \end{equation*}
    is a $G$-left adjoint. Furthermore, the functor
    \begin{equation*}
        \phi_{G/G}\colon \PMack_{G}(\Sp)(G/G) \xrightarrow{} \CorCof \Sp)(G/G)
    \end{equation*}
    sends the sphere spectrum $\Sphere \in \Sp_{G}$ to the contstant functor
    \begin{equation*}
        \const(\Sphere)\colon \OG^{\simeq} \xrightarrow{} \Sp
    \end{equation*}
    taking value $\Sphere\in \Sp$.
\end{prop}
\begin{proof}
    Combine Definition \ref{phi}, Proposition \ref{phileftadj}, and Proposition \ref{SphereGH}.
\end{proof}

\subsection{Parametrized rational \texorpdfstring{$G$}{G}-spectra}
We are now ready to define the $G$-category of rational $G$-spectra and construct the comparison $G$-functor to our model $\CorCof\Sp_{\RatQ})$. At the end of this subsection, we show that the comparison $G$-functor is an equivalence. In the nonparametrized setting, this is a content of section 4 in \cite{StratCat}.
\begin{rem}\label{RelCat}
    Let $\RelCat$ be a category whose objects are pairs $(\Cc,\Cc_{W})$, where $\Cc_W$ is a wide subcategory of $\Cc$. The functor
    \begin{equation*}
        \Cat \xrightarrow{} \RelCat, \quad \Cc\mapsto (\Cc,\Cc^{\simeq})
    \end{equation*}
    admits a left adjoint
    \begin{equation*}
        \Ll \colon \RelCat \xrightarrow{} \Cat
    \end{equation*}
    assigning to a pair $(\Cc,\Cc_{W})$ the Dwyer-Kan localization $\Ll=\Cc[\Cc_{W}^{\inv}]$. By applying $\Fun(\OGo,-)$ we obtain an adjunction
    \begin{equation*}
        \Ll\colon \Fun(\OGo,\RelCat) \rightleftarrows \Cat_{G}\colon\mkern-3mu (-,(-)^{\simeq}).
    \end{equation*}
\end{rem}

\begin{cons}\label{EquipWithRatEq}
    The $G$-category $\PMack_{G}(\Sp)$ lifts to a functor
    \begin{equation*}
        (\PMack_{G}(\Sp), \text{Rational equivalences})\colon \OGo \xrightarrow{} \RelCat.
    \end{equation*}
    This functor sends $G/H\in G$ to a pair $\PMack_{G}(\Sp)(G/H)$ equipped with a wide subcategory on those morphisms $f\colon \Xx\xrightarrow{}\Yy$ which induce an equivalence
    \begin{equation*}
        f_{\RatQ} \colon \colim (\Xx \xrightarrow{p_{1}} \Xx \xrightarrow{p_{1}p_{2}} \Xx \xrightarrow{p_{1}p_{2}p_{3}} ...) \xrightarrow{\simeq} \colim (\Yy \xrightarrow{p_{1}} \Yy \xrightarrow{p_{1}p_{2}} \Yy \xrightarrow{p_{1}p_{2}p_{3}} ...).
    \end{equation*}
    In particular, the restriction functors preserve colimits, hence preserve these wide subcategories.
\end{cons}

\begin{defi}
    We denote the $G$-category
    \begin{equation*}
        \Ll(\PMack_{G}(\Sp), \text{Rational equivalences})\in\Cat_{G}
    \end{equation*}
    by
    \begin{equation*}
        \PSp_{G,\RatQ},
    \end{equation*}
    and call it the \textit{$G$-category of rational $G$-spectra}.
\end{defi}

Observe that in this case the unit of the adjunction from Remark \ref{RelCat} is a $G$-functor
\begin{equation*}
    \underline{L}_{\RatQ} \colon \PSp_{G} \xrightarrow{} \PSp_{G,\RatQ},
\end{equation*}
levelwise sending rational equivalences to equivalences.

\begin{cons}\label{ParamComparFunct}
    The composite $G$-functor 
    \begin{equation*}
        \PMack_{G}(\Sp) \xrightarrow{\phi} \CorCof \Sp) \xrightarrow{(L_{\RatQ})_*} \CorCof \Sp_{\RatQ})
    \end{equation*}
    sends rational equivalences to equivalences, in each degree. To see this, observe that this composite preserves $G$-colimits, and that we have already inverted objectwise rational equivalences on the right. Hence, by Remark \ref{RelCat}, we obtain a $G$-functor
    \begin{equation*}
        \phi_{\RatQ}\colon \PSp_{G,\RatQ} \xrightarrow{} \CorCof \Sp_{\RatQ}),
    \end{equation*}
    fitting in a commutative diagram
    \[
        \begin{tikzcd}
        \PSp_{G} \arrow[r,"{(L_{\RatQ})_*\phi}"] \arrow[d,"{\underline{L}_{\RatQ}}"'] & \CorCof \Sp_{\RatQ})
        \\
        \PSp_{G,\RatQ} \arrow[ur,"{\phi_{\RatQ}}"']
        \end{tikzcd}
    \]
    of $G$-categories.
\end{cons}

\begin{theo}\label{GEquivalence}
    The $G$-functor
    \begin{equation*}
        \phi_{\RatQ}\colon \PSp_{G,\RatQ} \xrightarrow{} \CorCof \Sp_{\RatQ})
    \end{equation*}
    is an equivalence.
\end{theo}

\begin{proof}
    The adjunction 
    \begin{equation*}
        L_{\RatQ}\colon\Sp \rightleftarrows \Sp_{\RatQ}\colon\mkern-3mu \incl,
    \end{equation*}
    from Remark \ref{RatvsMackRat}, induces $G$-adjunction
    \begin{equation*}
        (L_{\RatQ})_{*}\colon\PMack_{G}(\Sp) \rightleftarrows \PMack_{G}(\Sp_{\RatQ})\colon\mkern-3mu \incl_{*}.
    \end{equation*}
    The $G$-functor $(L_{\RatQ})_{*}$ inverts rational equivalences in each degree. Hence, we obtain a $G$-functor $\PSp_{G,\RatQ} \xrightarrow{} \PMack_{G}(\Sp_{\RatQ})$, making the diagram
    \[
        \begin{tikzcd}
        \PSp_{G} \arrow[r,"{(L_{\RatQ})_{*}}"] \arrow[d,"{\underline{L}_{\RatQ}}"'] & \PMack_{G}(\Sp_{\RatQ})
        \\
        \PSp_{G,\RatQ}, \arrow[ur]
        \end{tikzcd}
    \]
    of $G$-categories commute. It follows directly from Remark \ref{RatvsMackRat}, that the $G$-functor
    \begin{equation*}
        \PSp_{G,\RatQ} \xrightarrow{} \PMack_{G}(\Sp_{\RatQ})
    \end{equation*}
    is an equivalence. Furthermore, we claim that the following diagram
    \[
        \begin{tikzcd}
        \PSp_{G} \arrow[r,"{(L_{\RatQ})_{*}\phi}"] \arrow[d,"{(L_{\RatQ})_{*}}"'] & \CorCof \Sp_{\RatQ})
        \\
        \PMack_{G}(\Sp_{\RatQ}) \arrow[ur,"{\phi}"']
        \end{tikzcd}
    \]
    of $G$-categories commutes. It is immediate from the definition of the $G$-functor $\phi$ that it is enough to show that, for each $G/H\in \OG$, the diagram
    \[
        \begin{tikzcd}
            \PMack_{G}(\Sp) \arrow[r,"{\phi^{G/H}}"] \arrow[d,"{(L_{\RatQ})_{*}}"'] & \Fun(\{G/H\}^{\simeq}_{/-}\Sp) \arrow[d,"{(L_{\RatQ})_{*}}"']
            \\
            \PMack_{G}(\Sp_{\RatQ}) \arrow[r,"{\phi^{G/H}}"] & \Fun(\{G/H\}^{\simeq}_{/-},\Sp_{\RatQ})
        \end{tikzcd}
    \]
    commutes. To see this, observe that we have commutative diagrams
    \[
        \begin{tikzcd}
        \PMack_{\leq G/H}(\Sp) \arrow[d,"{(L_{\RatQ})_{*}}"'] \arrow[r,"{(r^{G/H})^{*}}"] & \PMack_{G/H}(\Sp) \arrow[d,"{(L_{\RatQ})_{*}}"'] \arrow[r,"{\theta^{G/H}}"] & \Fun(\{G/H\}^{\simeq}_{/-},\Sp) \arrow[d,"{(L_{\RatQ})_{*}}"']
        \\
        \PMack_{\leq G/H}(\Sp_{\RatQ}) \arrow[r,"{(r^{G/H})^{*}}"] & \PMack_{G/H}(\Sp_{\RatQ}) \arrow[r,"{\theta^{G/H}}"] & \Fun(\{G/H\}^{\simeq}_{/-},\Sp_{\RatQ})
        \end{tikzcd}
    \]
    and
    \[
        \begin{tikzcd}
            \PMack_{G}(\Sp) & \PMack_{\leq G/H}(\Sp) \arrow[l,"{(j^{\leq G/H})^{*}}"'] 
            \\
            \PMack_{G}(\Sp_{\RatQ}) \arrow[u,"{\incl_{*}}"] & \PMack_{\leq G/H}(\Sp_{\RatQ}). \arrow[u,"{\incl_{*}}"] \arrow[l,"{(j^{\leq G/H})^{*}}"']
        \end{tikzcd}
    \]
    Passing to the $G$-right adjoints in the above diagram we obtain a commutative diagram, which combined with the previous commutative diagram, proves the claim. Therefore, the statement reduces to the $G$-functor
    \begin{equation*}
        \phi\colon \PMack_{G}(\Sp_{\RatQ}) \xrightarrow{} \CorCof \Sp_{\RatQ})
    \end{equation*}
    being an equivalence. This can be checked objectwise, where it is a content of Theorem 4.2 in \cite{StratCat}. We note that the construction of the comparison functor from Definition 2.31 in \cite{StratCat}, does not coincide with the construction of the functor 
    \begin{equation*}
        \phi_{G/K}^{G/H}\colon \PMack_{G}(\Aa)(G/K) \xrightarrow{} \CorCof \Aa)(G/K).
    \end{equation*}
    Let $H\leq K \leq G$ be subgroups, and let $\Aa$ be a presentable, semiadditive category. The difference comes from us defining
    \begin{equation*}
        \phi_{G/K}^{\leq G/H} \colon \PMack_{G}(\Aa)(G/K) \xrightarrow{} \PMack_{\leq G/H}(\Aa)(G/K)
    \end{equation*}
    as a left Kan extension along $\Span(j^{\leq G/H}_{G/K})$, where $j^{\leq G/H}_{G/K}$ is a right adjoint to the free coproduct completion of the inclusion of those $K$-orbits which admit a morphism from one of the $K$-orbits $K/g^{\inv}Hg$, for $g\in G$. The construction of the functor from Definition 2.31 in \cite{StratCat}, instead, considers the case of those $K$-orbits that admit a morphism from $K/H$, without any conjugation. Nonetheless, the values of the left Kan extensions at $K/H$ coincide, and it follows that the comparison functors are equivalent. This finishes the proof.
\end{proof}

\section{A model for the \texorpdfstring{$G$}{G}-symmetric monoidal \texorpdfstring{$G$}{G}-category of rational \texorpdfstring{$G$}{G}-spectra}\label{Section4}

In this section, we show that, whenever $\Cc$ is a symmetric monoidal category, the $G$-category $\CorCof\Cc)$ naturally carries the structure of a $G$-symmetric monoidal $G$-category. We then show that the comparison $G$-functor
\begin{equation*}
        \phi_{\RatQ}\colon \PSp_{G,\RatQ} \xrightarrow{} \CorCof \Sp_{\RatQ})
\end{equation*}
refines to a $G$-symmetric monoidal $G$-functor, which is automatically an equivalence.

\subsection{Algebraic model for \texorpdfstring{$G$}{G}-symmetric monoidal \texorpdfstring{$G$}{G}-category of rational \texorpdfstring{$G$}{G}-spectra}

We begin by lifting the functor $\CorCof-)$ to a functor
\begin{equation*}
    \CorCofMul-)\colon\SMCat\xrightarrow{}\SMCatG,
\end{equation*}
and establishing its properties.

\begin{cons}\label{CorCofMulConst}
    We consider the slice functor
    \begin{equation*}
        {\FinG}_{/-} \colon \FinG^{\op} \xrightarrow{} \Cat
    \end{equation*}
    with the pullback functoriality. This functor is a cartesian unstraightening of the cartesian fibration 
    \begin{equation*}
        \ev_{1}\colon \Fun([1],\FinG)\xrightarrow{} \FinG.
    \end{equation*}
    Let $f\colon X\xrightarrow{} Y$ be a morphism in $\FinG$. The pullback functor $X\times_{Y}-\colon {\FinG}_{/Y} \xrightarrow{} {\FinG}_{/X}$ admits a left adjoint $f\circ- \colon {\FinG}_{/X} \xrightarrow{} {\FinG}_{/Y}$. Furthermore, given a pullback diagram
    \[
        \begin{tikzcd}
            X' \arrow[dr, phantom, "\lrcorner", very near start] \arrow[r, "{f'}"] \arrow[d, "{\alpha}"'] & Y' \arrow[d, "{\beta}"']
            \\
            X \arrow[r, "{f}"] & Y
        \end{tikzcd}
    \]
    in $\FinG$, the mate natural transformation $(f'\circ-)\circ(X'\times_{X}-)\xrightarrow{} (Y'\times_{Y}-)\circ (f\circ-)$ is an equivalence. Therefore, by Barwick's unfurling construction, \cite[Example 3.4]{TwoVariable}, we obtain a functor
    \begin{equation*}
        {\FinG}_{/-} \colon \Span(\FinG) \xrightarrow{} \Cat,
    \end{equation*}
    with the pullback functoriality on the backwards morphisms, and the postcomposition functoriality on the forwards morphisms. This functor factors through the subcategory $\Cat^{\amalg}$ of categories admitting finite coproducts and finite coproduct preserving functors. We postcompose with the fully-faithful inclusion $\Cat^{\amalg}\xrightarrow{} \SMCat$, to obtain a functor
    \begin{equation*}
        ({\FinG}_{/-},\amalg) \colon \Span(\FinG) \xrightarrow{} \SMCat.
    \end{equation*}
    Finally, we postcompose with the underlying groupoid functor $(-)^{\simeq}\colon\SMCat\xrightarrow{}\SMCat$ and take the opposite of this functor. We denote the resulting functor by
    \begin{equation*}
        (({\FinG}_{/-})^{\simeq},\amalg) \colon \Span(\FinG)\simeq\Span(\FinG)^{\op} \xrightarrow{} (\SMCat)^{\op}.
    \end{equation*}
\end{cons}

\begin{defi}\label{CorCofMul}
    We denote by $\CorCofMul \bullet)$ the composite functor
    \begin{equation*}
        \PSMCat \xrightarrow{\FunTen(-,\bullet)} \Fun((\PSMCat)^{\op},\Cat) \xrightarrow{-\circ (({\FinG}_{/-})^{\simeq},\amalg)} \Fun(\Span(\FinG),\Cat) = \PSMCatG,
    \end{equation*}
    where the second functor is the precomposition with the functor from Construction \ref{CorCofMulConst}. This functor sends a symmetric monoidal precategory $\Ee$ to the functor
    \begin{equation*}
        \CorCofMul\Ee)=\FunTen((({\FinG}_{/-})^{\simeq},\amalg),\Ee).
    \end{equation*}
\end{defi}

We now show that for a symmetric monoidal category $(\Cc,\otimes)$, the $G$-symmetric monoidal $G$-precategory $\CorCofMul(\Cc,\otimes))$ is in fact a $G$-symmetric monoidal $G$-category, and that this construction equips the $G$-category $\CorCof\Cc)$ with a $G$-symmetric monoidal structure.

\begin{rem}\label{CMon}
    In this section, we will need the notion of commutative monoids in a category $\Cc$ admitting finite products. One can define the category of commutative monoids in $\Cc$ by
    \begin{equation*}
        \CMon(\Cc)\coloneq \FunProd(\Span(\FCopC),\Cc).
    \end{equation*}
    By \cite[Proposition C.1]{NormsMotivic}, this definition agrees with the definition of commutative monoids as a full subcategory of the functor category $\Fun(\FCopC_{*},\Cc)$ spanned by objects satisfying Segal condition. The category of commutative monoids $\CMon(\Cc)$ is semiadditive by \cite[Corollary 2.5]{MultInfLoopSpc}. The forgetful functor
    \begin{equation*}
        \fgt=\ev_{*}\colon \CMon(\Cc)=\FunProd(\Span(\FCopC),\Cc) \xrightarrow{} \Cc
    \end{equation*}
    induces equivalences 
    \begin{equation*}
        \CMon(\ev_{*}), \ev_* \colon \CMon(\CMon(\Cc))\xrightarrow{} \CMon(\Cc)
    \end{equation*}
    by the same Corollary. In particular, we have that
    \begin{equation*}
        \SMCat \simeq \CMon(\Cat).
    \end{equation*}
\end{rem}

\begin{rem}\label{FinEnvOG}
    Let $\Ss$ be a groupoid. The symmetric monoidal envelope of $\Ss$ coincides with the free commutative monoid $\Free(\Ss)\in \CMon(\Spc)$. In particular, it is given by the formula
    \begin{equation*}
        \coprod_{n\geq 0}\Ss^{\times n}/\Sigma_{n}.
    \end{equation*}
    Using the above formula, we see that
    \begin{equation*}
        \Env(\OG^{\simeq})\simeq (\FCopC[\OG^{\simeq}]^{\simeq},\amalg) \simeq (\FCopC[\OG]^{\simeq},\amalg) \simeq (\FinG^{\simeq},\amalg).
    \end{equation*}
\end{rem}

\begin{lem}\label{SMMontoGSMMon}
    The functor
    \begin{equation*}
        \CorCofMul -) \colon \PSMCat \xrightarrow{} \PSMCatG,
    \end{equation*}
    from Definition \ref{CorCofMul}, restricts to a functor
    \begin{equation*}
        \CorCofMul -) \colon \SMCat \xrightarrow{} \SMCatG,
    \end{equation*}
    from the category of symmetric monoidal categories to the category of $G$-symmetric monoidal $G$-categories.
\end{lem}

\begin{proof}
    Let $(\Cc,\otimes)$ be a symmetric monoidal category. We need to show that the functor
    \begin{equation*}
        \CorCofMul(\Cc,\otimes))=\FunTen((({\FinG}_{/-})^{\simeq},\amalg),(\Cc,\otimes)) \colon \Span(\FinG) \xrightarrow{} \Cat
    \end{equation*}
    preserves finite products. Let $X_{i}$, $i\in I$ be a finite collection of finite $G$-sets. The product in $\Span(\FinG)$ is given by the coproduct $\coprod_{i\in I}X_{i}$ of $G$-sets. The projection morphisms are given by the spans
    \[
        \begin{tikzcd}[sep=tiny]
            &X_{j} \arrow[dl,"{\iota_{j}}"'] \arrow[dr, "{id_{X_{j}}}"]
            \\
            \coprod_{i\in I}X_{i} &&X_{j},
        \end{tikzcd}
    \]
    where the backwards morphism is the inclusion into the coproduct. We need to show that the morphism
    \begin{equation*}
        ((\iota_{i},id_{X_{i}})_{*})_{i\in I}\colon \CorCofMul(\Cc,\otimes))(\coprod_{i\in I}X_{i}) \xrightarrow{} \prod_{i\in I}\CorCofMul(\Cc,\otimes))(X_{i})
    \end{equation*}
    is an equivalence. By the construction of the functor $\CorCofMul-)$, we have that this morphism is induced by precomposing with functors
    \begin{equation*}
        \iota_{j}\circ-\colon(({\FinG}_{/X_{j}})^{\simeq},\amalg)\xrightarrow{}(({\FinG}_{/\coprod_{i\in I}X_{i}})^{\simeq},\amalg).
    \end{equation*}
    Observe that the finite coproduct preserving functors
    \begin{equation*}
        X_{j}\times_{\coprod_{i\in I}X_{i}}-\colon({\FinG}_{/\coprod_{i\in I}X_{i}})^{\simeq}\xrightarrow{}({\FinG}_{/X_{j}})^{\simeq}
    \end{equation*}
    induce an equivalence
    \begin{equation*}
        ({\FinG}_{/\coprod_{i\in I}X_{i}})^{\simeq} \simeq \prod_{i\in I}({\FinG}_{/X_{i}})^{\simeq},
    \end{equation*}
    fitting into the commutative diagram
    \begin{equation}\label{EqProdCoprodIdentification}
        \begin{tikzcd}
            &({\FinG}_{/X_{j}})^{\simeq} \arrow[dl,"{\iota_{j}\circ-}"'] \arrow[dr]
            \\
            ({\FinG}_{/\coprod_{i\in I}X_{i}})^{\simeq} \arrow[rr,"{\simeq}"] &&\prod_{i\in I}({\FinG}_{/X_{i}})^{\simeq}.
        \end{tikzcd}
    \end{equation}
    Where the top right functor is the finite coproduct preserving functor
    \begin{equation*}
        ({\FinG}_{/X_{j}})^{\simeq}\xrightarrow{} \prod_{i\in I}({\FinG}_{/X_{i}})^{\simeq}, \quad (f\colon Y \xrightarrow{} X_{j}) \mapsto (\emptyset,...,f\colon Y \xrightarrow{} X_{j},...,\emptyset).
    \end{equation*}
    Recall that $\SMCat$ is semiadditive, with the direct sum given by the underlying product of the categories, see \cite[Proposition 2.3]{MultInfLoopSpc}. One can show that the above functors exhibit the target as a coproduct in the category of symmetric monoidal categories. We will first show the Lemma for a finite collection $A_j$, $j\in J$, of transitive $G$-sets. This follows from the commutative diagram
    \[
        \begin{tikzcd}[column sep =huge]
        \CorCofMul(\Cc,\otimes))(\coprod_{j\in J}A_{j}) \arrow[r] \arrow[d,"="'] &\prod_{j\in J}\CorCofMul(\Cc,\otimes))(A_{j}) \arrow[dd,"{=}"']
        \\
        \FunTen((({\FinG}_{/\coprod_{j\in J}A_{j}})^{\simeq},\amalg),(\Cc,\otimes))
        \\
        \FunTen(\prod_{j\in J}(({\FinG}_{/A_{j}})^{\simeq},\amalg),(\Cc,\otimes)) \arrow[r] \arrow[u,"{\simeq}"] \arrow[d,"{\simeq}"'] &\prod_{j\in J}\FunTen((({\FinG}_{/A_{j}})^{\simeq},\amalg),(\Cc,\otimes)) \arrow[d,"{\simeq}"']
        \\
        \Fun(\coprod_{j\in J}({\OG}_{/A_{j}})^{\simeq},\Cc) \arrow[r,"{\simeq}"] &\prod_{j\in J}\Fun(({\OG}_{/A_{j}})^{\simeq},\Cc),
        \end{tikzcd}
    \]
    where the top square commutes by commutative diagram (\ref{EqProdCoprodIdentification}). The bottom square commutes using the adjunction $\Env \dashv \fgt$, combined with Remark \ref{FinEnvOG}. The claim for a general finite collection $X_{i}$, $i\in I$, of finite $G$-sets then immediately follows from the commutative diagram
    \[
        \begin{tikzcd}[column sep =huge]
            \CorCofMul(\Cc,\otimes))(\coprod_{i\in I}X_{i}) \arrow[d,"{\simeq}"'] \arrow[r] &\prod_{i\in I}\CorCofMul(\Cc,\otimes))(X_{i}) \arrow[d,"{\simeq}"']
            \\
            \CorCofMul(\Cc,\otimes))(\coprod_{i\in I}\coprod_{j\in J_{i}}A_{j}^{i})  \arrow[r,"{\simeq}"] &\prod_{i\in I}\prod_{j\in J_{i}}\CorCofMul(\Cc,\otimes))(A_{j}^{i}),
        \end{tikzcd}
    \]
    where the $A_{j}^{i}$ are cosets coming from the coset decompositions $X_i\simeq\coprod_{j\in J_{i}}A_{j}^{i}$. The bottom horizontal and the right vertical morphisms are equivalences from the discussion beforehand.
\end{proof}

\begin{cons}\label{CnsTauComparison}
    The functors
    \begin{equation*}
        ({\OG}_{/A})^{\simeq} \xrightarrow{} ({\FinG}_{/A})^{\simeq}
    \end{equation*}
    for various $A\in\OG$ assemble into a natural transformation
    \begin{equation*}
        \incl\colon ({\OG}_{/-})^{\simeq} \xrightarrow{} ({\FinG}_{/-})^{\simeq}|_{\OGo},
    \end{equation*}
    from the functor that is opposite to the one of Construction \ref{CorOG} to the restriction of the functor
    \begin{equation*}
        ({\FinG}_{/-})^{\simeq} \colon \Span(\FinG) \xrightarrow{} \Cat^{\op},
    \end{equation*}
    from Construction \ref{CorCofMulConst}, to $\OGo$. We define a natural transformation
    \[
        \begin{tikzcd}[column sep=huge]
            \PSMCat \arrow[r,"{\CorCofMul-)}"]  \arrow[d,"{\fgt}"'] &\PSMCatG \arrow[d,"{\fgt}"'] \arrow[Rightarrow, dl, "{\tau}"']
            \\
            \Cat \arrow[r,"{\CorCof-)}"'] &\Cat_{G},
        \end{tikzcd}
    \]
    using the following diagram
    \[
        \begin{tikzcd}[sep=large]
            \PSMCat \arrow[rr,"{\FunTen(-,\bullet)}"] \arrow[d,"{\fgt}"'] & \arrow[d,Rightarrow, "{\fgt}"'] &\Fun((\PSMCat)^{\op},\Cat) \arrow[rr,"{-\circ (({\FinG}_{/-})^{\simeq},\amalg)|_{\OGo}}"] && \Cat_{G}
            \\
            \Cat \arrow[rr,"{\Fun(-,\bullet)}"'] &\phantom{-} &\Fun(\Cat^{\op},\Cat) \arrow[u,"{-\circ \fgt}"] \arrow[urr,"{-\circ ({\FinG}_{/-})^{\simeq}|_{\OGo}}" description, ""{name =A, below}] \arrow[urr, out=0, in=-100, "{({\OG}_{/-})^{\simeq}.}"', ""{name =B}]
            \arrow[Rightarrow, from =A, to =B, shorten =1mm, "{-\circ \incl}"']
        \end{tikzcd}
    \]
    The 2-arrow in the left square is induced by the natural transformation
    \begin{equation*}
        \fgt \colon \FunTen(-,-) \xrightarrow{} \Fun(-,-),
    \end{equation*}
    coming from 2-functoriality of the forgetful functor $\fgt \colon \PSMCat \xrightarrow{} \Cat$. The commutativity of the top right triangle follows from the equivalence
    \begin{equation*}
        \fgt((({\FinG}_{/-})^{\simeq},\amalg)|_{\OGo}) \simeq ({\FinG}_{/-})^{\simeq}|_{\OGo}.
    \end{equation*}
    The right bottom 2-arrow is given by precomposing with the inclusion $\incl\colon ({\OG}_{/-})^{\simeq} \xrightarrow{} ({\FinG}_{/-})^{\simeq}|_{\OGo}$. The top composite functor is equivalent to the functor
    \begin{equation*}
        \fgt\circ \CorCofMul-)\colon \PSMCat \xrightarrow{} \Cat_{G},
    \end{equation*}
    and the bottom composite functor is equivalent to the functor
    \begin{equation*}
        \CorCof-)\circ \fgt \colon \PSMCatG \xrightarrow{} \Cat_{G}.
    \end{equation*}
    We obtain the natural transformation $\tau\colon \fgt\circ \CorCofMul-)\xrightarrow{}\CorCof-)\circ \fgt$ as the composite in the above diagram.
    
    We spell out the details. Let $\Ee\colon\Span(\FinG)\xrightarrow{}\Cat$ be a functor. The natural transformation $\tau$ is given by the composite
    \begin{align*}
        \tau_{\Ee}\colon \CorCofMul\Ee)|_{\OGo}=\FunTen((({\FinG}_{/-})^{\simeq},\amalg)|_{\OGo},\Ee) &\xrightarrow{\fgt} \Fun(({\FinG}_{/-})^{\simeq}|_{\OGo},\fgt(\Ee)) 
        \\
        &\xrightarrow{-\circ\incl} \Fun(({\OG}_{/-})^{\simeq},\fgt(\Ee))=\CorCof\fgt(\Ee)),
    \end{align*}
    where the second functor is given by precomposition.
\end{cons}

\begin{prop}\label{GSMlifting}
    The natural transformation
    \begin{equation*}
        \tau \colon  \fgt\circ \CorCofMul-)\xrightarrow{} \CorCof-)\circ \fgt
    \end{equation*}
    from Construction \ref{CnsTauComparison} makes the diagram
    \[
        \begin{tikzcd}[column sep=huge]
            \SMCat \arrow[r,"{\CorCofMul-)}"]  \arrow[d,"{\fgt}"'] &\SMCatG \arrow[d,"{\fgt}"']
            \\
            \Cat \arrow[r,"{\CorCof-)}"] &\Cat_{G}
        \end{tikzcd}
    \]
    commute.
\end{prop}
\begin{proof}
    It is enough to show that the natural transformation $\tau$ is an equivalence componentwise. Let $(\Cc,\otimes)$ be a $G$-symmetric monoidal $G$-category, and let $A$ be an object of $\OG$. We need to show that the functor
    \begin{align*}
        \tau_{\Ee}(A)\colon \FunTen((({\FinG}_{/A})^{\simeq},\amalg),(\Cc,\otimes)) &\xrightarrow{\fgt} \Fun(({\FinG}_{/A})^{\simeq},\fgt(\Cc,\otimes)) 
        \\
        &\xrightarrow{-\circ\incl} \Fun(({\OG}_{/A})^{\simeq},\fgt(\Cc,\otimes))
    \end{align*}
    is an equivalence. This immediately follows from the adjunction $\Env \dashv \fgt$ combined with Remark \ref{FinEnvOG}.
\end{proof}

As mentioned above, once we define a $G$-symmetric monoidal comparison $G$-functor, we will use the universal property of the $G$-category of $G$-spectra to show that its underlying $G$-functor coincides with the one from Section \ref{Section3}. A $G$-symmetric monoidal $G$-functor out of the $G$-category of $G$-spectra sends the sphere spectrum to the $G$-symmetric monoidal unit of the target $G$-symmetric monoidal $G$-category. We now show that, for each $G/H\in\OG$, the $G$-symmetric monoidal $G$-category $\CorCofMul(\Cc,\otimes))$ is given by the pointwise symmetric monoidal structure induced by the one on $\Cc$. In particular, the $G$-symmetric monoidal unit is given by the constant diagram at the unit of $\Cc$.

\begin{cons}
    The functor $G/H\colon*\xrightarrow{}\OG$ induces a finite product preserving functor
    \begin{equation*}
        l_{G/H}\colon \Span(\FCopC)\xrightarrow{} \Span(\FinG), \quad A\mapsto \coprod_{A}G/H.
    \end{equation*}
\end{cons}

\begin{lem}\label{Componentwise}
    The composite
    \begin{equation*}
        \SMCat \xrightarrow{\CorCofMul-)} \SMCatG \xrightarrow{(l_{G/H})^{*}} \SMCat
    \end{equation*}
    is equivalent to the functor
    \begin{equation*}
        \Fun(\Oo_{H}^{\simeq},-)_{*}\colon \SMCat \xrightarrow{} \SMCat.
    \end{equation*}
    In particular, for a symmetric monoidal category $(\Cc,\otimes)$, the symmetric monoidal unit of the symmetric monoidal category $\FunTen((({\FinG}_{/l_{G/H}(-)})^{\simeq},\amalg),(\Cc,\otimes))$ is given by the constant diagram at the symmetric monoidal unit of $(\Cc,\otimes)$.
\end{lem}

\begin{proof}
    In this proof, we denote $\SMCat$ by $\CMon(\Cat)$, see Remark \ref{CMon}. We consider a commutative diagram
    \[
        \begin{tikzcd}[column sep =2.5cm]
            \CMon(\Cat) \arrow[r,"{\CorCofMul-)}"] &\SMCatG \arrow[r,"{(l_{G/H})^{*}}"] &\CMon(\Cat)
            \\
            \CMon(\CMon(\Cat)) \arrow[r,"{\CMon(\CorCofMul-))}"] \arrow[u,"{\ev_{*}}", "{\simeq}"'] \arrow[d,"{\CMon(\ev_{*})}"', "{\simeq}"] &\CMon(\SMCatG) \arrow[r,"{\CMon((l_{G/H})^{*})}"] \arrow[u,"{\ev_{*}}"] \arrow[d,"{\CMon(\fgt)}"'] &\CMon(\CMon(\Cat)) \arrow[u,"{\ev_{*}}", "{\simeq}"'] \arrow[d,"{\CMon(\ev_{*})}"', "{\simeq}"]
            \\
            \CMon(\Cat) \arrow[r,"{\CMon(\CorCof-))}"] &\CMon(\Cat_{G}) \arrow[r,"{\CMon(\ev_{G/H})}"] &\CMon(\Cat),
        \end{tikzcd}
    \]
    where the bottom left square commutes by Lemma \ref{GSMlifting}, and other squares commute essentially by definition. We observe that the left and right composite functors are equivalent to the identity functors, see Remark \ref{CMon}. It is immediate that the bottom composite functor in the above diagram is equivalent to the functor
    \begin{equation*}
        \Fun(\Oo_{H}^{\simeq},-)_{*}\colon \SMCat \xrightarrow{} \SMCat,
    \end{equation*}
    finishing the proof.
\end{proof}

Next, we show that the functor
\begin{equation*}
    \CorCofMul-)\colon \SMCat \xrightarrow{} \SMCatG
\end{equation*}
is right adjoint to
\begin{equation*}
    l_{G/G}^{*}\colon \SMCatG \xrightarrow{} \SMCat.
\end{equation*}
We then define the $G$-symmetric monoidal comparison $G$-functor using this adjunction. In addition, we record the counit of this adjunction, which provides sufficient control over the $G$-symmetric monoidal comparison $G$-functor.

\begin{prop}\label{GSymmMonAdj}
    The functor 
    \begin{equation*}
        l_{G/G}^{*} \colon \PSMCatG \xrightarrow{} \PSMCat
    \end{equation*}
    is a 2-left adjoint of the functor
    \begin{equation*}
        \CorCofMul-)\colon \PSMCat \xrightarrow{} \PSMCatG.
    \end{equation*}
    This 2-adjunction restricts to a 2-adjunction
    \begin{equation*}
        l_{G/G}^{*}\colon \SMCatG \rightleftarrows \SMCat \colon \mkern-3mu \CorCofMul-)
    \end{equation*}
    between the categories of $G$-symmetric monoidal $G$-categories and symmetric monoidal categories. Furthermore, on the underlying categories, the counit of the adjunction is given for symmetric monoidal categories by evaluating at $G/G$. More precisely, let $(\Cc,\otimes)$ be a symmetric monoidal category then
    \begin{equation*}
        \fgt(\varepsilon)\simeq \ev_{G/G}\colon \Fun(\OG^{\simeq},\Cc) \xrightarrow{} \Cc.
    \end{equation*}
\end{prop}

\begin{proof}
    It is immediate that the functor $l_{G/G}^{*}\colon\PSMCatG \xrightarrow{} \PSMCat$ is a $\Cat$-linear, and that it admits a right adjoint $\hat{R}\colon \PSMCat \xrightarrow{} \PSMCatG$ given by a right Kan extension. It follows from Remark \ref{2-Cats}, that the adjunction $l_{G/G}^{*}\dashv \hat{R}$ lifts to a 2-adjunction. Let $\Ee$ be a $G$-symmetric monoidal $G$-precategory, and let $X$ be a finite $G$-set. We have a sequence of equivalences
    \begin{equation*}
        \hat{R}(\Ee)(X)\simeq \FunTenG(\Hom_{\Span(\FinG)}(X,-),\hat{R}(\Ee))\simeq \FunTen(\Hom_{\Span(\FinG)}(X,l_{G/G}(-)),\Ee),
    \end{equation*}
    natural in both $\Ee$ and $X$. The first equivalence is a consequence of Lemma \ref{Yoneda}, and the second equivalence follows from the adjunction. We claim that the $\hat{R}$ is equivalent to $\CorCofMul-)$. The claim reduces to showing that the functor
    \begin{equation*}
        \Hom_{\Span(\FinG)}(\bullet,l_{G/G}(-)) \colon \Span(\FinG) \xrightarrow{} \SMCat, \quad X\mapsto \Hom_{\Span(\FinG)}(X,l_{G/G}(-))
    \end{equation*}
    is equivalent to the functor
    \begin{equation*}
        (({\FinG}_{/-})^{\simeq},\amalg) \colon \Span(\FinG) \xrightarrow{} \SMCat,
    \end{equation*}
    from Construction \ref{CorCofMulConst}. One way to see this is to use the description of the cocartesian symmetric monoidal structure \cite[Proposition 3.1.1]{HomotCommutRingBisp}. Alternatively, observe that $\Span(\FinG)$ is naturally a 2-category, with 2-morphisms being commutative diagrams of the form
    \[
        \begin{tikzcd}[sep=tiny]
            &Z \arrow[dl, "{b}"'] \arrow[dr, "{f}"] \arrow[dd]
            \\
            X &&Y
            \\
            &Z'. \arrow[ul, "{b'}"] \arrow[ur, "{f'}"']
        \end{tikzcd}
    \]
    Let
    \begin{equation*}
        \HOM_{\Span(\FinG)}(-,-)\colon \Span(\FinG)^{\op}\times \Span(\FinG) \xrightarrow{} \Cat
    \end{equation*}
    denote the functor coming from the $\Cat$-enrichment. We get an induced functor
    \begin{equation*}
        \HOM_{\Span(\FinG)}(\bullet,l_{G/G}(-)) \colon \Span(\FinG) \xrightarrow{} \SMCat, \quad X\mapsto \HOM_{\Span(\FinG)}(X,l_{G/G}(-)).
    \end{equation*}
    Using the description
    \begin{equation*}
        \HOM_{\Span(\FinG)}(X,Y)\simeq {\FinG}_{/X\times Y},
    \end{equation*}
    it is easy to check that after postcomposing with the underlying category functor $\fgt\colon \SMCat \xrightarrow{} \Cat$, we get an equivalence
    \begin{equation*}
        \HOM_{\Span(\FinG)}(\bullet,G/G)\simeq {\FinG}_{/-}\colon \Span(\FinG) \xrightarrow{} \Cat.
    \end{equation*}
    Then using the definition of the cocartesian symmetric monoidal structure \cite[Definition 2.4.0.1]{HA}, one deduces that this equivalence lifts to an equivalence
    \begin{equation*}
        \HOM_{\Span(\FinG)}(\bullet,l_{G/G}(-))\simeq({\FinG}_{/-},\amalg)\colon \Span(\FinG) \xrightarrow{} \SMCat.
    \end{equation*}
    We obtain the desired equivalence by applying the underlying groupoid functor to the above equivalence. Hence, we get that $\hat{R}\simeq \CorCofMul-)$ and therefore we have a 2-adjunction
    \begin{equation*}
        l_{G/G}^{*}\colon \PSMCatG \rightleftarrows \PSMCat \colon \mkern-3mu \CorCofMul-).
    \end{equation*}
    The above 2-adjunction restricts to a 2-adjunction
    \begin{equation*}
        l_{G/G}^{*}\colon \SMCatG \rightleftarrows \SMCat \colon \mkern-3mu \CorCofMul-).
    \end{equation*}
    This is a direct consequence of Lemma \ref{SMMontoGSMMon}, together with the fact that $l_{G/G}$ preserves finite products.
    
    All that is left is to check that the counit $\varepsilon$ of the adjunction $l_{G/G}^{*}\dashv\CorCofMul-)$ is given for symmetric monoidal categories by evaluating at $G/G$. Consider a commutative triangle
    \begin{equation}\label{EqCounitTriangle}
        \begin{tikzcd}
            \PSMCat \arrow[dr,"{\FunTen((\FinG^{\simeq},\amalg),-)}"'] \arrow[rr,"{\CorCofMul-)}"] &&\PSMCatG \arrow[dl,"{\ev_{G/G}}"]
            \\
            &\Cat.
        \end{tikzcd}
    \end{equation}
    From Lemma \ref{Yoneda}, we get a natural equivalence
    \begin{equation*}
        \ev_{G/G}\simeq \FunTenG(\Hom_{\Span(\FinG)}(G/G,-),\bullet)\colon\PSMCatG \xrightarrow{} \Cat,
    \end{equation*}
    and therefore by Remark \ref{2-Cats}, an adjunction
    \begin{equation*}
        \Hom_{\Span(\FinG)}(G/G,-)\times\bullet \colon \Cat \rightleftarrows \PSMCatG \colon \mkern-3mu \ev_{G/G}.
    \end{equation*}
    Furthermore, the counit $\epsilon$ of this adjunction after evaluating at $G/G$ is given by identity when restricted to $\id_{G/G}\in \Hom_{\Span(\FinG)}(G/G,G/G)$. More precisely, for a $G$-symmetric monoidal $G$-precategory $\Ff$ we have that the value of the counit at $G/G$ restricts to a functor
    \begin{equation*}
        \{\id_{G/G}\}\times \Ff(G/G) \xrightarrow{} \Hom_{\Span(\FinG)}(G/G,G/G)\times \Ff(G/G) \xrightarrow{} \Ff(G/G)
    \end{equation*}
    equivalent to the identity functor. Let $\Ee$ be a symmetric monoidal precategory. From the commutative diagram (\ref{EqCounitTriangle}), we obtain a commutative diagram of counits
    \[
        \begin{tikzcd}
            \Hom_{\Span(\FinG)}(G/G,l_{G/G}(-))\times\FunTen((\FinG^{\simeq},\amalg),\Ee) \arrow[dr,"{\ev}"'] \arrow[rr,"{\epsilon}"]  &&\CorCofMul\Ee)(l_{G/G}(-)) \arrow[dl,"{\varepsilon}"]
            \\
            &\Ee.
        \end{tikzcd}
    \]
    By the above discussion, after evaluating this diagram at $*\in\Span(\FCopC)$ and precomposing with the functor $\{\id_{G/G}\}\times \FunTen((\FinG^{\simeq},\amalg),\Ee) \xrightarrow{} \Hom_{\Span(\FinG)}(G/G,G/G)\times \FunTen((\FinG^{\simeq},\amalg),\Ee)$ we obtain a commutative diagram
    \[
        \begin{tikzcd}
            \FunTen((\FinG^{\simeq},\amalg),\Ee) \arrow[dr,"{\ev_{id_{G/G}}}"'] \arrow[rr,"{\id}"]  &&\FunTen((\FinG^{\simeq},\amalg),\Ee) \arrow[dl,"{\varepsilon_{G/G}}"]
            \\
            &\Ee(*),
        \end{tikzcd}
    \]
    where the functor $\ev_{id_{G/G}}$ is equivalent to the composite
    \begin{equation*}
        \FunTen((\FinG^{\simeq},\amalg),\Ee) \xrightarrow{(G/G,\id)} \FinG^{\simeq}\times\FunTen((\FinG^{\simeq},\amalg),\Ee) \xrightarrow{\ev} \Ee(*).
    \end{equation*}
    We now specialize to the case $\Ee=(\Cc,\otimes)$ of a symmetric monoidal category. The 2-functor $\fgt\colon \SMCat \xrightarrow{} \Cat$ induces a commutative diagram
    \[
        \begin{tikzcd}[column sep =large]
             \FunTen((\FinG^{\simeq},\amalg),(\Cc,\otimes)) \arrow[r,"{(G/G,\id)}"] \arrow[d,"{\fgt}"']  &\FinG^{\simeq}\times\FunTen((\FinG^{\simeq},\amalg),(\Cc,\otimes)) \arrow[d,"{id_{\FinG^{\simeq}}\times\fgt}"'] \arrow[dr, "{\ev}"]
             \\
             \Fun(\FinG^{\simeq},\Cc) \arrow[r,"{(G/G,\id)}"] \arrow[d,"{-\circ \incl}"'] &\FinG^{\simeq}\times\Fun(\FinG^{\simeq},\Cc) \arrow[r,"{\ev}"] &\Cc.
             \\
             \Fun(\OG^{\simeq},\Cc) \arrow[urr, out=0, in=-160, "{\ev_{G/G}}"']
        \end{tikzcd}
    \]
    The composite of the leftmost horizontal functors is the equivalence from Proposition \ref{GSMlifting}. This finishes the proof. 
\end{proof}

\subsection{Normed comparison functor}

We are now ready to extend the parametrized comparison functor from Construction \ref{ParamComparFunct} to the normed case. We will need to consider the $G$-symmetric monoidal $G$-category of $G$-spectra. The construction of this $G$-symmetric monoidal category can be found in Section 9.2 in \cite{NormsMotivic}, and Section 5.5 in \cite{HighTamb}. We will briefly recall the relevant facts without going into details.

The $G$-category $\PSp_{G}$ upgrades by \cite[Theorem 5.5.1]{HighTamb}, to a $G$-symmetric monoidal $G$-category, which we denote by $\PSp_{G}^{\otimes}$. The way to think about this $G$-symmetric monoidal $G$-category structure is the following. One first equips the $G$-category $\PSpc_{G,*}$ with the $G$-symmetric monoidal $G$-category structure given by smash product. More precisely, on the forward morphism $K\leq H\leq G$, it is given by localization at equivariant weak homotopy equivalences of the functor $\Fun(\textup{B}K,\textup{SSet}_{*}) \xrightarrow{} \Fun(\textup{B}H,\textup{SSet}_{*})$ sending $X\colon \textup{BK}\xrightarrow{}\textup{SSet}_{*}$ to the $|H/K|$-fold smash product $\bigwedge_{|H/K|}X$ with the $H$-action given by restricting along a certain homomorphism $H\xrightarrow{}\Sigma_{|H/K|}\wr K$, see \cite[Construction 5.2.10, Construction A.3.4]{HighTamb} for details. The functors $\Sigma^{\infty}\colon\PSpc_{G,*}(G/H)\xrightarrow{}\PSp_{G}(G/H)$ assemble into a $G$-functor. Furthermore, these symmetric monoidal functors, by \cite[Proposition 5.5.6]{HighTamb}, exhibit $\PSp_{G}(G/H)$ as the universal symmetric monoidal inversion of representation spheres in the sense of \cite[Lemma 4.1, Remark 4.2]{NormsMotivic}. For a morphism $f\colon X\xrightarrow{}Y$ in $\FinG$, the functor $f_{\otimes}\colon \PSp_{G}(X) \xrightarrow{} \PSp_{G}(Y)$ is then the unique sifted colimit preserving, symmetric monoidal functor making the square
\[
    \begin{tikzcd}
        \PSpc_{G,*}(X) \arrow[r,"{\Sigma^{\infty}}"] \arrow[d,"{f_{\otimes}}"] &\PSp_{G}(X) \arrow[d,"{f_{\otimes}}"]
        \\
        \PSpc_{G,*}(Y) \arrow[r,"{\Sigma^{\infty}}"] &\PSp_{G}(Y),
    \end{tikzcd}
\]
commute. In particular, for a morphism $i$ corresponding to the subgroup inclusion $K\leq H\leq G$, we get an equivalence $i_{\otimes}\simeq \textup{N}_{K}^{H}$ with the HHR-norm. 

\begin{defi}\label{GSMComparisonFunctor}
    We denote by
    \begin{equation*}
         \phi^{\otimes}\colon \PSp_{G}^{\otimes} \xrightarrow{} \CorCofMul(\Sp,\otimes))
    \end{equation*}
    the $G$-symmetric monoidal $G$-functor obtained using the adjunction $l_{G/G}^{*}\dashv \CorCofMul-)$ (Proposition \ref{GSymmMonAdj}) from the symmetric monoidal functor
    \begin{equation*}
        \Phi^{G}\colon (\Sp_{G},\otimes) \xrightarrow{} (\Sp,\otimes).
    \end{equation*}
\end{defi}

We will show that the $G$-symmetric monoidal $G$-functor $\phi^{\otimes}$ restricts to the $G$-functor $\phi\colon \PSp_{G} \xrightarrow{} \CorCof\Sp)$ (from Proposition \ref{ComparFunBeforeRationalizing}) on the underlying $G$-categories. We will do this using the universal property of the $G$-category of $G$-spectra, which we now recall.

Let $\underline{\Cc}$ and $\underline{\Dd}$ be $G$-presentable $G$-categories, we denote by $\Fun_{G}^{L}(\underline{\Cc},\underline{\Dd})$ the full subcategory of $\Fun_{G}(\underline{\Cc},\underline{\Dd})$ spanned by the $G$-colimit preserving $G$-functors. The next theorem is a combination of \cite[Theorem 10.7]{Partial} and \cite[Theorem 5.4.4]{HighTamb}.

\begin{theo}\cite[Theorem 10.7]{Partial}\label{ParametrizedUnivPropPresSp}
    Let $\underline{\Cc}$ be a $G$-stable, $G$-presentable $G$-category. The evaluation at the sphere spectrum functor
    \begin{equation*}
        \Fun_{G}^{L}(\PSp_{G},\underline{\Cc})\xrightarrow{\simeq} \underline{\Cc}(G/G)
    \end{equation*}
    is an equivalence.
\end{theo}

To apply the universal property of the $G$-category of $G$-spectra, it remains to check one statement. We already know the value of both $\phi$ and $\phi^{\otimes}$ on the sphere spectrum, and that $\phi$ preserves $G$-colimits. It therefore remains to verify that $\phi^{\otimes}$ also preserves $G$-colimits.

\begin{lem}\label{GSMCompFunPreservesColimits}
    The $G$-functor $\fgt(\phi^{\otimes})$ underlying the $G$-symmetric monoidal $G$-functor
    \begin{equation*}
         \phi^{\otimes}\colon \PSp_{G}^{\otimes} \xrightarrow{} \CorCofMul(\Sp,\otimes))
    \end{equation*}
    preserves $G$-colimits.
\end{lem}

\begin{proof}
    The $G$-categories $\PSp_{G}$ and $\CorCof\Sp)$ are $G$-presentable, for $\PSp_{G}$ this follows from \cite[Theorem 5.4.10]{HighTamb}, and for $\CorCof\Sp)$ this is Lemma \ref{CorCofPres}. Therefore, to see that the $G$-functor $\fgt(\phi^{\otimes})$ preserves $G$-colimits, it is enough to check the conditions of Remark \ref{PPrescolim}. 
    
    We start by checking condition (1), let $G/H$ be an object of $\OG$, we need to show that the functor
    \begin{equation*}
        \phi^{\otimes}_{G/H}\colon \PSp_{G}(G/H) \xrightarrow{} \CorCof\Sp)(G/H)\simeq \Fun(({\OG}_{/(G/H)})^{\simeq},\Sp)
    \end{equation*}
    preserves colimits. Since the target is a functor category, where the colimits are calculated objectwise, it is enough to check that for any morphism $\alpha\colon G/K\xrightarrow{}G/H$ in $\OG$ the functor
    \begin{equation*}
        \ev_{\alpha}\phi^{\otimes}_{G/H}\colon \PSp_{G}(G/H) \xrightarrow{} \Sp
    \end{equation*}
    preserves colimits. We can assume after conjugating that $\alpha\colon G/K\xrightarrow{}G/H$ corresponds to a subgroup inclusion $K\leq H$. We then calculate
    \begin{align*}
        \ev_{\alpha}\phi^{\otimes}_{G/H} &\simeq \ev_{id_{G/K}}\alpha^{*}\phi^{\otimes}_{G/H}
        \\
        &\simeq \ev_{id_{G/K}}\phi^{\otimes}_{G/K}\alpha^{*}.
    \end{align*}
    Let $i\colon G/K \xrightarrow{} G/G$ be the unique morphism to the terminal object. Consider a commutative diagram
    \[
        \begin{tikzcd}[column sep =huge, row sep =large]
             \FunTen((\FinG^{\simeq},\amalg),(\Sp,\otimes)) \arrow[r,"{i_{\otimes}}"] \arrow[d,"{\fgt}"']  &\FunTen((\FinG^{\simeq},\amalg),(\Sp,\otimes)) \arrow[d,"{\fgt}"'] \arrow[dr, "{\ev_{\id_{G/G}}}"]
             \\
             \Fun(({\FinG}_{/(G/K)})^{\simeq},\Sp) \arrow[r,"{-\circ (G/K\times_{G/G}-)}"] \arrow[d,"{-\circ \incl}"'] &\Fun(({\FinG}_{/(G/G)})^{\simeq},\Sp) \arrow[r,"{\ev_{id_{G/G}}}"] \arrow[d,"{-\circ \incl}"'] &\Sp.
             \\
             \Fun(({\OG}_{/(G/K)})^{\simeq},\Sp) \arrow[r,"{i_{\otimes}}"] &\Fun(({\OG}_{/(G/G)})^{\simeq},\Sp) \arrow[ur,"{\ev_{\id_{G/G}}}"']
        \end{tikzcd}
    \]
    The middle horizontal composite in the above diagram is equivalent to the functor $\ev_{\id_{G/K}}$, and we get that $\ev_{id_{G/G}}i_{\otimes}\simeq \ev_{id_{G/K}}$. Therefore, we have
    \begin{align*}
        \ev_{\alpha}\phi^{\otimes}_{G/H} &\simeq \ev_{id_{G/K}}\phi^{\otimes}_{G/K}\alpha^{*}
        \\
        &\simeq \ev_{id_{G/G}}i_{\otimes}\phi^{\otimes}_{G/K}\alpha^{*}
        \\
        &\simeq \ev_{id_{G/G}}\phi^{\otimes}_{G/G}i_{\otimes}\alpha^{*}.
    \end{align*}
    Now, from Proposition \ref{GSymmMonAdj} and Definition \ref{GSMComparisonFunctor}, we note that $\ev_{\id_{G/G}}$ is the functor underlying the counit of the adjunction $l_{G/G}^{*}\dashv \CorCofMul-)$, and therefore 
    \begin{equation*}
        \ev_{id_{G/G}}\phi^{\otimes}_{G/G}\simeq \ev_{id_{G/G}}l_{G/G}^{*}(\phi^{\otimes})\simeq \Phi^{G}.
    \end{equation*}
    We continue the calculation
    \begin{align*}
        \ev_{\alpha}\phi^{\otimes}_{G/H} &\simeq \ev_{id_{G/G}}\phi^{\otimes}_{G/G}i_{\otimes}\alpha^{*}
        \\
        &\simeq \Phi^{G}i_{\otimes}\alpha^{*}
        \\
        &\simeq \Phi^{K}\alpha^{*}
        \\
        &\simeq \Phi^{K},
    \end{align*}
    where the third equivalence simply says that $\Phi^{G}\textup{N}_{K}^{G}\simeq \Phi^{K}$. We can now conclude, since the functor $\Phi^{K}\colon \Sp_{H}\xrightarrow{}\Sp$ preserves colimits.
    
    It is left to check condition (2) of Remark \ref{PPrescolim}. For a morphism $f\colon G/L \xrightarrow{} G/H$ in $\OG$, we get a commutative square
    \[
        \begin{tikzcd}
            \PSp_{G}(G/L)  \arrow[r,"{\phi^{\otimes}_{G/L}}"] &\CorCof\Sp)(G/L) 
            \\
            \PSp_{G}(G/H) \arrow[u,"{f^{*}}"'] \arrow[r,"{\phi^{\otimes}_{G/H}}"] &\CorCof\Sp)(G/H). \arrow[u,"{f^{*}}"']
        \end{tikzcd}
    \]
    We need to show that the Beck-Chevalley transformation
    \[
        \begin{tikzcd}
            \PSp_{G}(G/L) \arrow[d,"{f_{!}}"'] \arrow[r,"{\phi^{\otimes}_{G/L}}"] &\CorCof\Sp)(G/L) \arrow[d,"{f_{!}}"'] \arrow[dl, Rightarrow, "{\BC_{f}}"']
            \\
            \PSp_{G}(G/H) \arrow[r,"{\phi^{\otimes}_{G/H}}"'] &\CorCof\Sp)(G/H)
        \end{tikzcd}
    \]
    is an equivalence. Since the target category is a functor category and evaluation functors are jointly conservative, we can show the equivalence after first evaluating. Let $\alpha\colon G/K \xrightarrow{} G/H$ be a morphism in $\OG$, which we once again can assume to correspond to a subgroup inclusion $K\leq H$. Let $S$ denote the following set
    \begin{equation*}
        S\coloneq \pi_{0}(({\OG}_{/(G/L)})^{\simeq}_{/\alpha}),
    \end{equation*}
    where we fix a representative $(\beta_{P}\colon G/P\xrightarrow{}G/L,\gamma_{P}\colon G/P\xrightarrow{\simeq}G/K)$ for each path component. In addition, we assume that $\beta_{P}\colon G/P\xrightarrow{}G/L$ corresponds to a subgroup inclusion $P\leq L$. We have the following commutative diagram
    \begin{equation}\label{EqBCTransfSq1}
        \begin{tikzcd}[column sep =large]
            \PSp_{G}(G/L)  \arrow[r,"{\phi^{\otimes}_{G/L}}"] &\CorCof\Sp)(G/L) \arrow[r,"{(\ev_{\beta_{P}})_{[P]\in S}}"] &\prod_{[P]\in S}\Sp
            \\
            \PSp_{G}(G/H) \arrow[u,"{f^{*}}"'] \arrow[r,"{\phi^{\otimes}_{G/H}}"] &\CorCof\Sp)(G/H) \arrow[u,"{f^{*}}"'] \arrow[r,"{\ev_{\alpha}}"] &\Sp. \arrow[u,"{\Delta}"]
        \end{tikzcd}
    \end{equation}
    In the above diagram, the Beck-Chevalley transformation obtained by passing to the left adjoints of vertical functors in the right-hand square is an equivalence. To see this, let $\Xx\colon ({\OG}_{/(G/L)})^{\simeq} \xrightarrow{} \Sp$ be a functor. By the pointwise formula for left Kan extensions, we have
    \begin{equation*}
        \ev_{\alpha}f_{!}(\Xx) \simeq \Lan_{f\circ -}\Xx \simeq \colim_{({\OG}_{/(G/L)})^{\simeq}_{/\alpha}}\Xx,
    \end{equation*}
    and the claim follows by observing that the functor
    \begin{equation*}
        S\xrightarrow{} ({\OG}_{/(G/L)})^{\simeq}_{/\alpha}, \quad [P]\mapsto (\beta_{P},\gamma_{P}\colon f\beta_{P}\simeq \alpha)
    \end{equation*}
    is an equivalence of discrete groupoids. Therefore, it suffices to show that the Beck-Chevalley transformation obtained by passing to the left adjoints of the outer vertical functors in the diagram (\ref{EqBCTransfSq1}) is an equivalence. Under identifications similar to the ones in the first half of this proof, we can equivalently work with the diagram
    \begin{equation}\label{EqBCTransfSq2}
        \begin{tikzcd}[column sep =large]
            \PSp_{G}(G/L) \arrow[r,"{(t_{\otimes}(\beta_{P}\gamma_{P}^{\inv})^{*})_{S}}"'] &\prod_{[P]\in S}\PSp_{G}(G/G) \arrow[r,"{\phi^{\otimes}_{G/G}}"'] &\prod_{[P]\in S}\Fun(\OG^{\simeq},\Sp) \arrow[r,"{\prod_{S}\ev_{G/G}}"'] &\prod_{[P]\in S}\Sp
            \\
            \PSp_{G}(G/H) \arrow[u,"{f^{*}}"] \arrow[r,"{t_{\otimes}\alpha^{*}}"'] &\PSp_{G}(G/G) \arrow[u,"{\Delta}"] \arrow[r,"{\phi^{\otimes}_{G/G}}"'] &\Fun(\OG^{\simeq},\Sp) \arrow[u,"{\Delta}"] \arrow[r,"{\ev_{G/G}}"'] &\Sp \arrow[u,"{\Delta}"],
        \end{tikzcd}
    \end{equation}
    where $t$ denotes the unique map $G/K\xrightarrow{}G/G$. The left-most square in the above diagram commutes by functoriality of $\PSp_{G}$, and the other two squares commute in the obvious way. For the rest of the proof, we will denote by $\Pp_{\Sigma}(-)$ the sifted colimit completion functor with the covariant functoriality, and by $-\otimes-$ the Lurie tensor product of presentable categories. Using results from \cite[Appendix A]{Descent}, we have an identification
    \begin{equation*}
        \Phi^{K}\simeq \ev_{*}(\Pp_{\Sigma}(\Span((-)^{K}))\otimes \Sp)\colon \FunProd(\Span({\FinG}_{/(G/K)}),\Sp)\xrightarrow{} \FunProd(\Span({\FCopC}),\Sp)\simeq\Sp.
    \end{equation*}
    We combine the equivalence $\ev_{G/G}\phi^{\otimes}_{G/G}t_{\otimes}\simeq \Phi^{K}$, together with the above identification to rewrite the diagram (\ref{EqBCTransfSq2}), as
    \begin{equation*}
        \begin{tikzcd}[column sep =3 cm]
            \PSp_{G}(G/L)  \arrow[r,"{(\Pp_{\Sigma}(\Span((\beta_{P}\gamma_{P}^{\inv})^{*}))\otimes\Sp)_{S}}"'] &\prod_{[P]\in S}\PSp_{G}(G/K) \arrow[r,"{\prod_{S}\Pp_{\Sigma}(\Span((-)^{K}))\otimes\Sp}"'] &\prod_{[P]\in S}\FunProd(\Span({\FCopC}),\Sp)
            \\
            \PSp_{G}(G/H) \arrow[r,"{\Pp_{\Sigma}(\Span(\alpha^{*}))\otimes\Sp}"'] \arrow[u,"{\Pp_{\Sigma}(\Span(f^{*}))\otimes\Sp}"'] &\PSp_{G}(G/K) \arrow[r,"{\Pp_{\Sigma}(\Span((-)^{K}))\otimes\Sp}"'] \arrow[u,"{\Delta}"'] &\FunProd(\Span({\FCopC}),\Sp). \arrow[u,"{\Delta}"']
        \end{tikzcd}    
    \end{equation*}
    We note that the left adjoint of the functor $\Pp_{\Sigma}(\Span(f^{*}))\otimes\Sp$ is given by $\Pp_{\Sigma}(\Span(f\circ-))\otimes\Sp$, and the left adjoint of $\Delta\colon \Sp\xrightarrow{}\prod_{[P]\in S}\Sp$ is given by $\Pp_{\Sigma}(\Span(\amalg))\otimes\Sp$. Hence, $\BC_{f}$ being an equivalence reduces to showing that the mate natural transformation makes the diagram
    \[
        \begin{tikzcd}[column sep =large]
            {\FinG}_{/(G/L)} \arrow[d,"{f\circ -}"'] \arrow[r,"{((\beta_{P}\gamma_{P}^{\inv})^{*})_{S}}"] &\prod_{[P]\in S}{\FinG}_{/(G/K)} \arrow[r,"{\prod_{S}(-)^{K}}"] &\prod_{[P]\in S}\FCopC \arrow[d,"{\amalg}"']
            \\
            {\FinG}_{/(G/H)} \arrow[r,"{\alpha^{*}}"] &{\FinG}_{/(G/K)} \arrow[r,"{(-)^{K}}"] &\FCopC
        \end{tikzcd}
    \]
    commute. Let $\mu\colon X\xrightarrow{}G/L$ be an object of ${\FinG}_{/(G/L)}$. We have bijections
    \begin{equation*}
        (\alpha^{*}(f\mu))^{K}\cong \{\lambda\colon G/K\xrightarrow{}X|f\mu\lambda=\alpha\},
    \end{equation*}
    and
    \begin{equation*}
        \coprod_{[P]\in S}((\beta_{P}\gamma_{P}^{\inv})^{*}(\mu))^{K} \cong \coprod_{[P]\in S}\{\widetilde{\lambda}\colon G/K\xrightarrow{}X|\mu\widetilde{\lambda}=\beta_{P}\gamma_{P}^{\inv}\}.
    \end{equation*}
    Under these identifications, one can check that the canonical map
    \begin{equation*}
        \coprod_{[P]\in S}((\beta_{P}\gamma_{P}^{\inv})^{*}(\mu))^{K}\xrightarrow{}(\alpha^{*}(f\mu))^{K}, \quad \widetilde{\lambda} \mapsto \widetilde{\lambda}
    \end{equation*}
    is a bijection. This finishes the proof.
\end{proof}

We are now ready to compare the $G$-functors $\phi$ and $\fgt(\phi^{\otimes})$.

\begin{prop}\label{UnderlyingCompFuncIdentification}
    The $G$-functor
    \begin{equation*}
        \phi \colon \PSp_{G} \xrightarrow{} \CorCof \Sp)
    \end{equation*}
    from Proposition \ref{ComparFunBeforeRationalizing}, is the underlying $G$-functor of the $G$-symmetric monoidal $G$-functor
    \begin{equation*}
         \phi^{\otimes}\colon \PSp_{G}^{\otimes} \xrightarrow{} \CorCofMul(\Sp,\otimes))
    \end{equation*}
    from Definition \ref{GSMComparisonFunctor}.
\end{prop}

\begin{proof}
    We will use the universal property of the $G$-category of $G$-spectra from Theorem \ref{ParametrizedUnivPropPresSp}. The $G$-category $\CorCof \Sp)$ is $G$-stable and $G$-presentable by Proposition \ref{CorCofGStableGpresentable}. By Proposition \ref{ComparFunBeforeRationalizing}, the $G$-functor 
    \begin{equation*}
        \phi \colon \PSp_{G} \xrightarrow{} \CorCof \Sp)
    \end{equation*}
    preserves $G$-colimits and sends the sphere spectrum to the constant functor taking value $\Sphere\in\Sp$. The $G$-functor
    \begin{equation*}
        \fgt(\phi^{\otimes})\colon \PSp_{G} \xrightarrow{} \CorCof \Sp)
    \end{equation*}
    preserves $G$-colimits by Lemma \ref{GSMCompFunPreservesColimits}. The functor
    \begin{equation*}
        \phi^{\otimes}_{G/G}\colon \PSp_{G}(G/G) \xrightarrow{} \CorCof \Sp)(G/G)
    \end{equation*}
    is a symmetric monoidal functor with respect to the symmetric monoidal structure obtained by applying the functor
    \begin{equation*}
        l_{G/G}^{*} \colon \SMCatG \xrightarrow{} \SMCat.
    \end{equation*}
    This symmetric monoidal structure on the category $\CorCof \Sp)(G/G)$ is the componentwise one according to Lemma \ref{Componentwise}, and in particular, the functor $\phi^{\otimes}_{G/G}(\Sphere)$ is equivalent to the constant functor taking value $\Sphere\in\Sp$. We now apply Theorem \ref{ParametrizedUnivPropPresSp} to obtain an equivalence $\fgt(\phi^{\otimes})\simeq\phi$, finishing the proof.
\end{proof}

We now invert rational equivalences in our $G$-symmetric monoidal $G$-categories and show that, after rationalization, the $G$-symmetric monoidal $G$-functor $\phi^{\otimes}$ becomes an equivalence.

\begin{rem}
    Similar to Remark \ref{RelCat}, the adjunction
    \begin{equation*}
        \Ll \colon \RelCat \rightleftarrows \Cat \colon \mkern-3mu (-,(-)^{\simeq})
    \end{equation*}
    induces an adjunction
    \begin{equation*}
        \Ll \colon \Fun(\Span(\FinG),\RelCat) \rightleftarrows \Fun(\Span(\FinG),\Cat) \colon \mkern-3mu (-,(-)^{\simeq}).
    \end{equation*}
    The finite products in $\RelCat$ are computed componentwise, and the Dwyer-Kan localization functor preserves finite products. Therefore, the above adjunction restricts to an adjunction
    \begin{equation*}
        \Ll \colon \FunProd(\Span(\FinG),\RelCat) \rightleftarrows \SMCat \colon \mkern-3mu (-,(-)^{\simeq}).
    \end{equation*}
\end{rem}

\begin{cons}
    We lift the $G$-symmetric monoidal $G$-category $\PSp_{G}^{\otimes}\colon\Span(\FinG)\xrightarrow{}\Cat$ to a finite product preserving functor
    \begin{equation*}
        (\PSp_{G}^{\otimes},\text{Rational equivalences})\colon \Span(\FinG) \xrightarrow{} \RelCat
    \end{equation*}
    which is defined on transitive $G$-sets in the same way as in Construction \ref{EquipWithRatEq}. The requirement of preserving finite products uniquely extends this structure to all $G$-sets. Let $f\colon X\xrightarrow{}Y$ be a morphism of finite $G$-sets. The statement that $f^*$ preserves rational equivalences boils down to the argument in Construction \ref{EquipWithRatEq}. To show that the functor $f_{\otimes}$ preserves rational equivalences it is enough to consider the following two cases
    \begin{equation*}
        G/H\amalg G/H \xrightarrow{\nabla} G/H, \quad and \quad G/K \xrightarrow{\alpha} G/H,
    \end{equation*}
    in particular, we can assume that $\alpha$ corresponds to a subgroup inclusion $K\leq H$. In the case of $\nabla$ the statement follows from the natural equivalence
    \begin{equation*}
        (\Xx\otimes\Yy)\otimes\Sphere_{\RatQ}\simeq (\Xx\otimes\Yy)\otimes(\Sphere_{\RatQ}\otimes\Sphere_{\RatQ}) \simeq (\Xx\otimes\Sphere_{\RatQ})\otimes(\Yy\otimes\Sphere_{\RatQ}).
    \end{equation*}
    The case of $\alpha\colon G/K\xrightarrow{}G/H$ follows from the fact that $\alpha_{\otimes}\colon\Sp_{K}\xrightarrow{}\Sp_{H}$ is symmetric monoidal and the equivalence $\alpha_{\otimes}(\Sphere_{\RatQ})\simeq \Sphere_{\RatQ}$. To verify the last claim, we look at the chain of equivalences
    \begin{align*}
        \alpha_{\otimes}(\Sphere_{\RatQ}) &\simeq \colim (\Sphere \xrightarrow{\alpha_{\otimes}(p_{1})} \Sphere \xrightarrow{\alpha_{\otimes}(p_{1}p_{2})} \Sphere \xrightarrow{\alpha_{\otimes}(p_{1}p_{2}p_{3})} ...)
        \\
        &\simeq \colim (\Sphere \xrightarrow{p_{1}^{|H/K|}} \Sphere \xrightarrow{(p_{1}p_{2})^{|H/K|}} \Sphere \xrightarrow{(p_{1}p_{2}p_{3})^{|H/K|}} ...)
        \\
        &\simeq \Sphere_{\RatQ}.
    \end{align*}
    The first equivalence is the consequence of the functor $\alpha_{\otimes}$ preserving sifted colimits. For the second equivalence observe that for a natural number $n$, the morphism $n\colon\Sphere\xrightarrow{}\Sphere$ is defined as a composite $\Sphere\xrightarrow{\Delta}\bigoplus_{n}\Sphere\xrightarrow{\nabla}\Sphere$. Using definition of $\alpha_{\otimes}$ we get
    \begin{equation*}
        \alpha_{\otimes}(\bigoplus_{n}\Sphere)\simeq \alpha_{\otimes}\Sigma^{\infty}\bigvee_{n} S^{0}\simeq \Sigma^{\infty}\bigwedge_{|H/K|}\bigvee_{n} S^{0} \simeq \Sigma^{\infty}\bigvee_{n^{|H/K|}} S^{0} \simeq \bigoplus_{n^{|H/K|}} \Sphere,
    \end{equation*}
    showing the claim.
\end{cons}

\begin{defi}\label{GSMLocFunc}
    We denote the $G$-symmetric monoidal $G$-category
    \begin{equation*}
        \Ll(\PSp_{G}^{\otimes},\text{Rational equivalences})\in\SMCatG
    \end{equation*}
    by
    \begin{equation*}
        \PSp_{G,\RatQ}^{\otimes},
    \end{equation*}
    and call it the \textit{$G$-symmetric monoidal $G$-category of rational $G$-spectra}.
\end{defi}

We get a $G$-symmetric monoidal localization functor 
\begin{equation*}
    \underline{L}_{\RatQ}^{\otimes}\colon \PSp_{G}^{\otimes} \xrightarrow{} \PSp_{G,\RatQ}^{\otimes}.
\end{equation*}
It is immediate from the above definition that the underlying $G$-category of $\PSp_{G,\RatQ}^{\otimes}$ is $\PSp_{G,\RatQ}$.

\begin{cons}
    The composite $G$-symmetric monoidal $G$-functor 
    \begin{equation*}
        \PSp_{G}^{\otimes} \xrightarrow{\phi^{\otimes}} \CorCofMul (\Sp,\otimes)) \xrightarrow{(L_{\RatQ})_*} \CorCofMul (\Sp_{\RatQ},\otimes))
    \end{equation*}
    sends rational equivalences to equivalences, in each degree. This follows immediately from the argument in Construction \ref{ParamComparFunct}. Hence, we obtain a $G$-symmetric monoidal $G$-functor
    \begin{equation*}
        \phi_{\RatQ}^{\otimes}\colon \PSp_{G,\RatQ}^{\otimes} \xrightarrow{} \CorCofMul (\Sp_{\RatQ},\otimes)),
    \end{equation*}
    fitting in a commutative diagram
    \[
        \begin{tikzcd}
        \PSp_{G}^{\otimes} \arrow[r,"{(L_{\RatQ}^{\otimes})_*\phi^{\otimes}}"] \arrow[d,"{\underline{L}_{\RatQ}^{\otimes}}"'] & \CorCofMul (\Sp_{\RatQ},\otimes))
        \\
        \PSp_{G,\RatQ}^{\otimes} \arrow[ur,"{\phi_{\RatQ}^{\otimes}}"']
        \end{tikzcd}
    \]
    of $G$-symmetric monoidal $G$-categories.
\end{cons}

\begin{theo}\label{SMEquivalence}
    The $G$-symmetric monoidal $G$-functor
    \begin{equation*}
        \phi_{\RatQ}^{\otimes}\colon \PSp_{G,\RatQ}^{\otimes} \xrightarrow{} \CorCofMul (\Sp_{\RatQ},\otimes))
    \end{equation*}
    is an equivalence.
\end{theo}

\begin{proof}
    We have a commutative diagram
    \[
        \begin{tikzcd}[column sep =large]
            \PSp_{G} \arrow[r,"{\phi}"] \arrow[d,"{\simeq}"'] &\CorCof \Sp) \arrow[d,"{\simeq}"'] \arrow[r,"{(L_{\RatQ})_{*}}"] & \CorCof \Sp_{\RatQ}) \arrow[d,"{\simeq}"']
            \\
            \fgt(\PSp_{G}^{\otimes}) \arrow[r,"{\fgt(\phi^{\otimes})}"] \arrow[d,"{\fgt(\underline{L}_{\RatQ}^{\otimes})}"'] &\fgt(\CorCofMul (\Sp,\otimes))) \arrow[r,"{\fgt((L_{\RatQ}^{\otimes})_{*})}"] & \fgt(\CorCofMul (\Sp_{\RatQ},\otimes)))
            \\
            \fgt(\PSp_{G,\RatQ}^{\otimes}) \arrow[urr, out=0, in=-160, "{\fgt(\phi_{\RatQ}^{\otimes})}"']
        \end{tikzcd}
    \]
    where the top left square commutes by Proposition \ref{UnderlyingCompFuncIdentification}, and the top right square commutes by Proposition \ref{GSMlifting}. The left vertical composite is equivalent to the functor
    \begin{equation*}
        \underline{L}_{\RatQ}\colon \PSp_{G} \xrightarrow{} \PSp_{G,\RatQ},
    \end{equation*}
    and it follows that we have an equivalence $\fgt(\phi_{\RatQ}^{\otimes})\simeq\phi_{\RatQ}$. By Theorem \ref{GEquivalence}, the $G$-functor $\phi_{\RatQ}$ is an equivalence. Since the forgetful functor $\fgt\colon\SMCatG\xrightarrow{}\Cat_{G}$ is conservative, It follows that the $G$-symmetric monoidal $G$-functor
    \begin{equation*}
        \phi_{\RatQ}^{\otimes}\colon \PSp_{G,\RatQ}^{\otimes} \xrightarrow{} \CorCofMul (\Sp_{\RatQ},\otimes))
    \end{equation*}
    is an equivalence.
\end{proof}

\section{A model for normed algebras in the \texorpdfstring{$G$}{G}-symmetric monoidal \texorpdfstring{$G$}{G}-category of rational \texorpdfstring{$G$}{G}-spectra}\label{Section5}

In the last section we study normed algebras in the $G$-symmetric monoidal $G$-category $\PSp_{G,\RatQ}^{\otimes}$. We start by recalling the notion of normed algebras.

\begin{defi}\cite[Definition 3.2.7]{HighTamb}
    Let $\underline{\Cc}^{\otimes}$ be a $G$-symmetric monoidal $G$-category, and let $\Ii$ be an indexing system in the sense of Definition \ref{IndexingSystem}. We denote the category
    \begin{equation*}
        \FunLaxG(\Span_{\text{all},\Ii}(\FinG),\Uncc(\underline{\Cc}^{\otimes}))
    \end{equation*}
    by $\NAlg_{\Ii}(\underline{\Cc}^{\otimes})$, and call its objects \textit{$\Ii$-normed algebras in $\underline{\Cc}^{\otimes}$}. In particular, for a trivial group we denote $\NAlg_{\FCopC}(\Cc,\otimes)$ simply by $\CAlg(\Cc,\otimes)$.
\end{defi}

Remark \ref{LaxFuncCat}, tells us that $\Ii$-normed algebra in $\underline{\Cc}^{\otimes}$ is simply a functor from $\Span_{\text{all},\Ii}(\FinG)$ to $\Uncc(\underline{\Cc}^{\otimes})$ over $\Span(\FinG)$, which sends backwards morphisms to cocartesian edges. Informally, one can think about an $\Ii$-normed algebra $\Zz$ in $\underline{\Cc}^{\otimes}$ as the following data
\begin{itemize}
    \item[-] an object $\Zz(G/G)$ of $\underline{\Cc}^{\otimes}(G/G)$,
    \item[-] a commutative algebra structure on $\Zz(G/G)$ inside $\underline{\Cc}^{\otimes}(G/G)$,
    \item[-] and for each morphism $f\colon G/H\xrightarrow{}G/K$ in $\Ii$ a norm morphism $f_{\otimes}\alpha^{*}\Zz(G/G)\xrightarrow{}\beta^{*}\Zz(G/G)$, where $\alpha$ and $\beta$ are unique morphisms from $G/H$ and $G/K$ to $G/G$,
\end{itemize}
satisfying certain coherence conditions.

To understand $\Ii$-normed algebras in $\PSp_{G,\RatQ}^{\otimes}$, we will need a simplified description of the $\Ii$-normed algebras in $\CorCofMul(\Sp_{\RatQ},\otimes))$. To obtain such a description, we construct an adjunction 
\begin{equation*}
        \Cat_{/\Span(\FinG)}^{\FinG^{\op}\text{-cc}} \rightleftarrows \Cat_{/\Span(\FCopC)}^{\FCopC^{\op}\text{-cc}}
\end{equation*}
compatible, in a suitable sense, with the adjunction $l_{G/G}^{*}\dashv \CorCofMul-)$. This adjunction will allow us to describe the $\Ii$-normed algebras in $\CorCofMul(\Cc,\otimes))$, for a symmetric monoidal category $\Cc$. We start by constructing the left adjoint
\begin{equation*}
    \Lambda_{G/G}\colon\Cat_{/\Span(\FinG)}^{\FinG^{\op}\text{-cc}} \xrightarrow{} \Cat_{/\Span(\FCopC)}^{\FCopC^{\op}\text{-cc}},
\end{equation*}
mentioned above. Then we check that there exists a natural equivalence $\Env\Lambda_{G/G}\simeq l_{G/G}^{*}\Env$ and that $\Lambda_{G/G}$ admits a 2-right adjoint. Finally, in Proposition \ref{NAlgVsAlg}, we obtain the simplified description of $\Ii$-normed algebras in $\CorCofMul(\Cc,\otimes))$.

\begin{cons}\label{ConsOpLAdj}
    We consider an adequate triple $(\FinG,\FCopC[\OG^{\simeq}],\FinG)$, and the associated inclusion on span categories
    \begin{equation*}
        \nu\colon\Span_{\FCopC[\OG^{\simeq}],\text{all}}(\FinG) \xrightarrow{} \Span(\FinG).
    \end{equation*}
    The unique functor $\OG\xrightarrow{}*$ induces a morphism of adequate triples $(\FinG,\FCopC[\OG^{\simeq}],\FinG)\xrightarrow{}(\FCopC,\FCopC,\FCopC)$, we denote the associated functor on span categories by
    \begin{equation*}
        \sigma \colon \Span_{\FCopC[\OG^{\simeq}],\text{all}}(\FinG) \xrightarrow{} \Span(\FCopC).
    \end{equation*}
    We have a composite functor
    \begin{equation*}
        \OpLAdj\colon \Cat_{/\Span(\FinG)} \xrightarrow{\nu^{*}} \Cat_{/\Span_{\FCopC[\OG^{\simeq}],\text{all}}(\FinG)} \xrightarrow{\sigma\circ-} \Cat_{/\Span(\FCopC)}.
    \end{equation*}
    The functor $\sigma \colon \Span_{\FCopC[\OG^{\simeq}],\text{all}}(\FinG) \xrightarrow{} \Span(\FCopC)$ admits cocartesian lifts of backwards morphisms. Furthermore, every backwards morphism in $\Span_{\FCopC[\OG^{\simeq}],\text{all}}(\FinG)$ is cocartesian. It is immediate from the above statement on cocartesian edges that the functor $\OpLAdj$ factors as
    \begin{equation*}
        \Cat_{/\Span(\FinG)}^{\FinG^{\op}\text{-cc}} \xrightarrow{\nu^{*}} \Cat_{/\Span_{\FCopC[\OG^{\simeq}],\text{all}}(\FinG)}^{\FCopC[\OG^{\simeq}]^{\op}\text{-cc}} \xrightarrow{\sigma\circ-} \Cat_{/\Span(\FCopC)}^{\FCopC^{\op}\text{-cc}}.
    \end{equation*}
\end{cons}

\begin{prop}\label{EnvLambdalEnv}
    There exists a natural transformation making the square
    \[
        \begin{tikzcd}
            \PSMCatG \arrow[r,"{l_{G/G}^{*}}"] &\PSMCat
            \\
            \POpG \arrow[u,"{\Env}"] \arrow[r,"{\OpLAdj}"] &\POp \arrow[u,"{\Env}"]
        \end{tikzcd}
    \]
    commute. In particular, the functor 
    \begin{equation*}
        \OpLAdj\colon \POpG \xrightarrow{} \POp,
    \end{equation*}
    from Construction \ref{ConsOpLAdj}, restricts to a functor
    \begin{equation*}
        \OpLAdj\colon \OpG \xrightarrow{} \Op.
    \end{equation*}
\end{prop}

\begin{proof}
    We first assume that the square commutes and show the statement about $\OpLAdj$ restricting to the categories of operads. This follows from \cite[Proposition 2.50]{NormsEq}, combined with the commutativity of the square, and $l_{G/G}^{*}$ sending $G$-symmetric monoidal $G$-categories to symmetric monoidal categories. 
    
    We now show that there exists a natural equivalence making the aforementioned square commute. Expanding definitions, the functor $\Env\OpLAdj\colon\POpG\xrightarrow{}\PSMCat$ is given by the restriction of the composite functor
    \begin{align*}
        \Cat_{/\Span(\FinG)} \xrightarrow{\nu^{*}} \Cat_{/\Span_{\FCopC[\OG^{\simeq}],\text{all}}(\FinG)} \xrightarrow{\sigma\circ-} \Cat_{/\Span(\FCopC)}
        \\
        \xrightarrow{\ev_{0}^{*}}\Cat_{/\Ar_{\FCopC}(\Span(\FCopC))} \xrightarrow{\ev_{1}\circ-}\Cat_{\Span(\FCopC)},
    \end{align*}
    and the functor $l_{G/G}^{*}\Env$ is given by the restriction of
    \begin{equation*}
        \Cat_{/\Span(\FinG)} \xrightarrow{\ev_{0}^{*}} \Cat_{/\Ar_{\FinG}(\Span(\FinG))} \xrightarrow{\ev_{1}\circ-} \Cat_{/\Span(\FinG)} \xrightarrow{l_{G/G}^{*}} \Cat_{\Span(\FCopC)}.
    \end{equation*}
    By Remark 2.25 in \cite{NormsEq} we have equivalences
    \begin{equation*}
        \Ar_{\FCopC}(\Span(\FCopC)) \simeq \Span_{\text{pb},\text{all}}(\Ar(\FCopC)), \quad \text{and} \quad \Ar_{\FinG}(\Span(\FinG)) \simeq \Span_{\text{pb},\text{all}}(\Ar(\FinG)),
    \end{equation*}
    where in the span categories, on the right side of the above equivalences, backwards morphisms are given by pullback squares. It is not difficult to see that to exhibit a natural equivalence between the above two functors, it is enough to construct an equivalence, compatible with certain functors, between the pullbacks of the diagrams
    \[
        \begin{tikzcd}
            \bullet \arrow[r] \arrow[d] &\Span_{\FCopC[\OG^{\simeq}],\text{all}}(\FinG) \arrow[d,"{\sigma}"'] &\bullet \arrow[r] \arrow[d] &\Span_{\text{pb},\text{all}}(\Ar(\FinG)) \arrow[d,"{\Span(\ev_{1})}"']
            \\
            \Span_{\text{pb},\text{all}}(\Ar(\FCopC)) \arrow[r,"\Span(\ev_{0})"] &\Span(\FCopC), &\Span(\FCopC) \arrow[r,"{l_{G/G}}"] &\Span(\FinG).
        \end{tikzcd}
    \]
    By Remark \ref{SpanPresLimits}, the above pullbacks are equivalently given by first taking pullbacks in the category $\AdTrip$ and then applying the span functor. We consider the composite functor
    \begin{equation*}
        \FCopC\times_{\FinG}\Ar(\FinG) \xrightarrow{\pr_{\Ar(\FinG)}} \Ar(\FinG) \xrightarrow{(\Ar(\kappa),\ev_{0})} \Ar(\FCopC)\times_{\FCopC}\FinG,
    \end{equation*}
    where $\kappa \colon \FinG \xrightarrow{} \FCopC$ denotes the unique finite coproduct preserving functor sending transitive $G$-sets to a set with one element. One observes that the above functor is an equivalence of adequate triples, and the induced equivalence on span categories produces a natural equivalence
    \begin{equation*}
        \Env\OpLAdj\simeq l_{G/G}^{*}\Env,
    \end{equation*}
    finishing the proof.
\end{proof}

\begin{prop}\label{OpLAdjIsLAdj}
    The functor
    \begin{equation*}
        \OpLAdj\colon \POpG \xrightarrow{} \POp
    \end{equation*}
    is a 2-functor, and admits a 2-right adjoint.
\end{prop}

\begin{proof}
    First observe that by \cite[Proposition 2.3.7]{EnvAlgPat}, both categories $\POpG$, and $\POp$ are $\Vv$-presentable. Therefore, by the Adjoint Functor Theorem, it suffices to show that the functor $\OpLAdj$ preserves large colimits. It suffices to show this after postcomposing with the conservative, left adjoint functor $\Env\colon \POp\xrightarrow{}\PSMCat$, see \cite[Theorem 2.23]{NormsEq}. By Proposition \ref{EnvLambdalEnv}, we have a natural equivalence $\Env\OpLAdj\simeq l_{G/G}^{*}\Env$, and the claim follows since $l_{G/G}^{*}\Env$ preserves large colimits.
    
    It only remains to show that the functor $\OpLAdj$ is $\Cat$-linear, since then, by Remark \ref{2-Cats}, the adjunction automatically upgrades to a 2-adjunction. The functor $\OpLAdj$ is defined as the composite
    \begin{equation*}
        \Cat_{/\Span(\FinG)}^{\FinG^{\op}\text{-cc}} \xrightarrow{\nu^{*}} \Cat_{/\Span_{\FCopC[\OG^{\simeq}],\text{all}}(\FinG)}^{\FCopC[\OG^{\simeq}]^{\op}\text{-cc}} \xrightarrow{\sigma\circ-} \Cat_{/\Span(\FCopC)}^{\FCopC^{\op}\text{-cc}}.
    \end{equation*}
    We show that both of the above functors are $\Cat$-linear separately. The functor $\nu^{*}$ is $\Cat$-linear since it preserves finite products, and for a category $\Ee$ there is an equivalence 
    \begin{equation*}
        \nu^{*}(\Ee\times\Span(\FinG))\simeq \Ee\times \Span_{\FCopC[\OG^{\simeq}],\text{all}}(\FinG),
    \end{equation*}
    natural in $\Ee$. For the functor $\sigma\circ-$, first observe that there is an adjunction
    \begin{equation*}
        \sigma\circ- \colon \Cat_{/\Span_{\FCopC[\OG^{\simeq}],\text{all}}(\FinG)} \rightleftarrows \Cat_{/\Span(\FCopC)} \colon \sigma^{*}.
    \end{equation*}
    This adjunction, by the description of cocartesian edges over the backwards morphisms from Construction \ref{ConsOpLAdj}, restricts to an adjunction 
    \begin{equation*}
        \sigma\circ- \colon \Cat_{/\Span_{\FCopC[\OG^{\simeq}],\text{all}}(\FinG)}^{\FCopC[\OG^{\simeq}]^{\op}\text{-cc}} \rightleftarrows \Cat_{/\Span(\FCopC)}^{\FCopC^{\op}\text{-cc}} \colon \sigma^{*}.
    \end{equation*}
    The right adjoint $\sigma^{*}$ is $\Cat$-linear by the same argument as for $\nu^{*}$, and it follows that $\sigma\circ-$ is an oplax $\Cat$-linear functor, see \cite[Theorem 3.4.7]{LaxMonTwoVar}. We conclude by observing that the above oplax $\Cat$-linear structure is actually $\Cat$-linear since, for any functor $\Cc\xrightarrow{}\Span_{\FCopC[\OG^{\simeq}],\text{all}}(\FinG)$ and any category $\Ee$, the natural functor 
    \begin{equation*}
        (\Ee\times \Span(\FCopC))\times_{\Span(\FCopC)}\Cc \xrightarrow{} (\Ee\times\Span_{\FCopC[\OG^{\simeq}],\text{all}}(\FinG))\times_{\Span_{\FCopC[\OG^{\simeq}],\text{all}}(\FinG)}\Cc,
    \end{equation*}
    is an equivalence. Therefore, the functor $\OpLAdj$ upgrades to a 2-functor, finishing the proof.
\end{proof}

\begin{notation}
    For an indexing system $\Ii\subseteq \FinG$, we denote the full subcategory of $\Ii$ spanned by transitive $G$-sets by $\Oo_{G,\Ii}$.
\end{notation}

\begin{prop}\label{NAlgVsAlg}
    Let $(\Cc,\otimes)$ be a symmetric monoidal category, and let $\Ii$ be an indexing system. There exists a natural equivalence of categories
    \begin{equation*}
        \NAlg_{\Ii}(\CorCofMul(\Cc,\otimes)))\simeq \Fun(\Oo_{G,\Ii},\CAlg(\Cc,\otimes)).
    \end{equation*}
\end{prop}

\begin{proof}
    First, recall that, by definition,
    \begin{equation*}
        \NAlg_{\Ii}(\CorCofMul(\Cc,\otimes)))\coloneq \FunLaxG(\Span_{\text{all},\Ii}(\FinG),\Uncc(\CorCofMul(\Cc,\otimes)))).
    \end{equation*}
    It follows from Proposition \ref{OpLAdjIsLAdj}, that $\OpLAdj$ admits a 2-right adjoint, which we denote by 
    \begin{equation*}
        \widetilde{R}\colon \POp \xrightarrow{} \POpG,
    \end{equation*}
    for the duration of this proof. Passing to the right adjoints in the commutative square from Proposition \ref{EnvLambdalEnv}, we obtain a commutative square
    \[
        \begin{tikzcd}[column sep =huge]
            \PSMCatG \arrow[d,"{\Uncc}"'] &\PSMCat \arrow[d,"{\Uncc}"'] \arrow[l,"{\CorCofMul-)}"']
            \\
            \POpG &\POp \arrow[l,"{\widetilde{R}}"'].
        \end{tikzcd}
    \]
    Combining all of the above, we get a chain of equivalences
    \begin{align*}
        \NAlg_{\Ii}(\CorCofMul(\Cc,\otimes))) &\coloneq \FunLaxG(\Span_{\text{all},\Ii}(\FinG),\Uncc(\CorCofMul(\Cc,\otimes))))
        \\
        &\simeq \FunLaxG(\Span_{\text{all},\Ii}(\FinG),\widetilde{R}\Uncc(\Cc,\otimes))
        \\
        &\simeq \FunLax(\OpLAdj(\Span_{\text{all},\Ii}(\FinG)),\Uncc(\Cc,\otimes))
        \\
        &\simeq \FunLax(\Span_{\FCopC[\OG^{\simeq}],\Ii}(\FinG),\Uncc(\Cc,\otimes)).
    \end{align*}
    To finish the proof, we consider the functor
    \begin{equation*}
        \CAlg(-)\coloneq \FunLax(\Span(\FCopC),-)\colon \Op \xrightarrow{} \Cat, 
    \end{equation*}
    by combining Proposition 2.26, Definition 2.28, and Theorem 2.57 in \cite{NormsEq}, this functor admits a 2-left adjoint. The 2-left adjoint of $\CAlg(-)$ is given on a category $\Aa$ by
    \begin{equation*}
        \Span_{\text{ct},\text{all}}(\Unct(\hat{\Aa}))\xrightarrow{} \Span(\FCopC),
    \end{equation*}
    where the functor $\hat{\Aa}\colon \FCopC^{\op}\xrightarrow{}\Cat$ is the unique finite product preserving functor sending the point to $\Aa$, and the backwards morphisms in the category $\Span_{\text{ct},\text{all}}(\Unct(\hat{\Aa}))$ are given by cartesian edges. One observes that the total space of $\Unct(\Aa)$ is equivalent to the finite coproduct completion $\FCopC[\Aa]$ of $\Aa$. For $\Aa=\Oo_{G,\Ii}$, we get equivalences
    \begin{equation*}
        \Span_{\text{ct},\text{all}}(\Unct(\hat{\Oo}_{G,\Ii}))\simeq \Span_{\FCopC[\Oo_{G,\Ii}^{\simeq}],\text{all}}(\Ii)\simeq \Span_{\FCopC[\OG^{\simeq}],\Ii}(\FinG),
    \end{equation*}
    where the first equivalence follows from $\Ii$ being a finite coproduct completion of $\Oo_{G,\Ii}$. It immediately follows that
    \begin{align*}
        \NAlg_{\Ii}(\CorCofMul(\Cc,\otimes))) &\simeq \FunLax(\Span_{\FCopC[\OG^{\simeq}],\Ii}(\FinG),(\Cc,\otimes))
        \\
        &\simeq \Fun(\Oo_{G,\Ii},\CAlg(\Cc,\otimes)),
    \end{align*}
    finishing the proof.
\end{proof}

We are ready to prove the main result of this section.

\begin{theo}\label{MainTh}
    Let $\Ii$ be an indexing system. There exists an equivalence of categories
    \begin{equation*}
        \NAlg_{\Ii}(\PSp_{G,\RatQ}^{\otimes}) \simeq \Fun(\Oo_{G,\Ii},\CAlg(\Sp_{\RatQ},\otimes)).
    \end{equation*}
\end{theo}

\begin{proof}
    Theorem \ref{SMEquivalence}, states that there is an equivalence $\PSp_{G,\RatQ}^{\otimes}\simeq \CorCofMul(\Sp_{\RatQ},\otimes))$. This combined with Proposition \ref{NAlgVsAlg}, shows the claim.
\end{proof}

In particular, the above theorem lets us recover an equivalence from Theorem 4.16 in \cite{Wimmer}. We let $\Sp_{G}^{O}$ denote the (1,1)-category of orthogonal $G$-spectra, see \cite[Definition II.2.6]{MandellMay}, or \cite[Definition 3.1.3, and Remark 3.1.8]{GlobalSchwede}. The (1,1)-category $\Sp_{G}^{O}$ admits a positive stable model category structure, see \cite[Theorem III.5.3]{MandellMay}. We refer to the weak equivalences in this model structure as $\pi_{*}$-isomorphisms. This model category structure, by Theorem III.8.1 in \cite{MandellMay}, lifts to a model structure on the (1,1)-category of strictly commutative monoids in $\Sp_{G}^{O}$.\!\footnote{We note that there is a mistake in \cite{MandellMay}, but the arguments still go through with a modification, see Remark 2.6 in \cite{Wimmer} for more details.} We denote the (1,1)-category of strictly commutative monoids by $\CMon(\Sp_{G}^{O})$, and refer to the weak equivalences from the lifted model category structure again by $\pi_{*}$-isomorphisms. Combining Theorem 7.27 with Remark A.6 in \cite{NormsEq}, we obtain an equivalence
\begin{equation*}
    \NAlg_{\FinG}(\PSp_{G}^{\otimes})\simeq\CMon(\Sp_{G}^{O})[\pi_{*}\text{-isomorphisms}^{\inv}].
\end{equation*}
There is a forgetful functor
\begin{equation*}
    \CMon(\Sp_{G}^{O})[\pi_{*}\text{-isomorphisms}^{\inv}]\xrightarrow{}\Sp_{G}^{O}[\pi_{*}\text{-isomorphisms}^{\inv}]\simeq \Sp_{G},
\end{equation*}
we lift rational equivalences to the category $\CMon(\Sp_{G}^{O})[\pi_{*}\text{-isomorphisms}^{\inv}]$, and denote the localization at the lifted rational equivalences by $(\CMon(\Sp_{G}^{O})[\pi_{*}\text{-isomorphisms}^{\inv}])_{\RatQ}$.

\begin{lem}\label{RatNAlgVsNAlgRat}
    Let $\Ii$ be an indexing system. There exists an equivalence
    \begin{equation*}
        \NAlg_{\Ii}(\PSp_{G,\RatQ}^{\otimes})\simeq \NAlg_{\Ii}(\PSp_{G}^{\otimes})_{\RatQ},
    \end{equation*}
    over $\Sp_{G,\RatQ}$, where the right hand side denotes the localization on those morphisms that forget to rational equivalences in $\Sp_{G}$.
\end{lem}

\begin{proof}
    We consider the $G$-adjunction
    \begin{equation*}
        (L_{\RatQ})_{*}\colon\PMack_{G}(\Sp) \rightleftarrows \PMack_{G}(\Sp_{\RatQ})\colon\mkern-3mu \incl_{*},
    \end{equation*}
    where the left adjoint admits a lift
    \begin{equation*}
        \underline{L}_{\RatQ}^{\otimes}\colon \PSp_{G}^{\otimes} \xrightarrow{} \PSp_{G,\RatQ}^{\otimes},
    \end{equation*}
    to $G$-symmetric monoidal $G$-categories, see discussion after Definition \ref{GSMLocFunc}. A standard argument shows that there is a relative adjunction
    \begin{equation*}
        \Uncc(\underline{L}_{\RatQ}^{\otimes})\colon\Uncc(\PSp_{G}^{\otimes}) \rightleftarrows \Uncc(\PSp_{G,\RatQ}^{\otimes})\colon\mkern-3mu \incl
    \end{equation*}
    over $\Span(\FinG)$, with a fully faithful right adjoint, which preserves cocartesian edges over backwards morphisms. By 2-functoriality of $\FunLaxG(-,-)$ we obtain an adjunction
    \begin{equation*}
        \FunLaxG(\Span_{\text{all},\Ii}(\FinG),\Uncc(\PSp_{G}^{\otimes})) \rightleftarrows \FunLaxG(\Span_{\text{all},\Ii}(\FinG),\Uncc(\PSp_{G,\RatQ}^{\otimes})),
    \end{equation*}
    with a fully faithful right adjoint. We observe that there is a commutative square
    \[
        \begin{tikzcd}
            \FunLaxG(\Span_{\text{all},\Ii}(\FinG),\Uncc(\PSp_{G}^{\otimes})) \arrow[d,"{\ev_{G/G}}"'] \arrow[r] &\FunLaxG(\Span_{\text{all},\Ii}(\FinG),\Uncc(\PSp_{G,\RatQ}^{\otimes})) \arrow[d,"{\ev_{G/G}}"']
            \\
            \Sp_{G} \arrow[r] &\Sp_{G,\RatQ},
        \end{tikzcd}
    \]
    where the left and right vertical morphisms are conservative. Therefore, the functor $\NAlg_{\Ii}(\PSp_{G}^{\otimes})\xrightarrow{} \NAlg_{\Ii}(\PSp_{G,\RatQ}^{\otimes})$ is a localization functor, and by the above commutative square, combined with the forgetful functors being conservative, we deduce that the localization is at the morphisms which forget to rational equivalences. 
\end{proof}

We are now ready to record the following implication of the previous theorem, which recovers Theorem 4.16 from \cite{Wimmer}.

\begin{cor}
    There exists an equivalence of categories
    \begin{equation*}
        (\CMon(\Sp_{G}^{O})[\pi_{*}\text{-isomorphisms}^{\inv}])_{\RatQ}\simeq \Fun(\OG,(\CMon(\Sp^{O})[\pi_{*}\text{-isomorphisms}^{\inv}])_{\RatQ}).
    \end{equation*}
\end{cor}

\begin{proof}
    By combining Theorem \ref{MainTh}, and Lemma \ref{RatNAlgVsNAlgRat}, one computes
    \begin{align*}
        (\CMon(\Sp_{G}^{O})[\pi_{*}\text{-isomorphisms}^{\inv}])_{\RatQ} &\simeq \NAlg_{\FinG}(\PSp_{G}^{\otimes})_{\RatQ}
        \\
        &\simeq \NAlg_{\FinG}(\PSp_{G,\RatQ}^{\otimes})
        \\
        &\simeq \Fun(\Oo_{G},\CAlg(\Sp_{\RatQ},\otimes))
        \\
        &\simeq \Fun(\Oo_{G},\CAlg(\Sp,\otimes)_{\RatQ})
        \\
        &\simeq \Fun(\OG,(\CMon(\Sp^{O})[\pi_{*}\text{-isomorphisms}^{\inv}])_{\RatQ}),
    \end{align*}
    finishing the proof.
\end{proof}

\begin{rem}
    All results in this paper remain valid if one replaces $\RatQ$ with any subring $R$ of rational numbers in which the order of $G$ is invertible. In particular, one replaces $\Sphere_{\RatQ}$ with 
    \begin{equation*}
        \Sphere_{R}\coloneq \colim(\Sphere\xrightarrow{r_{1}}\Sphere\xrightarrow{r_{1}r_{2}}\Sphere\xrightarrow{r_{1}r_{2}r_{3}}...),
    \end{equation*}
    where the sequence $(r_{i})_{i\in\mathbb{N}}$ enumerates primes that are invertible in $R$.
\end{rem}

\bibliographystyle{amsalpha}
\bibliography{Ref.bib}

@book{HTT,
  title={Higher Topos Theory},
  author={Lurie, Jacob},
  isbn={9780691140483},
  lccn={2008038170},
  series={Academic Search Complete},
  url={https://books.google.nl/books?id=OT6uGouiuzgC},
  year={2009},
  publisher={Princeton University Press}
}

@unpublished{HA,
    author = {Lurie, Jacob},
    title = {Higher algebra},
    year = {2017},
    note = {Available at \protect\url{https://www.math.ias.edu/~lurie/papers/HA.pdf}}
}

@article{BarwickSpectral,
  title={Spectral Mackey functors and equivariant algebraic {K}-theory (I)},
  author={Barwick, Clark},
  journal={Advances in Mathematics},
  volume={304},
  pages={646--727},
  year={2017},
  publisher={Elsevier}
}

@inproceedings{TwoVariable,
  title={Two-variable fibrations, factorisation systems and-categories of spans},
  author={Haugseng, Rune and Hebestreit, Fabian and Linskens, Sil and Nuiten, Joost},
  booktitle={Forum of Mathematics, Sigma},
  volume={11},
  pages={e111},
  year={2023},
  organization={Cambridge University Press}
}

@article{StratCat,
  title={Stratified categories, geometric fixed points and a generalized Arone-Ching theorem},
  author={Glasman, Saul},
  journal={arXiv preprint arXiv:1507.01976},
  year={2015}
}

@article{Expose,
  title={Parametrized higher category theory and higher algebra: Expos{\'e} {I}--Elements of parametrized higher category theory},
  author={Barwick, Clark and Dotto, Emanuele and Glasman, Saul and Nardin, Denis and Shah, Jay},
  journal={arXiv preprint arXiv:1608.03657},
  year={2016}
}

@article{ParStab,
  title={Parametrized stability and the universal property of global spectra},
  author={Cnossen, Bastiaan and Lenz, Tobias and Linskens, Sil},
  journal={Journal of Topology},
  volume={18},
  number={4},
  pages={e70044},
  year={2025},
  publisher={Wiley Online Library}
}

@article{NormsEq,
  title={Norms in equivariant homotopy theory},
  author={Lenz, Tobias and Linskens, Sil and P{\"u}tzst{\"u}ck, Phil},
  journal={arXiv preprint arXiv:2503.02839},
  year={2025}
}

@article{martini2021yoneda,
  title={Yoneda's lemma for internal higher categories},
  author={Martini, Louis},
  journal={arXiv preprint arXiv:2103.17141},
  year={2021}
}

@article{martini2021colimits,
  title={Colimits and cocompletions in internal higher category theory},
  author={Martini, Louis and Wolf, Sebastian},
  journal={arXiv preprint arXiv:2111.14495},
  year={2021}
}

@article{6-fun,
  title={6-functor formalisms and smooth representations},
  author={Heyer, Claudius and Mann, Lucas},
  journal={arXiv preprint arXiv:2410.13038},
  year={2024}
}

@article{EnvAlgPat,
  title={Envelopes for algebraic patterns},
  author={Barkan, Shaul and Haugseng, Rune and Steinebrunner, Jan},
  journal={arXiv preprint arXiv:2208.07183},
  year={2022}
}

@article{Partial,
  title={Partial parametrized presentability and the universal property of equivariant spectra},
  author={Cnossen, Bastiaan and Lenz, Tobias and Linskens, Sil},
  journal={Transactions of the American Mathematical Society},
  volume={378},
  number={10},
  pages={6913--6974},
  year={2025}
}

@article{Descent,
  title={Descent and vanishing in chromatic algebraic {K}-theory via group actions},
  author={Clausen, Dustin and Mathew, Akhil and Naumann, Niko and Noel, Justin},
  journal={arXiv preprint arXiv:2011.08233},
  year={2020}
}

@article{GuiMay,
  title={Models of {G}--spectra as presheaves of spectra},
  author={Guillou, Bertrand J and May, J Peter},
  journal={Algebraic \& Geometric Topology},
  volume={24},
  number={3},
  pages={1225--1275},
  year={2024},
  publisher={Mathematical Sciences Publishers}
}

@unpublished{BookProjStable,
    author = {Cnossen, Bastiaan},
    title = {Stable homotopy theory and higher algebra},
    note = {Ongoing book project. It can be found on the authors website \url{https://sites.google.com/view/bastiaan-cnossen/home}}
}

@article{MultInfLoopSpc,
  title={Universality of multiplicative infinite loop space machines},
  author={Gepner, David and Groth, Moritz and Nikolaus, Thomas},
  journal={Algebraic \& Geometric Topology},
  volume={15},
  number={6},
  pages={3107--3153},
  year={2016},
  publisher={Mathematical Sciences Publishers}
}

@article{HighTamb,
  title={Normed equivariant ring spectra and higher {T}ambara functors},
  author={Cnossen, Bastiaan and Haugseng, Rune and Lenz, Tobias and Linskens, Sil},
  journal={arXiv preprint arXiv:2407.08399},
  year={2024}
}

@article{NormsMotivic,
  title={Norms in motivic homotopy theory},
  author={Bachmann, Tom and Hoyois, Marc},
  journal={arXiv preprint arXiv:1711.03061},
  year={2017}
}

@article{HomotCommutRingBisp,
  title={Homotopical commutative rings and bispans},
  author={Cnossen, Bastiaan and Haugseng, Rune and Lenz, Tobias and Linskens, Sil},
  journal={Journal of the London Mathematical Society},
  volume={111},
  number={6},
  pages={e70200},
  year={2025},
  publisher={Wiley Online Library}
}

@article{BlumbergHill,
  title={Operadic multiplications in equivariant spectra, norms, and transfers},
  author={Blumberg, Andrew J and Hill, Michael A},
  journal={Advances in Mathematics},
  volume={285},
  pages={658--708},
  year={2015},
  publisher={Elsevier}
}

@article{BlumbergHillMainDef,
  title={Incomplete {T}ambara functors},
  author={Blumberg, Andrew J and Hill, Michael A},
  journal={Algebraic \& Geometric Topology},
  volume={18},
  number={2},
  pages={723--766},
  year={2018},
  publisher={Mathematical Sciences Publishers}
}

@article{LaxMonTwoVar,
  title={Lax monoidal adjunctions, two-variable fibrations and the calculus of mates},
  author={Haugseng, Rune and Hebestreit, Fabian and Linskens, Sil and Nuiten, Joost},
  journal={Proceedings of the London Mathematical Society},
  volume={127},
  number={4},
  pages={889--957},
  year={2023},
  publisher={Wiley Online Library}
}

@article{Wimmer,
  title={A model for genuine equivariant commutative ring spectra away from the group order},
  author={Wimmer, Christian},
  journal={arXiv preprint arXiv:1905.12420},
  year={2019}
}

@book{MandellMay,
  title={Equivariant orthogonal spectra and {$S$}-modules},
  author={Mandell, Michael A and May, J Peter},
  year={2002},
  publisher={American Mathematical Soc.}
}

@book{GlobalSchwede,
  title={Global homotopy theory},
  author={Schwede, Stefan},
  volume={34},
  year={2018},
  publisher={Cambridge University Press}
}

@article{Serre,
  title={Homologie Singuliere Des Espaces Fibres},
  author={Jean-Pierre Serre},
  journal={Annals of Mathematics},
  year={1951},
  volume={54},
  pages={425}
}

@book{lewisMaySteinberger,
  title={Equivariant Stable Homotopy Theory},
  author={Lewis, L.G.J. and McClure, J.E. and May, J.P. and Steinberger, M.},
  isbn={9783540168201},
  lccn={86025968},
  series={Lecture Notes in Mathematics},
  year={1986},
  publisher={Springer Berlin Heidelberg}
}

@article{GreenleesConjecture,
  author  = {J. P. C. Greenlees},
  title   = {Triangulated categories of rational equivariant cohomology theories},
  journal = {Oberwolfach Reports},
  volume  = {8},
  year    = {2006},
  pages   = {480--488}
}

@book{GreenleesMay,
  author    = {J. P. C. Greenlees and J. P. May},
  title     = {Generalized Tate Cohomology},
  publisher = {American Mathematical Society},
  series    = {Memoirs of the American Mathematical Society},
  volume    = {113},
  year      = {1995},
  pages     = {viii+178}
}

@book{GreenleesConj1,
  author    = {J. P. C. Greenlees},
  title     = {Rational {$S^1$}-Equivariant Stable Homotopy Theory},
  publisher = {American Mathematical Society},
  series    = {Memoirs of the American Mathematical Society},
  volume    = {138},
  year      = {1999},
  pages     = {xii+289}
}

@incollection{GreenleesConj2,
  author    = {J. P. C. Greenlees},
  title     = {Rational {O}(2)-Equivariant Cohomology Theories},
  booktitle    = {Fields Institute Communications},
  volume    = {19},
  publisher = {American Mathematical Society},
  year      = {1998},
  pages     = {103--110}
}

@incollection{GreenleesConj3,
  author    = {J. P. C. Greenlees},
  title     = {Rational {SO}(3)-Equivariant Cohomology Theories},
  booktitle    = {Contemporary Mathematics},
  volume    = {271},
  publisher = {American Mathematical Society},
  year      = {2001},
  pages     = {99--125}
}

@article{ShipleyCDGA,
  title={{$\textup{H}\mathbb{Z}$}-algebra spectra are differential graded algebras},
  author={Shipley, Brooke},
  journal={American journal of mathematics},
  volume={129},
  number={2},
  pages={351--379},
  year={2007},
  publisher={Johns Hopkins University Press}
}

@article{HHR,
  title={On the nonexistence of elements of Kervaire invariant one},
  author={Hill, Michael A and Hopkins, Michael J and Ravenel, Douglas C},
  journal={Annals of Mathematics},
  pages={1--262},
  year={2016},
  publisher={JSTOR}
}
\end{document}